\documentclass{amsart}
\usepackage[T1]{fontenc}
\usepackage{amsmath}
\usepackage{amssymb}
\usepackage{amsfonts}
\usepackage{graphicx}
\usepackage{enumerate}
\usepackage{color}

\newtheorem{theorem}{Theorem}[section]
\newtheorem{lemma}[theorem]{Lemma}

\theoremstyle{definition}
\newtheorem{remark}[theorem]{Remark}
\newtheorem{example}[theorem]{Example}
\newtheorem{definition}[theorem]{Definition}

\newcommand\F{\mathcal{F}}
\newcommand\E{\mathbb{E}}
\newcommand\R{\mathbb{R}}
\newcommand\e{\varepsilon}

\numberwithin{equation}{section}

\begin{document}

\title[Good-$\lambda$ inequalities]{Noncommutative good-$\lambda$ inequalities}

\author{Yong Jiao}
\address{School of Mathematics and Statistics, Central South University, Changsha 410085,
People's Republic of China}
\email{jiaoyong@csu.edu.cn}

\author{Adam Os\k ekowski}
\address{Faculty of Mathematics, Informatics and Mechanics\\
 University of Warsaw\\
Banacha 2, 02-097 Warsaw\\
Poland}
\email{ados@mimuw.edu.pl}

\author{Lian Wu}
\address{School of Mathematics and Statistics, Central South University, Changsha 410085,
People's Republic of China}
\email{wulian@csu.edu.cn}

\thanks{Yong Jiao is supported by the NSFC  (No.11471337, No.11722114). Adam Os\k ekowski is supported by Narodowe Centrum Nauki (Poland), grant DEC-2014/14/E/ST1/00532. Lian Wu is supported by the NSFC (No.11601526)}


\begin{abstract}
We propose a novel approach in noncommutative probability, which can be regarded as an analogue of good-$\lambda$ inequalities from the classical case due to Burkholder and Gundy (Acta Math {\bf124}: 249-304,1970). This resolves a longstanding open problem in noncommutative realm. Using this technique, we present new proofs of noncommutative Burkholder-Gundy inequalities, Stein's inequality, Doob's inequality and $L^p$-bounds for martingale transforms; all the constants obtained are of optimal orders. The approach also allows us to investigate the noncommutative analogues of decoupling techniques and, in particular, to obtain new estimates for noncommutative martingales with tangent difference sequences and sums of tangent positive operators. These in turn yield an enhanced version of Doob's maximal inequality for adapted sequences and a sharp estimate for a certain class of Schur multipliers. We also present fully new applications of good-$\lambda$ approach to noncommutative harmonic analysis, including inequalities for differentially subordinate operators motivated by the classical $L^p$-bound for the Hilbert transform and the estimate for the $j$-th Riesz transform on group von Neumann algebras with constants of optimal orders as $p\to\infty.$
\end{abstract}

\maketitle

\section{Introduction}
Good-$\lambda$ inequalities form a powerful tool used in the commutative probability theory and harmonic analysis to establish $L^p$- and $\Phi$-inequalities for various classes of processes and operators. The idea can be formulated as follows: given $0<p<\infty$, in order to prove the moment inequality
\begin{equation}\label{moment}
 ||Y||_{L^p}\leq c_p ||X||_{L^p}
\end{equation}
between two random variables $X$ and $Y$ on some probability space $(\Omega,\F,\mathbb{P})$,
it is enough to find positive parameters $\alpha$, $\beta$, $\delta$ satisfying $\alpha\beta^p<1$ such that
\begin{equation}\label{gl}
 \mathbb{P}\Big(|Y|\geq \beta \lambda,\,|X|\leq \delta \lambda\Big)\leq \alpha \mathbb{P}\big(|Y|\geq \lambda\big)
\end{equation}
for each $\lambda>0$. Then a straightforward integration argument (see \eqref{integrating} and \eqref{integrating1} below) yields \eqref{moment} with $c_p=\delta^{-1}(\beta^{-p}-\alpha)^{-1/p}$. A similar reasoning shows that if $\Phi:[0,\infty)\to [0,\infty)$ is a function satisfying appropriate growth conditions, then the good-$\lambda$ inequality \eqref{gl} implies
$$ \E \big(\Phi(|Y|)\big)\leq C_{\Phi,\alpha,\beta,\delta} \E \big(\Phi(|X|)\big),$$
with some constant $C_{\Phi,\alpha,\beta,\delta}$ depending only on the parameters indicated.

The origin of the approach goes back to the classical works of Burkholder and Gundy \cite{BG}. Though the estimate of the form \eqref{gl} cannot be found there, several related tail bounds proved in that paper can be regarded as predecessors of good-$\lambda$ inequalities.  Probably the first paper where the estimate \eqref{gl} appears explicitly is that of Burkholder \cite{Bu0}. In particular that work contains the proofs, based on the above argument, of Burkholder-Davis-Gundy inequalities, Stein-type bounds  and conditional square function estimates, both for discrete-time martingales and the continuous-time analogues arising in the context of Brownian motion.

The above approach has turned out to be very efficient. Furthermore, very soon after the appearance of \cite{Bu0}, the method of good-$\lambda$ inequalities was applied successfully outside probability theory. For instance, Burkholder \cite{Bu0.5,BG2} used the approach in the study of Hardy spaces associated with harmonic functions on the halfspace $\R^d_+$, while in \cite{Bu1} the method allowed him to study the range of analytic functions  on the unit disc. Muckenhoupt and Wheeden \cite{MW} used good-$\lambda$ inequalities to obtain weighted inequalities for fractional operators, and Coifman and Fefferman \cite{CF} exploited the technique to show estimates for singular integral operators and maximal functions. See also the more recent works of Buckley \cite{Buc} on maximal operators, Aimar et al. \cite{AFM} on the estimates for one-sided singular integrals, Bagby and Kurtz \cite{BK} for a rearrangement-inequality for singular integral operators as well as the works of Ba\~nuelos \cite{Ba} and Ba\~nuelos and Moore \cite{BM} for the study of tight estimates for Riesz transforms and caloric functions. Very recently, Hofmann et al. \cite{Hofmann} used a good-$\lambda$ inequality for the vertical square function to establish square function/non-tangential maximal function estimates for solutions of the homogeneous equation associated with divergence form elliptic operators.

It was also soon realized that the method of good-$\lambda$ inequalities, if applied appropriately (that is, if the parameters $\alpha$, $\beta$, $\delta$ are chosen in a clever manner), leads to best-order constants in many situations. This in particular allowed Hitczenko \cite{Hi} to prove that the $L^p$ constant in the martingale version of Rosenthal's inequality has the optimal order $O(p/\log p)$ as $p\to \infty$, and also enabled him to study $L^p$ estimates for tangent and conditionally independent seqeunces with constants not depending on $p$ in \cite{Hi2}. This ``efficiency phenomenon'' has also been observed in most of the analytic papers mentioned above.

The motivation for the results obtained in this paper comes from  a very natural question concerning the appropriate version of good-$\lambda$ inequalities in the context of noncommutative (or quantum) probability theory. More specifically, we will study this question in the language of noncommutative martingales. This branch of martingale theory has gained a lot of interest in literature in the recent twenty years. Many fundamental inequalities  have been successfully transferred from the classical to the noncommutative setting, often revealing quite surprising facts concerning the shape of the estimates and the sizes of the constants involved (\cite{JX2}). Let us briefly mention here several papers which are fundamental to the area. The work \cite{PX} of Pisier and Xu can be regarded as a starting point of the whole theory: it contains the introduction of the abstract noncommutative setup used in the later works, as well as the formulation of appropriate Burkholder-Gundy and Stein's inequalities. A few years later, Doob's maximal estimate and maximal ergodic theorem were respectively generalized to the noncommutative setting by Junge \cite{J} and Junge and Xu \cite{JX-M}; the appropriate analogues of Burkholder-Rosenthal inequalities were investigated by Junge and Xu in \cite{JX, JX3}. Much effort was put into the understanding of the structure of noncommutative martingales. In particular, the noncommutative analogue of Gundy's decomposition of a martingale was obtained by Parcet and Randrianantoanina in \cite{PR} and a version of Davis' decomposition was found by Perrin in \cite{Pe}; these have been greatly improved in very recent papers \cite{RW,RWX}.
We also refer the reader to the important works on the weak-type versions of the estimates above, given by Randrianantoanina \cite{R1,R2,R3}, certain noncommutative atomic decompositions \cite{BCPY} and its recent improvement together with a John-Nirenberg inequality by Hong and Mei \cite{HM}, and some recent advances regarding algebra atomic decompositions and asymmetric Doob's inequalities by Junge et al \cite{HJP2,HJP1}. Finally, we mention the works \cite{BC12,BCL,BCO,Ji,RW17} for martingale inequalities in the context of various noncommutative symmetric spaces, the articles \cite{JSXZ, JSZ,JSZZ} for the noncommutative analogs of Johnson-Schechtman inequalities, and the very recent paper \cite{JZWZ} for the duality of noncommutative dyadic martingale Hardy space.

The noncommutative extension of good-$\lambda$ inequalities is a longstanding open problem circulating in the noncommutative realm for more than fifteen years; the authors learned it from Quanhua Xu about ten years ago. The following is quoted from \cite[page 181]{BC12}: ``On the other hand, the noncommutative analogue of good-$\lambda$ inequality seems open. Then, in order to prove the noncommutative $\Phi$-moment inequalities we need new ideas''. Similar statement appeared in \cite[page 1577]{RW17}: ``The original proof was primarily based on careful analysis of distribution functions using stopping times and the so-called good-$\lambda$ inequality which are very powerful techniques in
the classical settings. Unfortunately, these techniques are not available in the noncommutative
settings.''
The main contribution of this paper is the extension of the good-$\lambda$ approach to the above noncommutative case. As usual, the first difficulty is how to invent the appropriate \emph{formulation/shape} of the noncommutative version of \eqref{gl}. We shall see that the passage from the commutative to the noncommutative realm enforces certain unexpected ideas. Furthermore, one should expect that in contrast to the classical case, where it is usually quite easy to  verify directly that two specific random variables satisfy the good-$\lambda$ inequality, it might be considerably harder in the noncommutative case to check that two measurable operators are eligible for the method. In other words, the second difficulty we encounter concerns the formulation of proper and universal conditions on the operators which guarantee the validity of the noncommutative good-$\lambda$ inequalities. Of course, such conditions should be verifiable in a rather easy and convenient way. We resolve this issue by proposing a certain set of requirements, which we call \emph{good-$\lambda$ testing conditions}. At the first glance, these requirements might seem complicated and of artificial shape. However, they are applicable in all the relevant settings and their verification is straightforward; furthermore, we offer a substitute, called \emph{strong good-$\lambda$ testing conditions}, which is much simpler and thus easier to be checked, at the cost of being slightly less general.

The paper is organized as follows. The next section contains some basic facts from operator theory which are necessary for our further investigation. Section~3 is devoted to the abstract formulation of noncommutative good-$\lambda$ inequalities. By performing the careful analysis of Burkholder-Gundy estimates, we present the (informal) reasoning which leads us to an appropriate formulation of the method. Then we verify rigorously that the technique is indeed efficient in the noncommutative realm. Section 4 contains applications to fundamental results in the noncommutative martingale theory, obtained earlier by Junge, Pisier, Randrianantoanina and Xu. Namely, as we shall see there, the good-$\lambda$ method offers a new, simpler and unified approach to Burkholder-Gundy, Stein and Doob's inequalities, as well as Burkholder's estimates for martingale transforms. In all the settings, we obtain the bounds with constants of optimal orders. In  Section 5 we investigate $L^p$-inequalities for noncommutative martingales with tangent martingale differences, and sums of tangent positive operators. This area, to the best of our knowledge, has not been studied in literature and we strongly believe that it has far reaching further connections with noncommutative probability and analysis. It is worth saying here that the passage from the classical to the noncommutative $L^p$-estimates for tangent martingales reveals an unexpected phenomenon (which should be compared to a similar behavior of  noncommutative Burkholder-Gundy and Burkholder-Rosenthal inequalities): these estimates hold true in the range $2\leq p<\infty$ only.
We present further interesting applications of the estimates for noncommutative tangent sequences; specifically, we will establish an enhanced version of noncommutative Doob's inequality for adapted sequences and provide  a sharp bound for a certain novel class of Schur multipliers.

We conclude the paper by presenting, in Section 6, several completely new applications of good-$\lambda$ approach in noncommutative harmonic analysis. It allows us to develop the notion of the differential subordination associated with a contractive semigroup on a semifinite von Neumann algebra which is strongly motivated by the classical results on the boundedness of Hilbert transform; we also investigate $L^p$-estimate of the $j$-th Riesz transform on group von Neumann algebras. Again, the constants are of optimal order as $p\to\infty$ in both cases. Our final application is to study square-function estimates for contractive semigroups on von Neumann algebras which improves the orders of constants in \cite[Theorem 2.4.10]{JM} to be linear.
We strongly believe that the method of good-$\lambda$ inequalities developed in this work has many further applications and connections to noncommutative harmonic analysis and noncommutative potential theory.

\section{Preliminaries}
In this section, we briefly introduce the necessary background and notation needed for the study of noncommutative martingale inequalities. The reader interested in the detailed exposition of the subject is referred to the monographs \cite{KR1,KR2,T}. Throughout the paper, the symbol $\mathcal{M}$ denotes a von Neumann algebra and we equip this object with a semifinite normal faithful trace $\tau$. We treat $\mathcal{M}$ as a subalgebra of the larger algebra of all bounded operators acting on some given Hilbert space $H$. A closed, densely defined operator $a$ on $H$ is \emph{affiliated} with $\mathcal{M}$ if for all unitary operators $u$ belonging to the commutant $\mathcal{M}'$ of $\mathcal{M}$ we have the identity $u^*au=a$. Such an operator $a$ is said to be \emph{$\tau$-measurable} if for any $\e>0$ there exists a projection $e$ such that $e(H)\subset D(a)$ and
$\tau(I-e)<\e$. Here and below, we use the symbol $I$ to denote the identity operator. The class of all $\tau$-measurable operators will be denoted by $L^0(\mathcal{M},\tau)$. It can be shown that the trace $\tau$ extends to a positive tracial functional on the positive part $L^0_+(\mathcal{M},\tau)$ of $L^0(\mathcal{M},\tau)$ (with no risk of confusion, this extension is still denoted by $\tau$). 
If  $a$ is a self-adjoint $\tau$-measurable operator,  let $a=\int_{-\infty}^\infty \lambda de_\lambda$ stand for its spectral decomposition. For any Borel subset $B$ of $\R$, the spectral
projection of $a$ corresponding to the set $B$ is defined by $I_B(a)=\int_{-\infty}^\infty \chi_B(\lambda)de_\lambda$.

Let $e$, $f$ be two projections belonging to $\mathcal M$. Then $e\vee f$ (resp., $e\wedge f$) stands for the projection onto the sum $e(H)\cup f(H)$ (resp., onto the intersection $e(H)\cap f(H)$). The projections $e$ and $f$ are said to be equivalent if there exists a partial isometry $u\in\mathcal M$ such that $u^*u=e$ and $uu^*=f$. In this case, we denote $e\sim f.$

For $0< p<\infty$, we recall that the noncommutative $L^p$-space associated  with $(\mathcal M,\tau)$ is defined by $L^p(\mathcal{M},\tau)=\{x\in L^0(\mathcal{M},\tau):\tau(|x|^p)<\infty\}$
equipped with the (quasi-)norm
$ \|x\|_p=(\tau(|x|^p))^{1/p}$, where $|x|=(x^*x)^{1/2}$ is the modulus of $x$. For $p=\infty$, the space $L^p(\mathcal{M},\tau)$ coincides with $\mathcal{M}$ with its usual operator norm. We refer to the survey \cite{px2003} and the references therein for more details. At some places below we will need to work with two or more von Neumann algebras at the same time. For the convenience of the reader and to avoid confusion, in such a case we will indicate the algebra with respect to which the $L^p$-norm is calculated (writing $\|x\|_{L^p(\mathcal{M})}$ instead of $\|x\|_p$, etc.).

Let us present some basic facts from the theory of noncommutative martingales. Suppose that $(\mathcal{M}_n)_{n\geq 0}$ is a filtration, i.e., a nondecreasing sequence of von Neumann subalgebras of $\mathcal{M}$ whose union is weak$^*$-dense in $\mathcal{M}$. Then for any $n\geq 0$ there exists a normal conditional expectation $\mathcal{E}_n$ from $\mathcal{M}$ onto $\mathcal{M}_n$, which satisfies the requirements
\begin{itemize}
\item[(i)] $\mathcal{E}_n(axb)=a\mathcal{E}_n(x)b$ for all $a,\,b\in\mathcal{M}_n$ and $x\in \mathcal{M}$;
\item[(ii)] $\tau\circ \mathcal{E}_n=\tau$.
\end{itemize}
It is then easy to verify that we  have $\mathcal{E}_m\mathcal{E}_n=\mathcal{E}_n\mathcal{E}_m=\mathcal{E}_{\min(m,n)}$ for all nonnegative integers $m$ and $n$. Furthermore, since $\mathcal{E}_n$ preserves the trace, it can be extended to a contractive projection from $L^p(\mathcal{M},\tau)$ onto $L^p(\mathcal{M}_n,\tau_n)$ for all $1\leq p\leq \infty$ (where $\tau_n$ is the restriction of $\tau$ to $\mathcal{M}_n$). We will sometimes work with two different filtered von Neumann algebras $(\mathcal{M},(\mathcal{M}_n)_{n\geq 0},\tau)$, $(\mathcal{N},(\mathcal{N}_n)_{n\geq 0},\nu)$, and then, to avoid confusion, we will denote the associated sequences of conditional expectations by $(\mathcal{E}^\mathcal{M}_n)_{n\geq 0}$ and $(\mathcal{E}^\mathcal{N}_n)_{n\geq 0}$.

A sequence $x=(x_n)_{n\geq 0}$ in $L^1(\mathcal{M})+\mathcal M$ is called a \emph{noncommutative martingale} (with respect, or adapted to $(\mathcal{M}_n)_{n\geq 0}$), if for any $n\geq 0$ we have the equality
$$ \mathcal{E}_n(x_{n+1})=x_n.$$
The associated difference sequence is given by the formulae $dx_0=x_0$ and $dx_n=x_n-x_{n-1}$ for $n\geq 1$. Furthermore, we define the associated square function $S(x)$ and conditioned square function $s(x)$ by
$$ S(x)=\left(\sum_{n=0}^\infty |dx_n|^2\right)^{1/2}\qquad \mbox{and}\qquad s(x)=\left(\sum_{n=0}^\infty \mathcal{E}_{n-1}(|dx_n|^2)\right)^{1/2}.$$
Sometimes we will also use the truncated versions of these objects, given by
$$ S_N(x)=\left(\sum_{n=0}^N |dx_n|^2\right)^{1/2}\qquad \mbox{and}\qquad s_N(x)=\left(\sum_{n=0}^N \mathcal{E}_{n-1}(|dx_n|^2)\right)^{1/2}$$
for any nonnegative integer $N$.

In literature, the adjoint square functions $S(x^*)$ and $s(x^*)$ also play a significant role. However, we should emphasize here that essentially all the operators and martingales we will study below will be assumed to be self-adjoint; our methods enable the successful treatment of such operators only. Fortunately, in most cases this does not affect the generality of results, as more or less standard decomposition arguments typically allow the reduction of a given inequality under investigation to its special version for self-adjoint objects.

\section{Noncommutative good-$\lambda$ inequalities}
The purpose of this section is to present an abstract formulation of good-$\lambda$ inequalities in the noncommutative setting, which in Section \ref{applications} will be applied to obtain proofs of various important estimates. For the sake of clarity of the exhibition, we have decided to split the argumentation into several intermediate steps. Subsection \ref{prel} is a little informal and contains the explanation of the reasoning which has led us to the appropriate form of noncommutative good-$\lambda$ inequalities. The rigorous formulation and the study of good-$\lambda$ inequalities are presented in Subsections \ref{St4} and \ref{St5}, which are fundamental to the whole paper. In the last subsection, we provide the proof of general moment inequalities via the good-$\lambda$ inequalities.

\subsection{On the search of a suitable good-$\lambda$ inequality}\label{prel}
Our construction rests on a careful investigation of Burkholder-Gundy inequality: for a given parameter $p$ and some finite constant $C_p$ depending only on $p$,
\begin{equation}\label{BDG}
 \big\|S_N(x)\big\|_{L^p(\mathcal{M})}\leq C_p\big\|x_N\big\|_{L^p(\mathcal{M})},
\end{equation}
where $x=(x_n)_{n= 0}^N$ is an arbitrary finite martingale in $L^p(\mathcal{M},\tau)$ and $S_N(x)$ is its truncated square function. We will frequently switch from the classical to the noncommutative version of this estimate and back, which, hopefully, should not lead to any confusion.
One of the reasons why we decided to model our approach on this particular inequality is that the optimal orders of the constant $C_p$ as $p\to \infty$ are different in the classical and the noncommutative situations (see \cite{JX2}); furthermore, in the classical setting the estimate above is true in the range $1<p<\infty$, while for noncommutative martingales, it holds for $p\geq 2$ only (for $1<p<2$, one has to formulate the inequality in a different manner). Thus, it seems plausible to expect that the estimate \eqref{BDG} should indicate the necessary modifications of good-$\lambda$ inequalities which need to be implemented in the noncommutative context.

\medskip

\noindent\emph{Step 1.} To gain some intuition about our approach, let us start with the commutative case. As we have already seen in the introductory section, a classical method (see \cite{Bu0}) would rest on exploiting the estimate of the form
\begin{equation}\label{good1}
 \mathbb{P}\Big(S_N(x)\geq \beta \lambda,\,x_N^*\leq \delta \lambda\Big)\leq \alpha\mathbb{P}\big(S_N(x)\geq \lambda\big),
\end{equation}
where $x_N^*=\sup_{0\leq n\leq N}|x_n|$ is the maximal function of $x$.
Here $\lambda$ ranges from $0$ to infinity, while $\alpha$, $\beta$ and $\delta$ are appropriately chosen positive parameters. Such an estimate, if true, implies
$$ \mathbb{P}\big(S_N(x)\geq \beta \lambda\big)\leq \mathbb{P}\big(x_N^*\geq \delta \lambda\big)+\alpha\mathbb{P}\big(S_N(x)\geq \lambda\big).$$
Multiplying throughout by $\lambda^{p-1}$ and integrating over $\lambda$ from $0$ to $\infty$ yields an estimate equivalent to
\begin{equation}\label{integrating}
(\beta^{-p}-\alpha)\E \big(S_N(x)^p\big)\leq \delta^{-p}\E \big((x_N^*)^p\big).
\end{equation}
Hence, if the parameters $\alpha$, $\beta$ and $\delta$  satisfy $\beta^{-p}>\alpha$, then we get the bound
\begin{equation}\label{integrating1}
\big\|S_N(x)\big\|_{L^p}\leq \frac{\delta^{-1}}{(\beta^{-p}-\alpha)^{1/p}}\big\|x_N^*\big\|_{L^p},\qquad 1\leq p<\infty,
\end{equation}
which in turn gives the desired BG estimate $||S_N(x)||_{L^p}\lesssim_p ||x_N||_{L^p}$ for $1<p<\infty$, by virtue of Doob's maximal inequality; some further optimization over the parameters $\alpha$, $\beta$ and $\delta$ can be carried over, to ensure the optimal order of the constant: $O(p^{1/2})$ as $p\to \infty$.

 Our first observation is that the inequality \eqref{good1} is \emph{not} a good starting point in the noncommutative situation. The fundamental obstacle is that if any version of it held true, then, performing an analogous argument as above (which involves summation rather than integration, as we shall see later), we must obtain the bound
$$ \big\|S_N(x)\big\|_{L^p(\mathcal{M})}\leq c_p\big\|x_N\big\|_{L^p(\mathcal{M})}\;\; \mbox{for }1<p<\infty.$$
However, this estimate fails to hold for $1<p<2$ no matter what $c_p$ is (as we have already said, noncommutative Burkholder-Gundy inequalities are formulated differently in this range). This indicates that instead of \eqref{good1}, one should search for another classical good-$\lambda$-type inequality which is more suitable for noncommutative extensions. Motivated by the above calculation, we can impose the following (a little informal) requirement. Namely, such a good-$\lambda$-type estimate must contain in its formulation some sort of a threshold $p_0$ indicating that it yields $L^p$-estimates  in the range $p_0<p<\infty$ only. Then in the classical case such a threshold would have to be set to be $1$, while in the noncommutative situation one would be forced to take $p_0=2$.

\medskip

\noindent\emph{Step 2.} Such an alternative classical good-$\lambda$-type inequality is also contained in \cite{Bu0}. In Lemma 3.1 there, Burkholder established (a slight extension of) the following bound:
\begin{equation}\label{good2.0}
 \lambda\mathbb{P}\Big(S_N(x)\geq \beta\lambda\Big)\leq 3\theta^{-1}\E \Big(|x_N| 1_{\{S_N(x)\geq \lambda\}}\Big),
\end{equation}
where $\lambda$ and $\theta$ are arbitrary positive numbers and $\beta=(1+2\theta^2)^{1/2}$. Multiplying both sides by $\lambda^{p-2}$ and integrating over $\lambda$ from $0$ to infinity, one gets
$$ \E\Big( S_N(x)^p\Big)\leq \frac{3\beta^p}{\theta}\frac{p}{p-1}\E \Big(|x_N|S_N(x)^{p-1}\Big),\qquad 1<p<\infty,$$
which, by the H\"older inequality, implies
$$ \big\|S_N(x)\big\|_{L^p}\leq \frac{3\beta^p}{\theta}\frac{p}{p-1}\big\|x_N\big\|_{L^p},\qquad 1<p<\infty.$$
Setting $\theta=p^{-1/2}$, we see that $\beta^p=(1+2/p)^{p/2}<e<3$, which gives the Burkholder-Gundy estimate with the constant of optimal order $O(p^{1/2})$ as $p\to \infty$. Obviously, \eqref{good2.0} has the same deficiency as previously: any noncommutative version of it would yield a false inequality for $1<p<2$. However, now it is clear how to modify the estimate: the threshold $p_0=1$ will increase to $2$ if we square the appropriate terms on the left and on the right:
\begin{equation}\label{good2}
 \lambda^2\mathbb{P}\Big(S_N(x)\geq \beta\lambda\Big)\leq 3\theta^{-1}\E\Big( x_N^2 1_{\left\{S_N(x)\geq \lambda\right\}}\Big),\qquad \lambda>0,
\end{equation}
for some positive parameters $\beta$, $\theta$ to be specified.
Indeed, the repetition of the above  argument now yields Burkholder-Gundy inequality \eqref{BDG} in the range $2<p<\infty$ only. Thus, it seems promising to consider \eqref{good2} as the right starting point for the noncommutative good-$\lambda$ inequality.

\medskip

\noindent\emph{Step 3.} As we have already seen above, the size of the constants $C_p$ in \eqref{BDG} depend only on the values of the parameters involved in the good-$\lambda$ inequality.  Our next step is to search for a proof of the classical estimate \eqref{good2} which would be easily transferable to the noncommutative realm: this will give us some additional hints on the shape of noncommutative good-$\lambda$ estimates. It turns out that such a proof naturally splits into two parts: first one establishes a slightly stronger version of \eqref{good2} and then deduces the desired bound by Chebyshev's inequality.

\smallskip

\noindent\emph{$\bullet$ An auxiliary bound}. Consider the stopping time
$$ \mu=\inf\Big\{n: S_n(x)\geq 1\Big\},$$
with the standard convention that $\inf\emptyset=\infty$. Then for any $n$, we have $S_{n-1}(x)<1$ on the set $\{\mu=n\}$ (we set $S_{-1}(x)=0$) and hence
\begin{equation}\label{Step1}
\E \Big(\left(S_n^2(x)-1\right)1_{\{\mu=n\}}\Big)\leq \E \Big(\left(S_{n}^2(x)-S_{n-1}^2(x)\right)1_{\{\mu=n\}}\Big)=\E \Big(dx_n^21_{\{\mu=n\}}\Big).
\end{equation}
Furthermore, for any $N>n$, using the fact that $x$ is a martingale, we have
\begin{equation}
\begin{split}
\E \Big(\left(S_{N}^2(x)-S_{n}^2(x)\right)1_{\{\mu=n\}}\Big)&\leq \sum_{k=n+1}^N \E \Big(dx_k^21_{\{\mu=n\}}\Big)\\
&=\E\Big((x_N-x_{n})^21_{\{\mu=n\}}\Big)\\
&=\E\Big(x_N^21_{\{\mu=n\}}\Big)-\E \Big(x_{n}^21_{\{\mu=n\}}\Big)\\
&\leq \E\Big(x_N^21_{\{\mu=n\}}\Big).
\end{split}
\end{equation}
Adding the above two simple observations, we get the estimate
$$ \E \Big(\big(S_N^2(x)-1\big)1_{\{\mu=n\}}\Big)\leq \E \Big(\big(x_N^2+dx_n^2\big)1_{\{\mu=n\}}\Big)\leq \E \Big(\big(x_N^2+(dx_N^*)^2\big)1_{\{\mu=n\}}\Big),$$
where $dx_N^*=\sup_{0\leq n\leq N}|dx_n|$.
Hence, summing over $n$, we finally obtain
\begin{equation}\label{good3}
 \E \Big(\big(S_N^2(x)-1\big)1_{\{\mu\leq N\}}\Big)\leq \E \Big(\big(x_N^2+(dx_N^*)^2\big)1_{\{\mu\leq N\}}\Big).
\end{equation}
%
%
This is precisely the auxiliary estimate. We turn to the second part of the proof.
\smallskip

\noindent\emph{$\bullet$ An application of Chebyshev's inequality.} Obviously, the random variable $\big(S_N^2(x)-1\big)1_{\{\mu\leq N\}}$ is positive (the reason for which we formulate this trivial observation is that the noncommutative counterpart of this statement will not be true in general). Consequently, Chebyshev's inequality yields, for any $\beta>1$,
\begin{equation}\label{good4}
\begin{split}
 \mathbb{P}\Big(S_N(x)\geq \beta\Big)&=\mathbb{P}\Big(S_N^2(x)-1\geq \beta^2-1\Big)\\
&\leq \frac{1}{\beta^2-1}\E\Big(\big(x_N^2+(dx_N^*)^2\big)1_{\{\mu\leq N\}}\Big)\\
&\leq \frac{1}{\beta^2-1}\E \Big(\big(x_N^2+(dx_N^*)^2\big)1_{\{S_N(x)\geq 1\}}\Big).
\end{split}
\end{equation}
This is a form of the estimate \eqref{good2} we would like to transfer to the noncommutative case: it could then be regarded as a noncommutative good-$\lambda$ bound corresponding to the Burkholder-Gundy inequalities.

\medskip
However, before we do this, let us check what constant we obtain with the use of this inequality.
Applying the bound to the martingale $x/\lambda$, multiplying both sides by $\lambda^{p-1}$ and integrating over $\lambda$ gives
$$ \E \Big(S_N^p(x)\Big)\leq \frac{p}{p-2}\frac{\beta^p}{\beta^2-1}\E \Big(S_N^{p-2}(x)\big(x_N^2+(dx_N^*)^2\big)\Big)$$
and hence
$$ \big\|S_N(x)\big\|_{L^p}^2\leq \frac{p}{p-2}\frac{\beta^p}{\beta^2-1}\big\|x_N^2+(dx_N^*)^2\big\|_{L^{p/2}}.$$
Since $dx_N^*\leq 2x_N^*$, triangle inequality and Doob's maximal estimate finally yield
$$\big\|S_N(x)\big\|_{L^p}\leq \left\{\frac{p}{p-2}\frac{\beta^p}{\beta^2-1}\left(1+4\left(\frac{p}{p-1}\right)^2\right)\right\}^{1/2}\big\|x_N\big\|_{L^p}.$$
Setting $\beta=1+1/p$, we obtain the constant of order $O(p^{1/2})$ as $p\to \infty$. This is bad news: it is well-known (see \cite{JX2}) that in the noncommutative setting the optimal order is $O(p)$. This proves that still some modification of \eqref{good3} (and hence also \eqref{good4}) is needed. An indication in the right direction is already contained in the above discussion. In the noncommutative situation there will be no reason for the (appropriate version of the) term $\big(S_N^2(x)-1\big)1_{\{\mu\leq N\}}$ to be positive. A little thought and experimentation suggests considering the following variant of \eqref{good3}:
\begin{equation}\label{good5}
 \E \Big(\big(S_N(x)-1\big)^21_{\{\mu\leq N\}}\Big) \leq \E \Big(\big(x_N^2+(dx_N^*)^2\big)1_{\{\mu\leq N\}}\Big).
\end{equation}
As we have already noted, in the classical case we have $S_N(x)\geq 1$ on $\{\mu\leq N\}$, and hence this new bound is weaker than \eqref{good3}. Applying Chebyshev's inequality gives, for any $\beta>1$,
\begin{equation}\label{good6}
\mathbb{P}\Big(S_N(x)\geq \beta\Big)\leq \frac{1}{(\beta-1)^2}\E \Big(\big(x_N^2+(dx_N^*)^2\big)1_{\{S_N(x)\geq 1\}}\Big).
\end{equation}
Repeating the above calculations shows that \eqref{good6} implies Burkholder-Gundy inequality
$$ \big\|S_N(x)\big\|_{L^p}\leq \left\{\frac{p}{p-2}\frac{\beta^p}{(\beta-1)^2}\left(1+4\left(\frac{p}{p-1}\right)^2\right)\right\}^{1/2}\big\|x_N\big\|_{L^p},$$
for which the optimal choice $\beta=p/(p-2)$ returns the constant of order $O(p)$. This indicates that \eqref{good5} and \eqref{good6} should indeed be the right noncommutative versions of good-$\lambda$ inequalities.

\subsection{Noncommutative version of \eqref{good5}}\label{St4}
Now we leave the context of Burkholder-Gundy inequality and, motivated by the above considerations, formulate the appropriate general version of the inequality \eqref{good5} which will be applicable in the study of various bounds of the form
$$ ||y||_{L^p(\mathcal{M})}\leq c_p||x||_{L^p(\mathcal{M})}.$$
Until the end of this section, we assume that $N$ is a fixed nonnegative integer, $y=(y_n)_{n= 0}^N$ is a finite, self-adjoint martingale (with respect to some filtration), while $x_N$ and $z_N$ are self-adjoint operators. 
Consider the sequence $R=(R_n)_{n\geq -1}$ of projections associated with $y$, given by $R_{-1}=I$ and, inductively,
$$ R_n=R_{n-1}I_{(-\infty,1)}(R_{n-1}y_nR_{n-1}),\qquad n=0,\,1,\,2,\,\ldots,\, N.$$
Note that in the classical case the projection $I-R_N$ corresponds to the indicator function of the set $\{\max_{0\leq m\leq N}y_m\geq 1\}$ and thus it is closely related to the term $1_{\{\mu\leq N\}}=1_{\{S_N\geq 1\}}$ in \eqref{good5}. Some elementary properties of $R=(R_n)_{n\geq -1}$ are enumerated below (see \cite{Cu}).
\begin{lemma}\label{prop}
The following statements  hold true:
\begin{enumerate}[{\rm (i)}]

\item for each $n\geq 0$, the projections $R_n$ belongs to $\mathcal{M}_n$;

\item for each $n\geq 0$, the projection $R_n$ commute with $R_{n-1}y_nR_{n-1}$;

\item for each $n\geq 0$, we have $ R_ny_nR_n\leq R_n.$
\end{enumerate}
\end{lemma}

The following assumption will play a key role in this paper. It concerns the structure of the operators which enable the effective functioning of the good-$\lambda$ approach.

\begin{definition}\label{test condition}
Let $x_N$, $y$, $z_N$ and $(R_n)_{n\geq -1}$ be as above. The triple $\big(x_N, y, z_N\big)$ is said to satisfy the \emph{good-$\lambda$ testing conditions} if we have

\begin{enumerate}[{\rm (i)}]
\item $\sum_{n=0}^{N} \sum_{k=n+1}^N \tau\big((R_{n-1}-R_{n})dy_kR_{n-1}dy_k(R_{n-1}-R_{n}) \big)\leq  \tau\big((I-R_N)x_N^2\big);$
\item for each $0\leq k\leq N$ and any projection $P\in \mathcal{M}_k$,
$\tau\left(Pdy_k^2P\right)\leq \tau\left(Pz_N^2P\right).$
\end{enumerate}
\end{definition}
Though these assumptions  might look complicated and artificial, we will see in later sections that they are satisfied in all the relevant settings. For instance, if $x=(x_n)_{n=0}^N$ and $y=(y_n)_{n=0}^N$ are martingales such that $dy_k^2\leq dx_k^2$ for all $0\leq k\leq N$, then (i) holds; if $z_N^2$ is a majorant of the sequence $dy^2$ (i.e., we have $z_N^2\geq dy_k^2$ for all $0\leq k\leq N$), then (ii) is valid. However, there are other settings in which both the conditions are satisfied.

{There is a set of slightly stronger requirements which has the advantage of being much more concise.
\begin{definition}\label{strong test condition def}
Let $x_N$, $y$, $z_N$ be as above. The triple $\big(x_N, y, z_N\big)$ is said to satisfy the \emph{strong good-$\lambda$ testing conditions} if we have, for every $k\geq0,$
\begin{equation}\label{strong test condition}
\sum_{m=k+1}^N\mathcal{E}_k(dy_m^2)\leq \mathcal{E}_k(x_N^2)\quad {\rm and}\quad dy_k^2\leq \mathcal{E}_k(z_N^2).
\end{equation}
\end{definition}
It is obvious that the two inequalities in \eqref{strong test condition} imply the conditions (i) and (ii) in Definition \ref{test condition}. Moreover, the strong good-$\lambda$ testing conditions are easier to be checked in practice, as they refer solely to $x$, $y$, $z$ and do not involve the sequence $(R_n)_{n\geq -1}$.} However, in all our applications of the good-$\lambda$ approach below, we have decided to verify the original good-$\lambda$ testing conditions because of its slight generality.

We will establish the following inequality. It is evident, at least optically, that this estimate can be regarded as a noncommutative analogue of \eqref{good5}: see the above interpretation of $I-R_N$.

\begin{theorem}\label{thm0}
Suppose that the triple $\big(x_N, y, z_N\big)$ satisfies the good-$\lambda$ testing conditions.
Then we have
\begin{equation}\label{good-la}
\tau\left((I-R_N)\left(y_N-I\right)^2\right)\leq 2\tau\Big((I-R_N)\left(x_N^2+z_N^2\right)\Big).
\end{equation}
\end{theorem}
\begin{proof}
We will first prove that
\begin{equation}\label{thmgood-la-eq1}\tau\Big((I-R_N)\left(y_N-I\right)^2\Big)\leq 2\sum_{n=0}^N \tau\Big((R_{n-1}-R_n)(y_N-I)R_{n-1}(y_N-I)\Big).
\end{equation}
To this end, note that
\begin{align*}
\sum_{n=0}^N \tau\Big((R_{n-1}-R_n)y_NR_{n-1}y_N\Big)&=\sum_{n=0}^N\sum_{k=n}^N \tau\Big((R_{n-1}-R_n)y_N(R_{k-1}-R_k)y_N\Big)\\
&\quad +\sum_{n=0}^N \tau\Big((R_{n-1}-R_n)y_NR_Ny_N\Big)\\
&\geq \sum_{k=0}^N\sum_{n=0}^k \tau\Big((R_{n-1}-R_n)y_N(R_{k-1}-R_k)y_N\Big)\\
&=\sum_{k=0}^N \tau\Big((I-R_k)y_N(R_{k-1}-R_k)y_N\Big),
\end{align*}
which immediately yields
\begin{equation*}
\begin{split}
2&\sum_{n=0}^N \tau\Big((R_{n-1}-R_n)y_NR_{n-1}y_N\Big)\\
&\geq \sum_{n=0}^N \tau\Big((R_{n-1}-R_n)y_NR_{n-1}y_N\Big) +\sum_{k=0}^N \tau\Big((I-R_k)y_N(R_{k-1}-R_k)y_N\Big)\\
&=\sum_{n=0}^N \tau\Big((I-R_n+R_{n-1})y_N(R_{n-1}-R_n)y_N\Big)\\
&=\tau\Big((I-R_N)y_N^2\Big)+\sum_{n=0}^N \tau\Big((R_{n-1}-R_n)y_N(R_{n-1}-R_n)y_N\Big).
\end{split}
\end{equation*}
Therefore
\begin{align*}
\tau&\Big((I-R_N)\left(y_N-I\right)^2\Big) +\sum_{n=0}^N \tau\Big((R_{n-1}-R_n)(y_N-I)(R_{n-1}-R_n)(y_N-I)\Big)\\
&= \tau\left((I-R_N)y_N^2\right) +\sum_{n=0}^N \tau\Big((R_{n-1}-R_n)y_N(R_{n-1}-R_n)y_N\Big)\\
&\qquad -4\tau\Big((I-R_N)y_N\Big)+2\tau(I-R_N)\\
&\leq 2\sum_{n=0}^N \tau\Big((R_{n-1}-R_n)y_NR_{n-1}y_N\Big)-4\tau\Big((I-R_N)y_N\Big)+2\tau(I-R_N)\\
&=2\sum_{n=0}^N \tau\Big((R_{n-1}-R_n)(y_N-I)R_{n-1}(y_N-I)\Big),
\end{align*}
which implies \eqref{thmgood-la-eq1}. Using the fact that $y$ is a martingale, the right hand side of \eqref{thmgood-la-eq1} can be written as
\begin{align*}
 2\sum_{n=0}^N\tau&\Big((R_{n-1}-R_n)(y_N-I)R_{n-1}(y_N-I)\Big)\\
& = 2\sum_{n=0}^N\tau\Big((R_{n-1}-R_n)(y_n-I)R_{n-1}(y_n-I)(R_{n-1}-R_n)\Big)\\
&\quad +2 \sum_{n=0}^N\sum_{k=n+1}^N\tau\Big((R_{n-1}-R_n)dy_kR_{n-1}dy_k(R_{n-1}-R_n)\Big).
\end{align*}
Exploiting the first assumption of the good-$\lambda$ testing conditions, we obtain
$$\sum_{n=0}^N\sum_{k=n+1}^N\tau\Big((R_{n-1}-R_n)dy_kR_{n-1}dy_k(R_{n-1}-R_n)\Big)\leq \tau\Big((I-R_N)x_N^2\Big).$$
According to the second assumption of the good-$\lambda$ testing conditions, the proof of the theorem will be complete if the following can be verified: for any $0\leq n\leq N$,
\begin{equation}\label{innterm}
\begin{split}
 &\tau\Big((R_{n-1}-R_n)dy_n^2\Big)\geq \tau\Big((R_{n-1}-R_n)(y_n-I)R_{n-1}(y_n-I)\Big).
\end{split}
\end{equation}
Indeed, observe that
\begin{align*}
\tau\Big((R_{n-1}-R_n)dy_n^2\Big)&\geq \tau\Big((R_{n-1}-R_n)dy_n(R_{n-1}-R_n)dy_n\Big)\\
&=\tau\Big((R_{n-1}-R_n)(y_n-y_{n-1})(R_{n-1}-R_n)(y_n-y_{n-1})\Big)\\
&\geq \tau\Big((R_{n-1}-R_n)(y_n-I)(R_{n-1}-R_n)(y_n-I)\Big).
\end{align*}
To see that the last passage is valid, we transform it into the equivalent estimate
\begin{equation}\label{Equiva}
\tau\Big((R_{n-1}-R_n)(I-y_{n-1})(R_{n-1}-R_n)(2y_n-y_{n-1}-I)(R_{n-1}-R_n)\Big)\geq 0.
\end{equation}
According to the definition of $R$, we know that
$$(R_{n-1}-R_n)(I-y_{n-1})(R_{n-1}-R_n)\geq 0$$
 and
\begin{align*}
&(R_{n-1}-R_n)(2y_n-y_{n-1}-I)(R_{n-1}-R_n)\\
&=2(R_{n-1}-R_n)(y_n-I)(R_{n-1}-R_n)+(R_{n-1}-R_n)(I-y_{n-1})(R_{n-1}-R_n)\geq 0.
\end{align*}
These imply that \eqref{Equiva} holds.
Observe that by the commuting property of $R$ (Lemma \ref{prop} (ii)), we have
$$ \tau\Big((R_{n-1}-R_n)(y_n-I)(R_{n-1}-R_n)(y_n-I)\Big)=\tau\Big((R_{n-1}-R_n)(y_n-I)R_{n-1}(y_n-I)\Big)$$
and hence \eqref{innterm} follows.
\end{proof}

\subsection{Noncommutative version of \eqref{good6}}\label{St5}
As in the case of \eqref{good-la}, we continue with the general setup of an arbitrary martingale $y=(y_n)_{n=0}^N$ and operators $x_N$, $z_N$ satisfying the domination principles (i) and (ii) in Definition 3.2. We need additional projections, which will capture the behavior of the tails $\left\{S(x)\geq \beta\right\}$ in \eqref{good6}. For a fixed number $\beta>1$, consider the family $(Q_n)_{n= 0}^N$ given by $Q_{-1}=I$ and, inductively,
\begin{equation}\label{defQ}
Q_{n}=Q_{n-1}I_{(-\infty,\beta)}(Q_{n-1}y_nQ_{n-1}),\qquad n=0,\,1,\,2,\,\ldots,\,N.
\end{equation}

The version of \eqref{good6}  can now be stated as follows.

\begin{theorem}\label{thm2}
Suppose that the triple $\big(x_N, y, z_N\big)$ satisfies the good-$\lambda$ testing conditions. Then we have
\begin{equation}\label{good-lambda}
\tau(I-Q_N)\leq 4(\beta-1)^{-2}\tau\Big((I-R_N)\big(x_N^2+z_N^2\big)\Big).
\end{equation}
\end{theorem}
\begin{proof}
As we have already seen above, in the classical case the assertion follows at once from Chebyshev's inequality. In the noncommutative setting, however, there are several technical issues which make the reasoning quite lengthy. We have decided to split the proof into a few intermediate parts.

\medskip

\noindent\emph{Step 1.} Fix $n\in \{0,\,1,\,2,\,\ldots,\,N\}$ and observe that by the very definition of $Q_n$,
\begin{align*}
&\tau(Q_{n-1}-Q_n)\\
&=\tau\Big(I_{[\beta,\infty)}\big((Q_{n-1}-Q_n)y_n(Q_{n-1}-Q_n)\big)\Big)\\
&=\tau\bigg(I_{[\beta,\infty)}\Big((Q_{n-1}-Q_n)\big(R_ny_nR_n+(I-R_n)y_nR_n+y_n(I-R_n)\big)(Q_{n-1}-Q_n)\Big)\bigg).
\end{align*}
By the properties of the projection $R_n$, the operator $d_n:=R_ny_nR_n+I-R_n$ is not bigger than $I$. Consequently,
\begin{align*}
&\tau(Q_{n-1}-Q_n)\\
&=\tau\bigg(I_{[\beta,\infty)}\Big((Q_{n-1}-Q_n)\big(d_n+(I-R_n)y_nR_n+(y_n-I)(I-R_n)\big)(Q_{n-1}-Q_n)\Big)\bigg)\\
&\leq \tau\bigg(I_{[\beta-1,\infty)}\Big((Q_{n-1}-Q_n)\big((I-R_n)y_nR_n+(y_n-I)(I-R_n)\big)(Q_{n-1}-Q_n)\Big)\bigg).
\end{align*}
Therefore, by Chebyshev's inequality,
\begin{equation}\label{Cheb}
\begin{split}
(&\beta-1)^2\tau(Q_{n-1}-Q_n)\\
& \leq \tau\bigg(\Big((Q_{n-1}-Q_n)\big((I-R_n)y_nR_n+(y_n-I)(I-R_n)\big)(Q_{n-1}-Q_n)\Big)^2\bigg)\\
&\leq \tau\bigg((Q_{n-1}-Q_n)\Big((I-R_n)y_nR_n+(y_n-I)(I-R_n)\Big)^2(Q_{n-1}-Q_n)\bigg)\\
&=\tau\Big((Q_{n-1}-Q_n)(y_n-I)(I-R_n)(y_n-I)(Q_{n-1}-Q_n)\Big)\\
&\quad +\tau\Big((Q_{n-1}-Q_n)(I-R_n)y_nR_ny_n(I-R_n)(Q_{n-1}-Q_n)\Big)\\
&:=I_1+I_2.
\end{split}
\end{equation}
We will analyze the terms $I_1$ and $I_2$ separately below.

\medskip

\noindent\emph{Step 2.} The analysis of the term $I_1$ is simple. Observe that  by the martingale property of $y$,
\begin{align*}
 \tau\Big(&(Q_{n-1}-Q_n)(y_N-I)(I-R_n)(y_N-I)(Q_{n-1}-Q_n)\Big)\\
&=\tau\Big((Q_{n-1}-Q_n)(y_n-I)(I-R_n)(y_n-I)(Q_{n-1}-Q_n)\Big)\\
&\quad +\sum_{k=n+1}^N \tau\Big((Q_{n-1}-Q_n)dy_k(I-R_n)dy_k(Q_{n-1}-Q_n)\Big)\\
&\geq \tau\Big((Q_{n-1}-Q_n)(y_n-I)(I-R_n)(y_n-I)(Q_{n-1}-Q_n)\Big).
\end{align*}
This implies
\begin{equation}\label{boundI1}
\begin{split}
I_1&\leq \tau\Big((Q_{n-1}-Q_n)(y_N-I)(I-R_n)(y_N-I)(Q_{n-1}-Q_n)\Big)\\
&\leq \tau\Big((Q_{n-1}-Q_n)(y_N-I)(I-R_N)(y_N-I)(Q_{n-1}-Q_n)\Big).
\end{split}
\end{equation}
The analysis of $I_2$ is more complicated. By the martingale property of $y$, we have, for any $k\geq n$,
\begin{align*}
&\tau\Big((Q_{n-1}-Q_n)(I-R_k)y_kR_ky_k(I-R_k)(Q_{n-1}-Q_n)\Big)\\
&\leq\tau\Big((Q_{n-1}-Q_n)(I-R_k)y_kR_ky_k(I-R_k)(Q_{n-1}-Q_n)\Big)\\
&\quad +\tau\Big((Q_{n-1}-Q_n)(I-R_k)dy_{k+1}R_kdy_{k+1}(I-R_k)(Q_{n-1}-Q_n)\Big)\\
&= \tau\Big((Q_{n-1}-Q_n)(I-R_k)y_{k+1}R_ky_{k+1}(I-R_k)(Q_{n-1}-Q_n)\Big)\\
&=\tau\Big((Q_{n-1}-Q_n)(I-R_k)y_{k+1}(R_k-R_{k+1})y_{k+1}(I-R_k)(Q_{n-1}-Q_n)\Big)\\
&\quad +\tau\Big((Q_{n-1}-Q_n)(I-R_k)y_{k+1}R_{k+1}y_{k+1}(I-R_k)(Q_{n-1}-Q_n)\Big).
\end{align*}
But we have $(R_k-R_{k+1})y_{k+1}R_{k+1}=R_{k+1}y_{k+1}(R_k-R_{k+1})=0$, by the commuting property of $R$ (see Lemma \ref{prop} (ii)). Plugging this above, we see that
\begin{equation*}\label{innnterm}
\begin{split}
&\tau\Big((Q_{n-1}-Q_n)(I-R_k)y_kR_ky_k(I-R_k)(Q_{n-1}-Q_n)\Big)\\
&\leq \tau\Big((Q_{n-1}-Q_n)(I-R_k)y_{k+1}(R_k-R_{k+1})y_{k+1}(I-R_k)(Q_{n-1}-Q_n)\Big)\\
&\quad +\tau\Big((Q_{n-1}-Q_n)(I-R_{k+1})y_{k+1}R_{k+1}y_{k+1}(I-R_{k+1})(Q_{n-1}-Q_n)\Big).
\end{split}
\end{equation*}
Therefore, by induction,
\begin{align*}
I_2&\leq \tau\Big((Q_{n-1}-Q_n)(I-R_N)y_NR_Ny_N(I-R_N)(Q_{n-1}-Q_n)\Big)\\
&\quad +\sum_{k=n}^{N-1}\tau\Big((Q_{n-1}-Q_n)(I-R_k)y_{k+1}(R_k-R_{k+1})y_{k+1}(I-R_k)(Q_{n-1}-Q_n)\Big).
\end{align*}
By the martingale property of $y$, we further get
\begin{align*}
I_2&\leq \tau\Big((Q_{n-1}-Q_n)(I-R_N)y_NR_Ny_N(I-R_N)(Q_{n-1}-Q_n)\Big)\\
&\quad +\sum_{k=n}^{N-1}\tau\Big((Q_{n-1}-Q_n)(I-R_k)y_{N}(R_k-R_{k+1})y_{N}(I-R_k)(Q_{n-1}-Q_n)\Big).
\end{align*}
Therefore, we have shown that
\begin{equation}\label{boundI2}
\begin{split}
I_2 &\leq \tau\Big((Q_{n-1}-Q_n)(I-R_N)(y_N-I)R_N(y_N-I)(I-R_N)(Q_{n-1}-Q_n)\Big)\\
&\quad +\sum_{k=0}^{N-1}\tau\Big((Q_{n-1}-Q_n)(I-R_k)(y_N-I)(R_k-R_{k+1})(y_N-I)(I-R_k)\Big)
\end{split}
\end{equation}
(by the tracial property, we removed one projection $Q_{n-1}-Q_n$ from the end of the last expression). This is the desired upper bound for $I_2$.

\medskip

\noindent\emph{Step 3.} Let us plug the estimates \eqref{boundI1} and \eqref{boundI2}  into \eqref{Cheb} and then sum over $n$. By the tracial property,
\begin{equation*}\label{sumboundI1}
\begin{split}
\sum_{n=0}^N&\tau\Big((Q_{n-1}-Q_n)(y_N-I)(I-R_N)(y_N-I)(Q_{n-1}-Q_n)\Big)\\
&=\tau\Big((I-Q_N)(y_N-I)(I-R_N)(y_N-I)\Big)\\
&\leq \tau\Big((y_N-I)(I-R_N)(y_N-I)\Big).
\end{split}
\end{equation*}
Analogously, we have
\begin{equation*}\label{sumboundI21}
\begin{split}
\sum_{n=0}^N &\bigg(\tau\Big((Q_{n-1}-Q_n)(I-R_N)(y_N-I)R_N(y_N-I)(I-R_N)(Q_{n-1}-Q_n)\Big)\bigg)\\
&\leq \tau\Big((I-R_N)(y_N-I)R_N(y_N-I)(I-R_N)\Big)
\end{split}
\end{equation*}
and
\begin{equation*}\label{sumboundI22}
\begin{split}
\sum_{n=0}^N\sum_{k=0}^{N-1}& \tau\bigg((Q_{n-1}-Q_n)(I-R_k)(y_N-I)(R_k-R_{k+1})(y_N-I)(I-R_k)\bigg)\\
& \leq \sum_{k=0}^{N-1} \tau\bigg((I-R_k)(y_N-I)(R_k-R_{k+1})(y_N-I)\bigg)\\
& \leq \sum_{k=0}^{N-1} \tau\bigg((I-R_N)(y_N-I)(R_k-R_{k+1})(y_N-I)\bigg)\\
& \leq \tau\Big((I-R_N)(y_N-I)(I-R_N)(y_N-I)\Big).
\end{split}
\end{equation*}
Plugging all these observations  into \eqref{Cheb}, we obtain
\begin{align*}
&(\beta-1)^2\tau(I-Q_N)\\
&=(\beta-1)^2\sum_{n=0}^N \tau(Q_{n-1}-Q_n)\\
&\leq \tau\bigg((y_N-I)(I-R_N)(y_N-I)+(I-R_N)(y_N-I)R_N(y_N-I)(I-R_N)\bigg)\\
&\quad +\tau\bigg((I-R_N)(y_N-I)(I-R_N)(y_N-I)\bigg)\\
&\leq 2\tau\Big((y_N-I)(I-R_N)(y_N-I)\Big).
\end{align*}
It suffices to apply \eqref{good-la} to obtain the desired assertion.
\end{proof}

\subsection{Proof of moment estimates via good-$\lambda$ inequalities} Equipped with the good-$\lambda$ inequality \eqref{good6}, we are ready for the proof of general  moment inequalities. As we have seen above, in the classical case the argument rests on a simple integration and application of H\"older's inequality. Here we proceed similarly, summing appropriately rescaled versions of \eqref{good-lambda},  but we also have to implement some necessary modifications to address the issues which arise in the noncommutative setting.

As previously, we assume that $y=(y_n)_{n=0}^N$ is a finite martingale and $x_N$, $z_N$ are given self-adjoint operators.
We now introduce a class of auxiliary objects.
For a fixed $\gamma>0$, let $(R_n^\gamma)_{n\geq 0}$ be the sequence as previously, built on the martingale $y/\gamma$: that is, we have $R_{-1}^\gamma=I$ and, for any $n\geq 0$,
$$ R_n^\gamma=R_{n-1}^\gamma I_{(-\infty,\gamma)}(R_{n-1}^\gamma y_n R_{n-1}^\gamma).$$
Next, we consider the following modification introduced by Randrianantoanina \cite{R1}. Namely, for a fixed $B>1$, $n\geq 0$ and $k\in\mathbb{Z}$, we set
\begin{equation}\label{defP}
P_n^{B^k}:=\bigwedge_{\ell\geq k} R_n^{B^\ell},
\end{equation}
(i.e., $P_n^{B^k}$ is the projection onto the intersection $\bigcap_{\ell\geq k}R_n^{B^\ell}(H)$).
The reason for the introduction of the family $P$ is to ensure the monotonicity property with respect to both $n$ and $k$. More precisely, note that for any fixed $k$, the projections $(R_n^{B^k})_{n\geq 0}$ are decreasing when $n$ increases; however, there is no monotonicity if we fix $n$ and change $k$. The new projections $(P_n^{B^k})_{n,k}$ have the monotonicity property with respect to both parameters: $P_n^{B^\ell}\leq P_m^{B^k}$ if $n\geq m$ and $\ell\leq k$. Note that in the commutative case we have $P_n^{B^k}=R_n^{B^k}$, and thus we may regard $P_n^{B^k}$ as the ``corrected'' noncommutative version of the indicator function of the set $\big\{\max_{0\leq m\leq n}y_m< B^k\big\}$. We will also use the auxiliary operator $a_N^+=a_N^+(y)$ given by
\begin{equation}\label{defaN}
 a_N^+=\sum_{k\in \mathbb{Z}} B^k(P_N^{B^{k+1}}-P_N^{B^k}).
\end{equation}
In the commutative case, we have $a_N^+=\sum_{k\in \mathbb{Z}} B^k 1_{\{B^{k}\leq  \max_{0\leq m\leq N}y_m< B^{k+1}\}}$ and hence $a$ can be regarded as a weak one-sided maximal operator of $y$. There is a symmetric version $a^-$ of $a^+$, given by $a_N^-=a_N^+(-y)$. The pair $(a^-,a^+)$ will control appropriately the martingale $y$, and hence it is enough to provide an efficient bound for these weak operators.

Here is one of the main results of the paper.

\begin{theorem}\label{main-theorem}
Let $2<p<\infty$ and $x_N, z_N$ belong to $L^{p}(\mathcal M)$. Suppose that for any $\mu>0$, the triple $(x_N/\mu , y/\mu, z_N/\mu)$ satisfies the good-$\lambda$ testing conditions.
Then $y_N$ belongs to $L^{p}(\mathcal M)$. Moreover, we have
\begin{equation}\label{main_moment}
 \|y_N\|_{L^p({\mathcal{M}})}\leq \frac{12 p}{\left(1-\left(1+\frac{1}{p}\right)^{2-p}\right)^{1/2}}\left(\|{x}_N\|_{L^p({\mathcal{M}})}^2+\|z_N\|^2_{L^{p}({\mathcal{M}})}\right)^{1/2}.
\end{equation}
\end{theorem}

\begin{remark}
In all the applications below, it will suffice to verify the testing condition for $\mu=1$; the case of general $\mu>0$ will follow at once by homogeneity argument.
\end{remark}

\begin{proof}[Proof of Theorem \ref{main-theorem}]
 It is easy to check that the good-$\lambda$ testing condition (ii) guarantees $y_N\in L^p(\mathcal M)$ under the assumption that $z_N\in L^p(\mathcal M)$.

 We now prove \eqref{main_moment}.  Let us apply the inequality \eqref{good-lambda} with $\beta=B$ to obtain
$$ \tau(I-R_N^B)\leq 4(B-1)^{-2}\tau\Big((I-R_N^1)({x}_N^2+z^2_N)\Big).$$
We translate this inequality into the language of the projections $P$. The right-hand side is easy to handle: we have $P_N^{1}\leq R_N^{1}$, so
$$ \tau\Big((I-R_N^{1})(x_N^2+z_N^2)\Big)\leq \tau\Big((I-P_N^1)(x_N^2+z_N^2)\Big).$$
To deal with the left-hand side, write
\begin{align*}
\tau\left(I-R_N^{B}\right)&=\tau\left(I-P_N^{B}\right)+\tau\left(P_N^{B}-R_N^{B}\right)\\
 &=\tau\left(I-P_N^{B}\right)+\tau\left(R_N^{B}\wedge P_N^{B^2}-R_N^{B}\right)\\
&\geq \tau\left(I-P_N^{B}\right)-\tau\left(I-P_N^{B^{2}}\right)\\
&=\tau\left(P_N^{B^2}-P_N^{B}\right).
\end{align*}
Combining the above observations, we get
$$ \tau\left(P_N^{B^2}-P_N^B\right)\leq 4(B-1)^{-2}\tau\Big((I-P_N^1)({x}_N^2+z_N^2)\Big),$$
which, by homogeneity (i.e., by replacing $x_N$, $y$, $z_N$  with ${x_N}/B^k$, $y/B^k$ and $z_N/B^k$), implies
\begin{equation}\label{good-hom}
 \tau\left(P_N^{B^{k+2}}-P_N^{B^{k+1}}\right)\leq 4B^{-2k}(B-1)^{-2}\tau\left(\left(I-P_N^{B^k}\right)\left({x}_N^2+z_N^2\right)\right).
\end{equation}

\def\rrttt{Multiply both sides by $\Phi(B^{2k+2})$ and sum over $k\in \mathbb{Z}$. Then the left-hand side of the obtained estimate is equal to $\tau(\Phi((a_N^+)^2))$, directly from \eqref{defaN}. To handle the right-hand side, note that
\begin{align*}
\sum_{k\in\mathbb{Z}} \Phi(B^{2k+2})B^{-2k}(I-P_N^{B^k})&=\sum_{\ell\in\mathbb{Z}}(P_N^{B^{\ell+1}}-P_N^{B^\ell})\sum_{k\leq \ell} \Phi(B^{2k+2})B^{-2k}.
\end{align*}
By \eqref{altalpha}, we have $(B^{2k+2})^{-\alpha_\Phi}\Phi(B^{2k+2})\leq (B^{2\ell+2})^{-\alpha_\Phi}\Phi(B^{2\ell+2})$ if $k\leq \ell$. Hence
\begin{align*}
\sum_{k\leq \ell} \frac{\Phi(B^{2k+2})}{B^{2k}}&\leq (B^{2\ell})^{-\alpha_\Phi}\Phi(B^{2\ell+2})\sum_{k\leq \ell} B^{2k(\alpha-1)}\\
&=\frac{B^{-2\ell}}{1-B^{2(1-\alpha_\Phi)}}\Phi(B^{2\ell+2})\\
&\leq \frac{B^{-2\ell+2\beta_\Phi}}{1-B^{2(1-\alpha_\Phi)}}\Phi(B^{2\ell})\\
&\leq \frac{B^{2\beta_\Phi}}{1-B^{2(1-\alpha_\Phi)}}\Phi'_+(B^{2\ell}),
\end{align*}
where in the third passage we have exploited \eqref{altbeta}. Therefore, coming back to the previous equality,
\begin{align*}
 \sum_{k\in\mathbb{Z}} \frac{\Phi(B^{2k+2})}{B^{2k}}(I-P_N^{B^k})&\leq \frac{B^{2\beta_\Phi}}{1-B^{2(1-\alpha_\Phi)}}\sum_{\ell\in\mathbb{Z}} \Phi'_+(B^{2\ell})(P_N^{B^{\ell+1}}-P_N^{B^\ell})\\
 &=\frac{B^{2\beta_\Phi}}{1-B^{2(1-\alpha_\Phi)}}\Phi'_+((a_N^+)^2).
\end{align*}
Thus we have obtained the estimate
\begin{equation}\label{thus1}
 \tau(\Phi((a_N^+)^2)\leq \frac{4B^{2\beta_\Phi}}{(1-B^{2(1-\alpha_\Phi)})(B-1)^2}\tau\big[\Phi'_+((a_N^+)^2)(z_N^2+b_N)\big].
\end{equation}
Now, set
$$ c=\frac{4B^{2\beta_\Phi}}{(1-B^{2(1-\alpha_\Phi)})(B-1)^2}.$$
By Young's inequality, we see that
\begin{align*}
c\tau\big[&\Phi'_+((a_N^+)^2)(z_N^2+b_N)\big]\\
&=\beta_\Phi^{-1}\tau\bigg[\Phi'_+((a_N^+)^2) \big(\beta_\Phi c(z_N^2+b_N)\big)\bigg]\\
&\leq \beta_\Phi^{-1}\tau\big[\Psi(\Phi'_+((a_N^+)^2))\big]+\beta_\Phi^{-1}\tau\bigg[\Phi\big(\beta_\Phi c(z_N^2+b_N)\big)\bigg]\\
&=\beta_\Phi^{-1}\tau\big[(a_N^+)^2\Phi_+'((a_N^+)^2)-\Phi((a_N^+)^2)\big]+\beta_\Phi^{-1}\tau\bigg[\Phi\big(\beta_\Phi c(z_N^2+b_N)\big)\bigg]\\
&\leq \beta_\Phi^{-1}(\beta_\Phi-1)\tau\big[\Phi((a_N^+)^2)\big]+\beta_\Phi^{-1}\tau\bigg[\Phi\big(\beta_\Phi c(z_N^2+b_N)\big)\bigg],
\end{align*}
where in the last line we have exploited the definition \eqref{defalfabeta} of $\beta_\Phi$. Plugging this into \eqref{thus1} and rearranging terms, we obtain
\begin{align*}
 \tau(\Phi((a_N^+)^2)&\leq \tau\bigg[\Phi\big(\beta_\Phi c(z_N^2+b_N)\big)\bigg]\leq \tau\bigg[\Phi\big(2\beta_\Phi cz_N^2\big)\bigg]+\tau\bigg[\Phi\big(2\beta_\Phi cb_N\big)\bigg].
\end{align*}
There is a symmetric version of this estimate, reading
\begin{align*}
 \tau(\Phi((a_N^-)^2)&\leq \tau\bigg[\Phi\big(\beta_\Phi c(z_N^2+b_N)\big)\bigg]\leq \tau\bigg[\Phi\big(2\beta_\Phi cz_N^2\big)\bigg]+\tau\bigg[\Phi\big(2\beta_\Phi cb_N\big)\bigg].
\end{align*}

It remains to relate $y_N$ to $a_N^\pm$. To this end, note that
$I_{[B^k,\infty)}(y_N)$ is equivalent to a subprojection of $I_{[B^k,\infty)}(a_N^+)$. Indeed, suppose that a nonzero vector $\xi$ belongs to $(I-I_{[B^k,\infty)}(a_N^+))(H)=P_N^{B^{k}}(H)$. From the very construction of the projections $P$ and $R$, we infer that $P_N^{B^k}\leq R_N^{B^k}$ and $R_N^{B^k}y_NR_N^{B^k}<B^k$, so
$$\langle y_N\xi,\xi\rangle=\langle R_N^+y_NR_N^+\xi,\xi\rangle < B^{k}||\xi||^2.$$
Thus $\xi\notin I_{[B^k,\infty)}(y_N)(H)$ which proves the aformentioned equivalence of the projection $I_{[B^k,\infty)}(y_N)$. A similar argument shows that $I_{(-\infty,-B^k]}(y_N)$ is equivalent to a subprojection of $I_{[B^k,\infty)}(a_N^-)$. Consequently, we get
\begin{align*}
 \tau(|y_N|\geq B^k)&=\tau(y_N\leq -B^k)+\tau(y_N\geq B^k)\\
&\leq \tau(a_N^-\leq B^k)+\tau(a_N^+\geq B^k).
\end{align*}
This enforces the appropriate control of the $\Phi$-norm of $y_N$ by the $\Phi$-norms of $a_N^\pm$. Indeed,
\begin{align*}
\tau(\Phi(|y_N|^2))&=\int_0^\infty 2\lambda\Phi'_+(\lambda^2)\tau(|y_N|\geq \lambda)\mbox{d}\lambda\\
&=\sum_{k\in\mathbb{Z}}\int_{B^k}^{B^{k+1}}  2\lambda\Phi'_+(\lambda^2)\tau(|y_N|\geq \lambda)\mbox{d}\lambda\\
&\leq \sum_{k\in\mathbb{Z}} 2B^{k+1}\Phi'_+(B^{2(k+1)})\int_{B^k}^{B^{k+1}} \tau(|y_N|\geq \lambda)\mbox{d}\lambda\\
&\leq \sum_{k\in\mathbb{Z}} 2B^{k+1}\Phi'_+(B^{2(k+1)}) \cdot B^k(B-1)\tau(|y_N|\geq B^k)\\
&\leq \sum_{k\in\mathbb{Z}} 2B^{k+1}\Phi'_+(B^{2(k+1)}) \cdot B^k(B-1)\big(\tau(a_N^+\geq B^k)+\tau(a_N^-\geq B^k)\big).
\end{align*}
Let us split the latter expression into two parts, depending solely on $a_N^+$ and $a_N^-$. We will analyze the first part only, the second is handled with similarly. By \eqref{defalfabeta} and \eqref{altbeta},
$$ \Phi_+'(B^{2(k+1)})\leq \frac{\beta_\Phi \Phi(B^{2(k+1)})}{B^{2(k+1)}}\leq \frac{\beta_\Phi B^{4\beta_\Phi}\Phi(B^{2(k-1)})}{B^{2(k+1)}}\leq \beta_\Phi B^{4(\beta_\Phi-1)}\Phi_+'(B^{2(k-1)}),$$
which implies
\begin{align*}
 &\sum_{k\in\mathbb{Z}} 2B^{k+1}\Phi'_+(B^{2(k+1)}) \cdot B^k(B-1)\tau(a_N^+\geq B^k)\\
 &=B^3\sum_{k\in\mathbb{Z}} 2B^{k-1} \Phi'_+(B^{2(k+1)})\int_{B^{k-1}}^{B^k} \tau(a_N^+\geq \lambda)\mbox{d}\lambda\\
 &\leq \beta_\Phi B^{3+4(\beta_\Phi-1)}\sum_{k\in\mathbb{Z}} 2B^{k-1}\Phi_+'(B^{2(k-1)})\int_{B^{k-1}}^{B^k} \tau(a_N^+\geq \lambda)\mbox{d}\lambda\\
 &\leq \beta_\Phi B^{3+4(\beta_\Phi-1)}\sum_{k\in\mathbb{Z}} \int_{B^{k-1}}^{B^k} 2\lambda \Phi'_+(\lambda^2)\tau(a_N^+\geq \lambda)\mbox{d}\lambda\\
 &=\beta_\Phi B^{3+4(\beta_\Phi-1)}\tau\big(\Phi((a_N^+)^2)\big).
\end{align*}
Putting all the above facts together, we finally arrive at
\begin{equation}\label{Phiin}
\begin{split}
 \tau(\Phi(|y_N|^2))&\leq \beta_\Phi B^{3+4(\beta_\Phi-1)}\bigg[\tau\big(\Phi((a_N^+)^2)\big)+\tau\big(\Phi((a_N^-)^2)\big)\bigg]\\
 &\leq 2\beta_\Phi B^{3+4(\beta_\Phi-1)}\bigg(\tau\bigg[\Phi\big(2\beta_\Phi cz_N^2\big)\bigg]+\tau\bigg[\Phi\big(2\beta_\Phi cb_N\big)\bigg]\bigg).
\end{split}
\end{equation}}
\noindent Let us now multiply the above inequality by $B^{kp}$ and sum over $k\in \mathbb{Z}$. Then the left-hand side of the obtained estimate is equal to $B^{-p}\tau((a_N^+)^p)$; to compute the right-hand side, observe that
\begin{equation}\label{fubini}
\begin{split}
\sum_{k\in\mathbb{Z}} B^{k(p-2)}\left(I-P_N^{B^k}\right)&=\sum_{k\in\mathbb{Z}}\sum_{\ell\geq k} B^{k(p-2)}\left(P_N^{B^{\ell+1}}-P_N^{B^\ell}\right)\\
&=\sum_{\ell\in\mathbb{Z}} \left(P_N^{B^{\ell+1}}-P_N^{B^\ell}\right)\sum_{k\leq \ell} B^{k(p-2)}=\frac{(a_N^+)^{p-2}}{1-B^{2-p}}.
\end{split}
\end{equation}
Thus we have established the estimate
$$ B^{-p}\tau\Big((a_N^+)^p\Big)\leq \frac{4(B-1)^{-2}}{1-B^{2-p}}\tau\Big((a_N^+)^{p-2}\big({x}_N^2+z_N^2\big)\Big).$$
Since $y_N\in L^p(\mathcal M)$, an argument similar to \cite[Lemma 5.3]{JOW} implies that
$a_N^+\in L^p({\mathcal{M}})$. Thus the application of H\"older's inequality and triangle inequality to the previous estimate yields
\begin{align*}
 B^{-p}\left\|a_N^+\right\|_{L^p({\mathcal{M}})}^2&\leq \frac{4(B-1)^{-2}}{1-B^{2-p}}\left\|{x}_N^2+z_N^2\right\|_{L^{p/2}({\mathcal{M}})}\\
 &\leq \frac{4(B-1)^{-2}}{1-B^{2-p}}\left(\left\|{x}_N\right\|_{L^p({\mathcal{M}})}^2+\left\|z_N\right\|^2_{L^{p}({\mathcal{M}})}\right).
\end{align*}
This is equivalent to saying that
\begin{equation}\label{max1}
 \left\|a_N^+\right\|_{L^p({\mathcal{M}})} \leq \frac{2B^{p/2}(B-1)^{-1}}{(1-B^{2-p})^{1/2}}\left(\|{x}_N\|_{L^p({\mathcal{M}})}^2+\|z_N\|^2_{L^{p}({\mathcal{M}})}\right)^{1/2}.
\end{equation}
Symmetrically, we obtain that
\begin{equation}\label{max2}
\|a_N^-\|_{L^p({\mathcal{M}})}\leq \frac{2B^{p/2}(B-1)^{-1}}{(1-B^{2-p})^{1/2}}\left(\|{x}_N\|_{L^p({\mathcal{M}})}^2+\|z_N\|^2_{L^{p}({\mathcal{M}})}\right)^{1/2}.
 \end{equation}
It remains to relate $y_N$ to $a_N^\pm$. To this end, note that
$I_{[B^k,\infty)}(y_N)$ is equivalent to a subprojection of $I_{[B^k,\infty)}(a_N^+)$. Indeed, suppose that a nonzero vector $\xi$ belongs to $\big(I-I_{[B^k,\infty)}(a_N^+)\big)(H)=P_N^{B^{k}}(H)$. From the very construction of the projections $P$ and $R$, we infer that $P_N^{B^k}\leq R_N^{B^k}$ and $R_N^{B^k}y_NR_N^{B^k}<B^k$, so
$$\langle y_N\xi,\xi\rangle=\left\langle P_N^{B^k}y_NP_N^{B^k}\xi,\xi\right\rangle < B^{k}||\xi||^2.$$
Thus $\xi\notin I_{[B^k,\infty)}(y_N)(H)$ which means
$$\Big(I-I_{[B^k,\infty)}(a_N^+)\Big) \wedge I_{[B^k,\infty)}(y_N)=0.$$
Then by the Kaplansky formula (cf. \cite[Theorem 6.1.7]{KR2}), we have
\begin{align*}
I_{[B^k,\infty)}(y_N) & =I_{[B^k,\infty)}(y_N)-\Big(I-I_{[B^k,\infty)}(a_N^+)\Big)\wedge I_{[B^k,\infty)}(y_N)\\
& \sim I_{[B^k,\infty)}(y_N)\vee \Big(I-I_{[B^k,\infty)}(a_N^+)\Big)-\Big(I-I_{[B^k,\infty)}(a_N^+)\Big)\\
&\leq I_{[B^k,\infty)}(a_N^+),
\end{align*}
which proves the aformentioned equivalence of the projection $I_{[B^k,\infty)}(y_N)$. A similar argument shows that $I_{(-\infty,-B^k]}(y_N)$ is equivalent to a subprojection of $I_{[B^k,\infty)}(a_N^-)$. Consequently, we get
\begin{align*}
 \tau\Big(I_{[B^k,\infty)}(|y_N|)\Big)&=\tau\Big(I_{(-\infty,-B^k]}(y_N)\Big)+\tau\Big(I_{[B^k,\infty)}(y_N)\Big)\\
&\leq \tau\Big(I_{[B^k,\infty)}(a_N^-)\Big)+\tau\Big(I_{[B^k,\infty)}(a_N^+)\Big).
\end{align*}
This enforces the appropriate control of the $L^p$ norm of $y_N$ by the $L^p$ norms of $a_N^\pm$. Indeed,
\begin{align*}
\left\|y_N\right\|_{L^p({\mathcal{M}})}^p&=p\int_0^\infty \lambda^{p-1}\tau\Big(I_{[\lambda,\infty)}(|y_N|)\Big)\mbox{d}\lambda\\
&=p\sum_{k\in\mathbb{Z}}\int_{B^k}^{B^{k+1}}  \lambda^{p-1}\tau\Big(I_{[\lambda,\infty)}(|y_N|)\Big)\mbox{d}\lambda\\
&\leq p\sum_{k\in\mathbb{Z}} B^{(k+1)(p-1)}\int_{B^k}^{B^{k+1}} \tau\Big(I_{[\lambda,\infty)}(|y_N|)\Big)\mbox{d}\lambda\\
&\leq p\sum_{k\in\mathbb{Z}} B^{(k+1)(p-1)} B^k(B-1)\tau\Big(I_{[B^k,\infty)}(|y_N|)\Big)\\
&\leq pB^{p-1}(B-1)\sum_{k\in\mathbb{Z}} B^{kp}\left(\tau\Big(I_{[B^k,\infty)}(a_N^-)\Big)+\tau\Big(I_{[B^k,\infty)}(a_N^+)\Big)\right)\\
&=pB^{p-1}(B-1)\frac{\|a_N^-\|_{L^p(\mathcal{M})}^p+\|a_N^+\|_{L^p(\mathcal{M})}^p}{1-B^{-p}},
\end{align*}
where in the last line we have performed a calculation similar to that in \eqref{fubini}. Thus, exploiting \eqref{max1} and \eqref{max2}, we arrive at
$$ \|y_N\|_{L^p({\mathcal{M}})}\leq C_{p,B}\left(\|{x}_N\|_{L^p({\mathcal{M}})}^2+\|z_N\|^2_{L^{p}({\mathcal{M}})}\right)^{1/2},$$
where
$$ C_{p,B}=\left(\frac{2p B^{p-1}(B-1)}{1-B^{-p}}\right)^{1/p}\cdot \frac{2B^{p/2}(B-1)^{-1}}{(1-B^{2-p})^{1/2}}.$$
Let us plug $B=1+1/p$. Since
$$ \frac{9}{4}\leq \left(1+\frac{1}{p}\right)^p\leq 3,$$
we easily check that
$$C_{p,B}\leq \frac{12 p}{\left(1-\left(1+\frac{1}{p}\right)^{2-p}\right)^{1/2}}.$$
This is precisely the claim.
\end{proof}

\section{Some classical inequalities revisited}\label{applications}
\subsection{Burkholder-Gundy inequalities}\label{ABDG}
The first application of the above approach concerns the noncommutative Burkholder-Gundy inequalities, which is the most fundamental result due to Pisier and Xu \cite{PX} in noncommutative martingale theory, where the best constants were investigated in \cite{JX2}. We start with the simpler bound; for the sake of notational convenience, we denote the underlying arbitrary martingale with the letter $y$.

\begin{theorem}
For any $p\geq 2$ and any finite self-adjoint martingale $y=(y_n)_{n=0}^N$, we have the estimate
$$ \|y\|_{L^p(\mathcal{M})}\leq C_p \|y\|_{H^p(\mathcal{M})},$$
where $C_p=O(p)$ as $p\to \infty$. The order is optimal, even in the classical case.
\end{theorem}
\begin{proof}
For $p=2$ the estimate holds with the constant $1$, so we may assume that $p>2$.
Let
$$x_N=z_N:=\left(\sum_{k=0}^N dy_k^2\right)^{1/2}.$$
We now verify that the triple $(x_N, y, z_N)$ satisfies the good-$\lambda$ testing conditions. Indeed, the second assumption (ii) is evident, since $z_N^2\geq dy_k^2$ for each $k$. Concerning the first condition (i), we check that
\begin{align*}
\sum_{n=0}^{N} &\sum_{k=n+1}^N\tau\Big( (R_{n-1}-R_{n})dy_kR_{n-1}dy_k(R_{n-1}-R_{n}) \Big)\\
&\leq\sum_{n=0}^{N} \sum_{k=n+1}^N \tau\Big((R_{n-1}-R_{n})dy_k^2(R_{n-1}-R_{n}) \Big)\\
&\leq\sum_{n=0}^{N}\tau\Big( (R_{n-1}-R_{n})x_N^2(R_{n-1}-R_{n}) \Big)\\
&=\tau\Big((I-R_N)x_N^2\Big).
\end{align*}
Therefore, the application of \eqref{main_moment} is allowed; however, this estimate is precisely the claim, with
$$ C_p=\frac{12 p}{\left(1-\left(1+\frac{1}{p}\right)^{2-p}\right)^{1/2}}\cdot 2^{1/2}. $$
For the optimality of the order $O(p)$, consult e.g. \cite{B2} or \cite{JX2}.
\end{proof}

We turn our attention to the reverse estimate.

\begin{theorem}
For any $p\geq 2$ and any finite self-adjoint martingale $x=(x_n)_{n=0}^N$, we have the bound
$$ \|x\|_{H^p(\mathcal{M})}\leq C_p \|x\|_{L^p(\mathcal{M})},$$
where $C_p=O(p)$ as $p\to \infty$. The order is optimal.
\end{theorem}
\begin{proof}
  Here we will need to embed $\mathcal{M}$ into a larger von Neumann algebra in order to represent the square function of $x$ as a modulus of a certain self-adjoint martingale $y$ (which, in turn, will enable us to use the machinery developed above).  Consider the larger algebra $\mathcal{N}=\mathbb{M}_{N+2}\overline{\otimes}  \mathcal{M}$ equipped with the standard tensor trace (which will be denoted by $\nu$), where $\mathbb{M}_{N+2}$ is the algebra of $(N+2)\times (N+2)$ matrices with the usual trace. This larger algebra can be viewed as $(N+2)\times (N+2)$-matrices with entries belonging to $\mathcal{M}$. We now introduce another sequence $(y_n)_{n=0}^N$, this time with terms in the larger algebra, given by
$$ y_n=\sum_{k=0}^n (e_{1,k+2}+e_{k+2,1})\otimes dx_k.$$
Here $e_{i,j}$ are the standard units of $\mathbb{M}_{N+2}$.
This is a self-adjoint martingale with respect to the filtration $(\mathbb{M}_{N+2}\overline{\otimes} \mathcal{M}_n)_{n=0}^N$. Furthermore, it is easy to see that $ y_n^2\geq e_{11}\otimes S_n^2(x)$, and hence also $ |y_n|\geq e_{11}\otimes S_n(x)$; thus, the analysis of the tail of $S_N(x)$ can be deduced from that of the tail of $y_N$. We will also need to transfer the martingale $x$ into the context of the larger algebra  $\mathcal{N}$. To this end, we consider the process $\widetilde{x}$ defined by
$$d\widetilde{x}_n=(e_{1,1}+e_{k+2,k+2})\otimes dx_n.$$
Then $\widetilde{x}$ is an adapted martingale, with the explicit formula given by
$$ \widetilde{x}_n=e_{1,1}\otimes x_n+\sum_{k=0}^n  e_{k+2,k+2}\otimes dx_k.$$

Consider the triple $(\widetilde x_N, y, z_N)$, where $ z_N=\left(\sum_{k=0}^N |d\widetilde{x}_k|^p\right)^{1/p}$. We shall verify that $(\widetilde x_N, y, z_N)$ satisfies the good-$\lambda$ testing conditions. First, observe that $dy_n^2= (d\widetilde{x}_n)^2$ for each $n$ and hence the first assumption (i) is satisfied. Indeed,
\begin{align*}
\sum_{n=0}^{N} & \sum_{k=n+1}^N \nu\Big((R_{n-1}-R_{n})dy_kR_{n-1}dy_k(R_{n-1}-R_{n}) \Big)\\
&\leq\sum_{n=0}^{N} \sum_{k=n+1}^N \nu\Big((R_{n-1}-R_{n})dy_k^2(R_{n-1}-R_{n}) \Big)\\
&=\sum_{n=0}^{N} \sum_{k=n+1}^N \nu\Big((R_{n-1}-R_{n})(d\widetilde{x}_k)^2(R_{n-1}-R_{n})\Big),
\end{align*}
which, by the martingale property of $\widetilde{x}$, is equal to
\begin{align*}
\sum_{n=0}^{N} &\nu\Big((R_{n-1}-R_{n})(\widetilde{x}_N-\widetilde{x}_n)^2(R_{n-1}-R_{n}) \Big)\\
&=\sum_{n=0}^{N} \nu\Big((R_{n-1}-R_{n})(\widetilde{x}_N^2-\widetilde{x}_n^2)(R_{n-1}-R_{n}) \Big)\\
&\leq \sum_{n=0}^{N} \nu\Big((R_{n-1}-R_{n})\widetilde{x}_N^2(R_{n-1}-R_n)\Big)\\
&=\nu\Big((I-R_N)\widetilde{x}_N^2\Big).
\end{align*}
Furthermore, the second assumption (ii) is also satisfied. Since the map $t\mapsto t^{2/p}$ is operator-monotone, we easily see that $z_N^2\geq (|d\widetilde{x}_k|^p)^{2/p}=(d\widetilde{x}_k)^2$ for each $k$, and hence also $z_N^2\geq dy_k^2$ for all $k$; the latter bound clearly yields the validity of (ii).

Since the triple $(\widetilde x_N, y, z_N)$ satisfies the good-$\lambda$ testing conditions, we may apply the inequality \eqref{main_moment} and obtain
$$ \|S_N(x)\|_{L^p(\mathcal{M})}\leq \|y_N\|_{L^p(\mathcal{N})}\leq \frac{12p}{\left(1-\left(1+\frac{1}{p}\right)^{2-p}\right)^{1/2}}\left(\|\widetilde{x}_N\|_{L^p(\mathcal{N})}^2+\|z_N\|_{L^p(\mathcal{N})}^2\right)^{1/2}.$$
By interpolation, we have the following estimate (see also Junge and Xu \cite{JX2}):
\begin{equation}\label{boundb}
\|z_N\|_{L^{p}(\mathcal{N})}=\left(\sum_{k=0}^N \|d\widetilde{x}_k\|_{L^{p}(\mathcal{N})}^p\right)^{1/p}\leq 2^{1-2/p}\|\widetilde{x}_N\|_{L^{p}(\mathcal{N})},
\end{equation}
and hence
$$ \|S_N(x)\|_{L^p(\mathcal{M})}\leq \|y_N\|_{L^p(\mathcal{N})}\leq \frac{12p(1+2^{2-4/p})^{1/2}}{\left(1-\left(1+\frac{1}{p}\right)^{2-p}\right)^{1/2}}\|\widetilde{x}_N\|_{L^p(\mathcal{N})}.$$
Similar to \eqref{boundb}, we obtain that
$$\|\widetilde{x}_N\|_{L^p(\mathcal{N})}^p=\|x_N\|_{L^p(\mathcal{M})}^p+\sum_{k=0}^N \|dx_k\|_{L^p(\mathcal{M})}^p\leq (1+2^{p-2})\|x_N\|_{L^p(\mathcal{M})}^p,$$
which combined with the previous bound finally gives
$$ \|S_N(x)\|_{L^p(\mathcal{M})}\leq C_p\|x_N\|_{L^p(\mathcal{M})},$$
with
$$ C_p=\frac{12p(1+2^{2-4/p})^{1/2}(1+2^{p-2})^{1/p}}{\left(1-\left(1+\frac{1}{p}\right)^{2-p}\right)^{1/2}}.$$
This yields the desired estimate. The sharpness of the order follows from \cite{JX2}.
\end{proof}

\subsection{Inequalities for martingale transforms}\label{transforms}
Our approach immediately yields the $L^p$ boundedness for martingale transforms in the range $1<p<\infty$. It was originally proved by Randrianantoanina \cite{R1}.

\begin{theorem}
Suppose that $x=(x_n)_{n\geq 0}$, $y=(y_n)_{n\geq 0}$ are self-adjoint martingales such that for each $n\geq 0$ we have $dy_n=v_ndx_n$ where $v=(v_n)_{n\geq 0}$ is a sequence with values in $[-1,1]$. Then for any $1<p<\infty$ there is a finite constant $C_p$ such that
$$ \|y_N\|_{L^p(\mathcal{M})}\leq C_p\|x_N\|_{L^p(\mathcal{M})}.$$
Furthermore, $C_p$ is of order $O((p-1)^{-1})$ as $p\to 1$ and $O(p)$ as $p\to \infty$. Both orders are optimal, as they are already optimal in the classical case.
\end{theorem}
\begin{proof}
For $p=2$, the inequality holds with the constant $1$. Suppose that $p>2$.
Set $z_N=\left(\sum_{k=0}^N |dx_k|^p\right)^{1/p}$. Then $z_N^2$ majorizes $dx^2$ and $dy^2$.
On the other hand, we have $dy_n^2\leq dx_n^2$ for all $n$ and hence a calculation from the previous subsection shows that the triple $(x_N, y, z_N)$ satisfies the good-$\lambda$ testing conditions.  Therefore, the estimate \eqref{main_moment} and \eqref{boundb} give
$$ \|y_N\|_{L^p(\mathcal{M})}\leq \frac{12p(1+2^{2-4/p})^{1/2}}{\left(1-\left(1+\frac{1}{p}\right)^{2-p}\right)^{1/2}}\|x_N\|_{L^p(\mathcal{M})}.$$
Note that the above constant is of order $O(p)$ as $p\to \infty$. By duality, we obtain the corresponding bound in the range $1<p<2$, with the constant of the order $O((p-1)^{-1})$ as $p\to 1$. Let us briefly remark here that standard interpolation argument allows to remove the blow-up of the constant as $p\to 2$. In the classical case, the optimal choice for $C_p$ is equal to $\max\{p-1,(p-1)^{-1}\}$ (see \cite{B0,B2}); this yields the optimality of the order above.
\end{proof}

\subsection{Noncommutative Stein and dual Doob inequalities}\label{ncDoob} Here is our next application. They can be respectively found in \cite{PX} and \cite{J}.

\begin{theorem}
Fix $p>1$ and a nonnegative integer $N$. Then for any sequence $(u_n)_{n=0}^N$ of elements of $\mathcal{M}$ (not necessarily adapted) we have
\begin{equation}\label{Stein}
 \left\|\left(\sum_{n=0}^N |\mathcal{E}_n(u_n)|^2\right)^{1/2}\right\|_{L^p(\mathcal{M})}\leq C_p \left\|\left(\sum_{n=0}^N |u_n|^2\right)^{1/2}  \right\|_{L^p(\mathcal{M})},
\end{equation}
where $C_p$ is of the order $O((p-1)^{-1})$ as $p\to 1$ and $O(p)$ as $p\to \infty$.
Furthermore, if $1\leq p<\infty$ and $u_n$ are positive, then
\begin{equation}\label{Doob}
 \left\|\sum_{n=0}^N \mathcal{E}_n(u_n)\right\|_{L^p(\mathcal{M})}\leq \widetilde{C}_p \left\|\sum_{n=0}^N u_n  \right\|_{L^p(\mathcal{M})},
\end{equation}
where $\widetilde{C}_p$ is of order $O(p^2)$ as $p\to \infty$. All the orders are the best possible.
\end{theorem}
\begin{proof}
We start with dual Doob's inequality \eqref{Doob}.
As in the context of Burkholder-Gundy inequalities, we start with an appropriate modification of the von Neumann algebra which enables to fit the above setting into the framework of Section 3. Let $(\Omega,\F,\mathbb{P})$ be a classical probability space and let $\e_0$, $\e_1$, $\e_2$, $\ldots$, $\e_N$ be a sequence of independent Rademacher variables. Consider the algebra $\mathcal{N}=\mathbb{M}_{N+2}\otimes L^\infty(\Omega,\F,\mathbb{P})\otimes \mathcal{M}$. We equip $\mathcal{N}$ with the usual tensor trace $\nu$ and the filtration $(\mathcal{N}_n)_{n=0}^N=(\mathbb{M}_{N+2}\otimes L^\infty(\Omega,\F_n,\mathbb{P})\otimes \mathcal{M}_n)_{n=0}^N$, where $\F_n$ stands for the $\sigma$-field generated by the variables $\e_0$, $\e_1$, $\e_2$, $\ldots$, $\e_n$. Consider the operator
$$ x_N=z_N:=e_{1,1}\otimes 1\otimes \left(\sum_{k=0}^N u_k\right)^{1/2}+\sum_{k=0}^N e_{k+2,k+2}\otimes 1\otimes u_k^{1/2}$$
and the sequence $y=(y_n)_{n=0}^N$ uniquely determined by
$$  dy_k=(e_{1,k+2}+e_{k+2,1})\otimes \e_k \otimes \mathcal{E}_k(u_k)^{1/2}$$
for $k=0,\,1,\,2,\,\ldots,\,N$. Clearly, the sequence $y$ is a martingale with respect to $(\mathcal{N}_n)_{n=0}^N$. Let us verify that the triple $(x_N, y, z_N)$ satisfies the good-$\lambda$ testing conditions. Observe that
\begin{equation}\label{dy}
 dy_k^2=(e_{1,1}+e_{k+2,k+2})\otimes 1\otimes \mathcal{E}_k(u_k)=\overline{\mathcal{E}}_k\Big((e_{1,1}+e_{k+2,k+2})\otimes 1\otimes u_k\Big),
\end{equation}
where $\overline{\mathcal{E}}_k$ is the conditional expectation associated with the subalgebra  $\mathcal{N}_k$. Consequently,
\begin{align*}
&\nu \left(\sum_{n=0}^{N} \sum_{k=n+1}^N (R_{n-1}-R_{n})dy_kR_{n-1}dy_k(R_{n-1}-R_{n}) \right)\\
&\quad \leq\nu\left(\sum_{n=0}^{N} \sum_{k=n+1}^N(R_{n-1}-R_{n})dy_k^2(R_{n-1}-R_{n}) \right)\\
&\quad \leq\nu\left(\sum_{n=0}^{N} \sum_{k=n+1}^N(R_{n-1}-R_{n})\left((e_{1,1}+e_{k+2,k+2})\otimes 1\otimes u_k\right)\right),
\end{align*}
in the light of \eqref{dy}. We can split the latter expression into two parts:
\begin{align*}
&\nu\left(\sum_{n=0}^{N} \sum_{k=n+1}^N(R_{n-1}-R_{n})(e_{1,1}\otimes 1\otimes u_k)\right)\\
& \quad +\nu\left(\sum_{n=0}^{N} \sum_{k=n+1}^N(R_{n-1}-R_{n})(e_{k+2,k+2}\otimes 1\otimes u_k)\right)\\
&\leq\sum_{n=0}^{N} \nu \left((R_{n-1}-R_n)\left(e_{1,1}\otimes 1\otimes \sum_{k=0}^N u_k\right)\right)\\
&\quad +\nu \left(\sum_{k=0}^{N} \sum_{n=0}^{k-1}(R_{n-1}-R_{n})(e_{k+2,k+2}\otimes 1\otimes u_k)\right)\\
&\leq \nu\left((I-R_N)\left(e_{1,1}\otimes 1\otimes \sum_{k=0}^N u_k+\sum_{k=0}^N e_{k+2,k+2}\otimes 1\otimes u_k\right)\right)\\
&=\nu \Big((I-R_N)x_N^2\Big),
\end{align*}
which is the condition (i). Concerning the assumption (ii), we check that for any projection $P\in \mathcal{N}_k$,
$$ \nu\big(Pdy_k^2P\big)=\tau\Big(P((e_{1,1}+e_{k+2,k+2})\otimes 1\otimes u_k)P\Big)\leq \tau\big(Pz_N^2P\big),$$
as desired. Therefore, the inequality \eqref{main_moment} gives
\begin{equation}\label{shoulder}
 \|y_N\|_{L^p(\mathcal{N})}\leq \frac{12p}{\left(1-\left(1+\frac{1}{p}\right)^{2-p}\right)^{1/2}}\cdot 2^{1/2}\|x_N\|_{L^p(\mathcal{N})}
 \end{equation}
for any $p>2$. We verify directly that
$$ y_N^2\geq e_{1,1}\otimes 1\otimes \sum_{n=0}^N \mathcal{E}_n(u_n),$$
which gives $$||y_N||_{L^p(\mathcal{N})}\geq \left\|\sum_{n=0}^N \mathcal{E}_n(u_n)\right\|^{1/2}_{L^{p/2}(\mathcal{M})}.$$
Furthermore,
\begin{align*}
 \|x_N\|_{L^p(\mathcal{N})}&=\left(\left\|\sum_{n=0}^N u_n\right\|_{L^{p/2}(\mathcal{M})}^{p/2}+\sum_{n=0}^N \|u_n\|_{L^{p/2}(\mathcal{M})}^{p/2}\right)^{1/p}\leq 2^{1/p}\left\|\sum_{n=0}^N u_n\right\|_{L^{p/2}(\mathcal{M})}^{1/2},
\end{align*}
by interpolation. Combining these observations with \eqref{shoulder} gives the desired dual Doob's inequality \eqref{Doob} with the constant of the order $O(p^2)$ as $p\to \infty$.

It remains to handle \eqref{Stein}. For $p=2$ there is nothing to prove, the estimate holds with the constant $1$, since $|\mathcal{E}_n(u_n)|^2\leq \mathcal{E}_n(|u_n|^2)$. For $p>2$, we deduce the inequality \eqref{Stein} immediately from Doob's estimate (apply \eqref{Doob} to the positive sequence $(|u_n|^2)_{n=0}^N$ and use the estimate $|\mathcal{E}_n(u_n)|^2\leq \mathcal{E}_n(|u_n|^2)$ again). The case $p<2$ of \eqref{Stein} follows at once by duality. As in the case of martingale transforms, an easy interpolation argument  allows to remove the blow-up of the constant as $p\to 2$. For the optimality of the orders of the constants in \eqref{Stein} and \eqref{Doob}, we refer the reader to \cite{JX}.
\end{proof}

\def\rrrrrrrr{\subsection{Noncommutative Burkholder/Rosenthal inequalities}
Our final application concerns the noncommutative version of Burkholder/Rosenthal inequalities.

\begin{theorem}
Fix $p\geq 2$. Let $y=(y_n)_{n=0}^N$ be a finite martingale and let  $s_{p,d}(y)=\left(\sum_{n=0}^N |dy_n|^p\right)^{1/p}$ be the associated diagonal square function. Then  we have the estimate
$$ ||y_N||_{L^p(\mathcal{M})}\leq C_p\Big(||s_N(y)||_{L^p(\mathcal{M})}+||s_{p,d}(y)||_{L^p(\mathcal{M})}\Big),$$
with the constant $C_p$ of the order $O(p)$ as $p\to \infty$.
\end{theorem}
\begin{proof}
If $p=2$, then the inequality is obvious, so we assume that $p$ is strictly bigger than $2$. We set $x_N=s_N(y)$ and $z_N=s_{p,d}(y)$. Then the triple $(x_N, y, z_N)$ satisfies the good-$\lambda$ testing conditions. Indeed, by properties of conditional expectation,
\begin{align*}
&\sum_{n=0}^{N} \sum_{k=n+1}^N \tau\Big((R_{n-1}-R_{n})dy_kR_{n-1}dy_k(R_{n-1}-R_{n}) \Big)\\
&\leq\sum_{n=0}^{N} \sum_{k=n+1}^N \tau\Big((R_{n-1}-R_{n})dy_k^2(R_{n-1}-R_{n}) \Big)\\
&=\sum_{n=0}^{N} \sum_{k=n+1}^N \tau\Big((R_{n-1}-R_{n})\mathcal{E}_{k-1}(dy_k^2)(R_{n-1}-R_{n}) \Big)\\
&\leq\sum_{n=0}^{N} \sum_{k=n+1}^N \tau\Big((R_{n-1}-R_{n})s_N(y)^2 \Big)\\
&=\tau\Big((I-R_N)x_N^2\Big),
\end{align*}
which implies that the assumption (i) is satisfied. The condition (ii) is also satisfied: we have already verified it in \S\ref{ABDG}.  Consequently, the inequality \eqref{main_moment} gives, for $p>2$,
$$ ||y_N||_{L^p(\mathcal{M})}\leq \frac{12 p}{\left(1-(1+\frac{1}{p})^{2-p}\right)^{1/2}}\bigg(||s_N(y)||^2_{L^p(\mathcal{M})}+||s_{p,d}(y)||_{L^p(\mathcal{M})}^2\bigg)^{1/2}.$$
This immediately yields the desired claim.
\end{proof}}

\section{Noncommutative tangent sequences, improved Doob's inequality and a certain class of Schur multipliers}
This part of the paper contains further very interesting applications of the good-$\lambda$ method. Namely, we will study certain novel estimates for \emph{tangent} sequences, which later will be connected to an enhanced Doob's inequality and the construction of a certain class of Schur multipliers. Let us start with the formal definition of tangency. In the classical case, this concept was originally introduced by Kwapie\'n and Woyczy\'nski in \cite{KW}; we propose the following noncommutative extension.

\begin{definition}
Two adapted sequences $a=(a_n)_{n\geq 0}$ and $b=(b_n)_{n\geq 0}$ are said to be tangent if for any bounded Borel function $\varphi$ we have
\begin{equation}\label{tangent_definition}
 \mathcal{E}_{n-1}(\varphi(a_n))=\mathcal{E}_{n-1}(\varphi(b_n)),\qquad n=0,\,1,\,2,\,\ldots.
\end{equation}
\end{definition}

It is not difficult to see that if the sequences $a$, $b$ in the above definition consist of self-adjoint terms only, then \eqref{tangent_definition} is equivalent to saying that for any $\lambda\in \R$ we have the equality
$$ \mathcal{E}_{n-1}(I_{(\lambda,\infty)}(a_n))=\mathcal{E}_{n-1}(I_{(\lambda,\infty)}(b_n)),\qquad n=0,\,1,\,2,\,\ldots.$$
In the classical setting, the above condition amounts to saying that for any $n$, the conditional distributions of $a_n$ and $b_n$ with respect to the algebra $\mathcal{M}_{n-1}$ coincide.
However, we would like to emphasize here that our definition makes sense also in the case when $a$, $b$ contain some non-self-adjoint operators.

 Tangent sequences  have played an important role in the classical probability theory. They are in close connection to the so-called decoupling technique (i.e., comparison of the size of a given probabilistic object with its version in which some components have been replaced with independent copies), which has been exploited extensively in the literature. We shall mention here several relevant works. The paper of McConnell and Taqqu \cite{MT} contains the applications to multilinear forms and double stochastic integrals, while the results of Hitczenko \cite{Hi-1, Hi0, Hi2}, Kwapie\'n and Woyczy\'nski \cite{KW}, Os\k ekowski \cite{O0} and Zinn \cite{Z} concern estimates for tangent martingales and sums of positive random variables (with or without certain additional assumptions of the sequences). For further extensions, consult the papers \cite{A,dP,JN} on $U$-statistics, and the articles \cite{M} for applications of decoupling to Malliavin calculus.
 We should also mention here the papers of Cox, van Neerven, Veraar and Weis \cite{CV2,NVW,NW1, NW2} on stochastic integration in Banach spaces which also depend heavily on decoupling and tangent sequences. Finally, we would like to refer the interested reader to the extensive monographs of de la Pe\~na and Gin\'e \cite{DG} and Kwapie\'n and Woyczy\'nski \cite{KW2} for more on the subject.

Before we proceed to the description of our results, let us present two noncommutative examples concerning the tangency condition. We will also encounter an interesting example in the proof of Theorem \ref{no_tangent} below.

\begin{example}
Suppose that $(u_n)_{n\geq 0}$ is a predictable sequence of operators (i.e., for each $n$ the operator $u_n$ belongs to $\mathcal{M}_{n-1}$). Let $(\xi_n)_{n\geq 0}$, $(\tilde{\xi}_n)_{n\geq 0}$ be two tangent sequences of classical random variables on some probability space $(\Omega,\F,\mathbb{P})$, adapted to a given filtration $(\F_n)_{n\geq 0}$. Then the sequences $(\xi_n\otimes u_n)_{n\geq 0}$ and $(\tilde{\xi}_n\otimes u_n)_{n\geq 0}$ (considered on the tensor von Neumann algebra $\mathcal{N}=L^\infty(\Omega,\F,\mathbb{P})\bar{\otimes} \mathcal{M}$ equipped with the tensor trace and the filtration $(L^\infty(\Omega,\F_n,\mathbb{P})\bar{\otimes} \mathcal{M}_n)_{n\geq 0}$) are tangent.
\end{example}

\begin{example}\label{Schur}
Our next construction will have more ``noncommutative'' flavor. Let $N$ be a positive integer and let $\mathcal{M}=\mathbb{M}_N$ be the algebra of matrices of dimension $N\times N$ equipped with the usual trace $\tau=\operatorname*{Tr}$. We consider the following filtration $(\mathcal{M}_n)_{n=0}^N$, studied by Junge and Xu in \cite{JX2}: for each $n$,
$$\mathcal{M}_n=\{\mu I+A\,:\,\mu \in \mathbb{C},\,A\mbox{ is an $n\times n$ matrix, placed in the upper left corner}\}.$$
(For $n=0$, $\mathcal{M}_n$ is just the trivial algebra $\{\mu I\,:\,\mu\in \mathbb{C}\}$). The associated conditional expectations $(\mathcal{E}_n)_{n=-1}^N$ act as follows. We have $\mathcal{E}_{-1}=\mathcal{E}_0$ and, for each $k=0,\,1,\,\ldots,\,N$ and $a=(\alpha_{i,j})_{1\leq i,j\leq N}\in \mathcal{M}$, $\mathcal{E}_ka$ is a matrix whose upper-left corner of dimension $k\times k$ coincides with that of $a$, the remaining part of the main diagonal is occupied by the numbers $(\alpha_{k+1,k+1}+\alpha_{k+2,k+2}+\ldots+\alpha_{N,N})/(N-k)$, and all the other entries are zero.

 For any $k=1,\,2,\,\ldots,\,N$, let $A_k=(c_{i,j})_{1\leq i,j\leq k-1}$ be a Hermitian matrix of dimension $(k-1)\times (k-1)$ and, for $i\in \{1,2,\ldots,k\}$ and $j\in \{1,2,\ldots,k-1\}$, let $\lambda_{i,j}$ be some fixed complex numbers. Finally, let $\alpha_1$, $\ldots$, $\alpha_N$, $\beta_1$, $\beta_2$, $\ldots$, $\beta_N$ be given real numbers. Define $(a_k)_{k=0}^N$ and $(b_k)_{k=0}^N$ by $a_0=b_0=\beta_0 I$ and, for $k\geq 1$,
  $$
    a_k=\left[\begin{array}{cccccccccc}
 c_{11} & c_{12} & \ldots & c_{1,k-1} & \lambda_{k,1} & & & &\\
 c_{21} & c_{22} & \ldots & c_{2,k-1}      & \lambda_{k,2} & &  & &\\
  \ldots & \ldots & \ldots & \ldots & \ldots & & & &\\
     c_{k-1,1} & c_{k-1,2} & \ldots & c_{k-1,k-1} & \lambda_{k,k-1} & & & &\\
      \overline{\lambda}_{k,1}& \overline{\lambda}_{k,2}& \ldots& \overline{\lambda}_{k,k-1} & \alpha_k & & & &\\
     & & & & & \beta_k & & &\\
       & & & & & & \beta_k & &\\
       & & & & & & & \ldots &\\
       & & & & & & & & \beta_k
    \end{array}\right]
 $$
and
$$
    b_k=\left[\begin{array}{cccccccccc}
 c_{11} & c_{12} & \ldots & c_{1,k-1} & -\lambda_{k,1} & & & &\\
 c_{21} & c_{22} & \ldots & c_{2,k-1}      & -\lambda_{k,2} & &  & &\\
  \ldots & \ldots & \ldots & \ldots & \ldots & & & &\\
     c_{k-1,1} & c_{k-1,2} & \ldots & c_{k-1,k-1} & -\lambda_{k,k-1} & & & &\\
      -\overline{\lambda}_{k,1}& -\overline{\lambda}_{k,2}& \ldots& -\overline{\lambda}_{k,k-1} & \alpha_k & & & &\\
     & & & & & \beta_k & & &\\
       & & & & & & \beta_k & &\\
       & & & & & & & \ldots &\\
       & & & & & & & & \beta_k
    \end{array}\right]
 $$
 (here and below, we use the convention that all blank entries are zero).
 Then $(a_k)_{k=0}^N$ and $(b_k)_{k=0}^N$ are tangent. Indeed, it is easy to check by induction that for any positive integer $m$, the matrices $a_k^m$ and $b_k^m$ are of the above form (of course, with  some different choice of the parameters $A_k$, $\lambda_{i,j}$, $\alpha_k$ and $\beta_k$). This immediately gives the equality $\mathcal{E}_{k-1}(a_k^m)=\mathcal{E}_{k-1}(b_k^m)$ and hence also
 $$ \mathcal{E}_{k-1}(P(a_k))=\mathcal{E}_{k-1}(P(b_k)),$$
 for any polynomial $P$. Since $a_k$, $b_k$ are bounded, this yields the tangency condition.
\end{example}

There is a natural question whether the tangency assumption implies certain estimates for the sequences involved. Motivated by the commutative comparison results obtained by Burkholder \cite{B0,B2}, Hitczenko \cite{Hi-1,Hi0, Hi2}, Kwapie\'n and Woyczy\'nski \cite{KW} and \cite{KW2}, Os\k ekowski \cite{O0} and Zinn \cite{Z}, we will study this problem in two cases: when the sequences under investigation are noncommutative martingale difference sequences or adapted positive operators. Then we will present some interesting applications of the results obtained. We split the remaining part of this section into four parts.

\subsection{Martingale inequalities} Suppose that $x=(x_n)_{n\geq 0}$, $y=(y_n)_{n\geq 0}$ are self-adjoint martingales with tangent difference sequences. Consider the following two problems.

\smallskip

(A) Does there exist a universal constant $C$ such that the weak type estimate
$$ \tau(I_{[1,\infty)}(|y_N|))\leq C\|x_N\|_1,\qquad N=0,\,1,\,2,\,\ldots,$$
holds true?

(B) Given $1<p<\infty$, does there exist a finite constant $C_p$ depending only on $p$ for which we have the inequality
$$ \|y_N\|_p\leq C_p\|x_N\|_p,\qquad N=0,\,1,\,2,\,\ldots ?$$

\smallskip

In the classical setting, the answer to both (A) and (B) is positive: see e.g. Kwapie\'n and Woyczy\'nski \cite{KW,KW2} and Os\k ekowski \cite{O0}.
Our first result here is somewhat surprising and shows that in the noncommutative realm the answer to (A) and to a part of (B) is negative.

\begin{theorem}\label{no_tangent}
The weak-type inequality and the $L^p$ estimate ($1<p<2$) do not hold in general for tangent martingales.
\end{theorem}
\begin{proof}
Let $N$ be a large positive odd integer and assume that $\e_1$, $\e_2$, $\ldots$, $\e_N$ are independent Rademacher variables on some (classical) probability space $(\Omega,\F,\mathbb{P})$. Suppose further that for each $n\geq 0$, $\F_n$ is the $\sigma$-algebra generated by $\e_1$, $\e_2$, $\ldots$, $\e_n$ (with the convention $\F_0=\{\emptyset,\Omega\}$ and $\F_n=\F$ if $n>N$). Consider the algebra $\mathcal{M}=L^\infty(\Omega,\F,\mathbb{P})\overline{\otimes}\mathbb{M}_{N+1}$ equipped with the tensor product trace, where, as usual, $\mathbb{M}_{N+1}$ stands for the algebra of $(N+1)\times (N+1)$ matrices with the usual trace. Introduce the filtration  $\mathcal{M}_n=L^\infty(\Omega,\F_n,\mathbb{P})\overline{\otimes}\mathbb{M}_{N+1}$, $n=0,\,1,\,2,\,\ldots$. Finally, consider the sequences $dx=(dx_n)_{n\geq 0}$, $dy=(dy_n)_{n\geq 0}$ given by $dx_n=\e_n \otimes (e_{1,n+1}+ e_{n+1,1})$ and $dy_n=\e_n \otimes (e_{11}+e_{n+1,n+1})$, $n=1,\,2,\,\ldots,\,N$. For remaining $n$, we let $dx_n=dy_n=0$. It is obvious that $dx$ and $dy$ are martingale differences, and we will check now that the tangency condition is satisfied. To this end, observe that for each $n$, if $k$ is an even integer, then we have
$$ dx_n^k=dy_n^k=\begin{cases}
1\otimes (e_{11}+e_{n+1,n+1}) & \mbox{if }1\leq n\leq N,\\
0 & \mbox{otherwise},
\end{cases}$$
and hence $\mathcal{E}_{n-1}(dx_n^k)=\mathcal{E}_{n-1}(dy_n^k)$. On the other hand, if $k$ is odd, then we have $dx_n^k=dx_n$ and $dy_n^k=dy_n$, so $\mathcal{E}_{n-1}(dx_n^k)=0=\mathcal{E}_{n-1}(dy_n^k)$. Consequently, we see that for any polynomial $P$ we have
$$ \mathcal{E}_{n-1}(P(dx_n))=\mathcal{E}_{n-1}(P(dy_n)).$$
Since $dx_n$ and $dy_n$ are bounded, the above equality holds if $P$ is replaced by any Borel function $\varphi$, so $x$ and $y$ are tangent. Directly from the definition of $dx$ and $dy$, we compute that
$$ y_N=\left[\begin{array}{ccccc}
\e_1+\e_2+\ldots+\e_N &  &  &  & \\
 & \e_1 &  &  & \\
 &  & \e_2 &  & \\
 &  &  & \ldots & \\
 &  &  &  & \e_N
\end{array}\right],$$
which implies
$$ |y_N|=\left[\begin{array}{ccccc}
|\e_1+\e_2+\ldots+\e_N| &  &  &  & \\
 & 1 &  &  & \\
 &  & 1 &  & \\
 &  &  & \ldots & \\
 &  &  &  & 1
\end{array}\right]$$
and $\tau(I_{[1,\infty)}(|y_N|))=N+1$ (here we use the assumption that $N$ is odd: this guarantees that the entry in the upper-left corner of $|y_N|$ is at least $1$). On the other hand, we have
$$ x_N=\left[\begin{array}{ccccc}
 & \e_1 & \e_2 & \ldots & \e_N\\
\e_1 &  &  &  & \\
\e_2 &  &  &  & \\
\ldots & &  &  & \\
\e_N &  & & &
\end{array}\right].$$
To derive the trace of $|x_N|$, note that
$$ x_N^2=\left[\begin{array}{ccccc}
N &  &  &  & \\
 & \e_1^2 & \e_1\e_2 & \ldots & \e_1\e_N\\
 & \e_2\e_1 & \e_2^2 & \ldots & \e_2\e_N\\
 & \ldots & \ldots & \ldots & \ldots\\
 & \e_N\e_1 & \e_N\e_2 & \ldots & \e_N^2
\end{array}\right]=N(P_1+P_\e),$$
where $P_1$, $P_\e$ are the projections onto the one-dimensional spaces spanned by $(1,0,0,\ldots,0)$ and $(0,\e_1,\e_2,\ldots,\e_N)$, respectively. These spaces are orthogonal, so $|x_N|=\sqrt{N}(P_1+P_\e)$ and hence $\tau(|x_N|)=2\sqrt{N}$. We have thus obtained that
$$ \frac{\tau(I_{[1,\infty)}(|y_N|))}{\tau(|x_N|)}=\frac{N+1}{2\sqrt{N}},$$
so the weak-type (1,1) estimate cannot hold with any finite universal constant. Furthermore,  we have $ ||y_N||_{p}\geq (N+1)^{1/p}$ and $||x_N||_p=2^{1/p}\sqrt{N}$, so the $L^p$ estimate does not hold for $1<p<2$ as well.
\end{proof}

However, we will prove that in the range $2\leq p<\infty$ the $L^p$-inequality for tangent sequences does hold true. Actually, we will show a much stronger statement, which is of independent interest and is motivated by the following result obtained by Os\k ekowski in \cite{O0}. Suppose that $p\geq 2$ is  a fixed number and $x=(x_n)_{n\geq 0}$, $y=(y_n)_{n\geq 0}$ are \emph{commutative} $L^p$-bounded martingales satisfying
\begin{equation}\label{assumption_osekowski}
 \mathcal{E}_{n-1}(dy_n^2)\leq \mathcal{E}_{n-1}(dx_n^2)\quad \mbox{ and }\quad \mathcal{E}_{n-1}(|dy_n|^p)\leq \mathcal{E}_{n-1}(|dx_n|^p)
\end{equation}
for each $n$. Then we have the moment estimate
$$\|y_N\|_p\leq 3p\|x_N\|_p,\qquad N=0,\,1,\,2,\,\ldots.$$
The good-$\lambda$ approach developed in the previous sections will enable us to establish the following stronger version of this result in the noncommutative setting.

\begin{theorem}\label{stronger_theorem}
Let $2\leq p<\infty$. Suppose that $x=(x_n)_{n\geq 0}$, $y=(y_n)_{n\geq 0}$ are self-adjoint, $L^p$-bounded martingales such that for some $\kappa\geq 1$,
\begin{equation}\label{domination}
 \mathcal{E}_{n-1}(dy_n^2)\leq \mathcal{E}_{n-1}(dx_n^2)\quad \mbox{ and }\quad \|dy_n\|_p\leq \kappa\|dx_n\|_p
\end{equation}
for all $n=0,\,1,\,2,\,\ldots$. Then
\begin{equation}\label{boundxy}
 ||y_N||_p\leq C_p \kappa||x_N||_p,
\end{equation}
for some constant $C_p$ of order $O(p)$ as $p\to \infty$. The order is already the best possible for tangent martingales in the commutative case.
\end{theorem}

Before we turn to the proof, we make three important observations.

\begin{remark}
(i) The inequality \eqref{boundxy}, with the worse constant of order $O(p^2)$ as $p\to \infty$, can be immediately deduced from the noncommutative version of Burkholder/Rosenthal inequality (see \cite{JX}, \cite{Z}): indeed, directly from \eqref{domination} we infer that
\begin{align*}
 \|y_N\|_p&\lesssim p\left(\|s_N(y)\|_p+\left(\sum_{n=0}^N \|dy_n\|_p^p\right)^{1/p}\right)\\
 &\leq p\left(\|s_N(x)\|_p+\kappa\left(\sum_{n=0}^N \|dx_n\|_p^p\right)^{1/p}\right)\lesssim p^2\kappa\|x_N\|_p.
\end{align*}
Here $(s_n(x))_{n\geq 0}$ stands for the conditional square function of $x$, while the symbol ``$A\lesssim B$'' means that the ratio $A/B$ is bounded from above by a universal constant. It is worth stressing that this type of argument cannot yield the sharp version of \eqref{boundxy} (i.e., with the linear growth of $C_p$ with respect to $p$) even in the commutative case. Indeed, exploiting the best orders of constants in Burkholder/Rosenthal inequality (see \cite{Hi}), one gets the non-optimal order $O(p\sqrt{p}/\log p)$ above.

(ii) By Theorem \ref{no_tangent}, the above result cannot hold in the range $1<p<2$.

(iii) The statement above is indeed stronger than the aforementioned result from \cite{O0}, since in \eqref{domination} we require only the domination of $p$-th norms of $dx$ over the $p$-th norm of $dy$ (instead of the estimates on ``conditional $p$-th moments'' as in \eqref{assumption_osekowski}). The argument used in the proof of our next result, Theorem \ref{positive_theorem}, will depend heavily on this weaker assumption (see \eqref{passage}).
\end{remark}

\begin{proof}[Proof of Theorem \ref{stronger_theorem}]
Set $z_N=\left(\sum_{k=0}^N |dy_k|^p\right)^{1/p}$. Then $(x_N,y,z_N)$ satisfy the good-$\lambda$ testing condition: this is almost word-by-word repetition of the arguments appearing in Subsections \ref{ABDG} and \ref{transforms} above (note that to check (i), we only need the martingale property of $x$, $y$ and the condition $\mathcal{E}_{n-1}(dy_n^2)\leq \mathcal{E}_{n-1}(dx_n^2)$ for each $n$). Consequently, by Theorem \ref{main-theorem}, we obtain
\begin{equation}\label{boundy}
  ||y_N||_{p}\leq \frac{12 p}{\left(1-\left(1+\frac{1}{p}\right)^{2-p}\right)^{1/2}}\left(||{x}_N||_{p}^2+||z_N||^2_{p}\right)^{1/2}.
 \end{equation}
It remains to observe that the assumption $\|dy_n\|_p\leq \|dx_n\|_p$, $n=0,\,1,\,2,\,\ldots$, combined with interpolating estimate \eqref{boundb}, implies
$$ ||z_N||_p=\left(\sum_{k=0}^N \tau(|dy_k|^p)\right)^{1/p}\leq \left(\sum_{k=0}^N \tau(|dx_k|^p)\right)^{1/p}\leq 2^{1-2/p}\|x_N\|_p.$$
Plugging this into \eqref{boundy} gives the desired inequality. The order $C_p=O(p)$ is already optimal for tangent martingales in the commutative setting, which can be extracted from the  examples of Burkholder \cite{B0}. Namely, for any $c<p-1$ there is an integer $N$ with the following property. If $\e_0$, $\e_1$, $\e_2$, $\ldots$, $\e_N$ are independent Rademacher variables, then there exists a sequence $(v_n)_{n=0}^N$ which is predictable with respect to the filtration generated by $(\e_n)_{n=0}^N$, such that
$$ \left\|\sum_{n=0}^N (-1)^n\e_n v_n\right\|_p>c\left\|\sum_{n=0}^N \e_nv_n\right\|_p.$$
It remains to note that the sequences $(\e_nv_n)_{n=0}^N$ and $((-1)^n\e_nv_n)_{n=0}^N$ are tangent.
\end{proof}

\subsection{Inequalities for sums of positive operators} The next statement compares the sizes of tangent sums of positive operators. The commutative version of this result was obtained by Hitczenko in \cite{Hi0}.

\begin{theorem}\label{positive_theorem}
Let $1\leq p<\infty$ and suppose that $u=(u_n)_{n\geq 0}$, $v=(v_n)_{n\geq 0}$ are tangent sequences of positive operators. Then
\begin{equation}\label{positive}
 \left\|\sum_{n=0}^N v_n\right\|_p\leq C_p\left\|\sum_{n=0}^N u_n\right\|_p
\end{equation}
for some constant $C_p$ of order $O(p)$ as $p\to \infty$. The order is optimal, as it is already the best in the classical case.
\end{theorem}
\begin{proof}
Let $\e_0$, $\e_1$, $\e_2$, $\ldots$ be independent Rademacher variables on some probability space $(\Omega,\F,\mathbb{P})$.
We will work with the extended algebra $\mathcal{N}=L^\infty(\Omega,\F,\mathbb{P})\overline{\otimes} \mathcal{M}$, equipped with the tensor trace and the filtration $(L^\infty(\Omega,\F_n,\mathbb{P})\overline\otimes \mathcal{M}_n)_{n\geq 0}$ (where $(\F_n)_{n\geq 0}$ is the filtration generated by the Rademacher sequence). We will consider the cases $1\leq p<2$ and $p\geq 2$ separately. In the first case, we consider the martingale difference sequences
$$ dx_n=\e_n\otimes u_n^{1/2},\quad dy_n=\e_n\otimes v_n^{1/2}$$
on the extended algebra. By Burkholder-Gundy inequality, we get
\begin{equation}\label{firs}
 \left\|\sum_{n=0}^N v_n\right\|_{L^p(\mathcal{M})}=\|S_N(y)\|_{L^{2p}(\mathcal{N})}^2\leq C_p'\|y_N\|_{L^{2p}(\mathcal{N})}^2
\end{equation}
for some constant $C_p'$.
Observe that for each $n$ we have, by tangency,
$$ \mathcal{E}_{n-1}^\mathcal{N}(dy_n^2)=1\otimes \mathcal{E}_{n-1}^\mathcal{M}(v_n)=1\otimes \mathcal{E}_{n-1}^\mathcal{M}(u_n)=\mathcal{E}_{n-1}^\mathcal{N}(dx_n^2)$$
and
$$
 \|dy_n\|_{L^{2p}(\mathcal{N})}=\|v_n\|_{L^p(\mathcal{M})}^{1/2}=\|u_n\|_{L^p(\mathcal{M})}^{1/2}=\|dx_n\|_{L^{2p}(\mathcal{N})}.
$$
Therefore, by Theorem \ref{stronger_theorem}, we get $\|y_N\|_{L^{2p}(\mathcal{N})}\leq C_p\|x_N\|_{L^{2p}(\mathcal{N})}$. Applying Burkholder-Gundy inequality again, we get
$$ \|x_N\|_{L^{2p}(\mathcal{N})}\leq C_p''\|S_N(x)\|_{L^{2p}(\mathcal{N})}=C_p''\left\|\sum_{n=0}^N u_n\right\|_{L^p(\mathcal{M})}^{1/2},$$
for some $C_p''$,
which combined with \eqref{firs} gives the claim. Clearly, this proof works for all $p$, but it yields too large constant when $p$ approaches infinity.
The following argument produces the correct order. Suppose that $p\geq 2$.
We start with the observation that by the triangle inequality,
$$ \left\|\sum_{n=0}^N v_n\right\|_p\leq \left\|\sum_{n=0}^N u_n\right\|_p+\left\|\sum_{n=0}^N (v_n-u_n)\right\|_p,$$
and hence it is enough to provide an appropriate upper bound for the second term on the right. Note that by tangency, $(v_n-u_n)_{n\geq 0}$ is a martingale difference sequence. Let $dx_n=2\e_n\otimes u_n$ and $dy_n=1\otimes (v_n-u_n)$, $n=0,\,1,\,2,\,\ldots$; these operators are martingale differences on $\mathcal{N}$. We have the estimate
$$ \mathcal{E}_{n-1}^\mathcal{M}(v_n-u_n)^2=\mathcal{E}_{n-1}^\mathcal{M}v_n^2-\mathcal{E}_{n-1}^\mathcal{M}v_nu_n-\mathcal{E}^\mathcal{M}_{n-1}u_nv_n+\mathcal{E}^\mathcal{M}_{n-1}u_n^2\leq 4\mathcal{E}^\mathcal{M}_{n-1}u_n^2,$$
since, by tangency, it is equivalent to $\mathcal{E}_{n-1}^\mathcal{M}(u_n+v_n)^2\geq 0$. Consequently,
$$ \mathcal{E}_{n-1}^\mathcal{N}(dy_n^2)=1\otimes \mathcal{E}^\mathcal{M}_{n-1}(v_n-u_n)^2\leq \mathcal{E}^\mathcal{N}_{n-1}({dx_n^2}).$$
Furthermore, exploiting the tangency and the triangle inequality again,
\begin{equation}\label{passage}
 \|{dy_n}\|_{L^p(\mathcal{N})}=\|v_n-u_n\|_{L^p(\mathcal{M})}\leq 2\|u_n\|_{L^p(\mathcal{M})}=\|{dx_n}\|_{L^p(\mathcal{N})}.
\end{equation}
Therefore Theorem \ref{stronger_theorem} yields
$$ \left\|\sum_{n=0}^N (v_n-u_n)\right\|_{L^p(\mathcal{M})}=\|{y_N}\|_{L^p(\mathcal{N})}\leq C_p\|x_N\|_{L^p(\mathcal{N})}=C_p\left\|\sum_{n=0}^N \e_n\otimes u_n\right\|_{L^p(\mathcal{N})}.$$
It remains to note that
$$ -\sum_{n=0}^N 1\otimes u_n\leq \sum_{n=0}^N \e_n\otimes u_n\leq \sum_{n=0}^N 1\otimes u_n,$$
which gives $\left\|\sum_{n=0}^N \e_n\otimes u_n\right\|_{L^p(\mathcal{N})}\leq \left\|\sum_{n=0}^N u_n\right\|_{L^p(\mathcal{M})}$. For the sharpness of the linear order in the classical setting, consult \cite{Hi0}.
\end{proof}

\begin{remark}\label{weak_remark}
A careful inspection shows that  the tangency assumption can be relaxed in the above theorem. Namely, if $p\geq 2$, then the inequality \eqref{positive} holds true if the sequence $(u_n)_{n\geq 0}$ consists of positive operators and $(v_n)_{n\geq 0}$ consist of self-adjoint operators satisfying $ \mathcal{E}_{n-1}(v_n)=\mathcal{E}_{n-1}(u_n)$ (which guarantees that $(v_n-u_n)_{n\geq 0}$ is a martingale difference sequence) and $\mathcal{E}_{n-1}(v_n^2)\leq \mathcal{E}_{n-1}(u_n^2)$,  $\|v_n\|_p\leq \kappa \|u_n\|_p$
for each $n=0,\,1,\,2,\,\ldots$. (Then the constant changes from $C_p$ to $C_p\kappa$).
\end{remark}

\subsection{A refinement of Doob's inequality} The purpose of this short subsection is to present the following striking version of noncommutative Doob's inequality.

\begin{theorem}\label{better_Doob}
Suppose that $(u_n)_{n=0}^N$ is an adapted sequence of positive operators. Then for $1\leq p<\infty$ we have
\begin{equation}\label{convexlemma}
 \left\|\sum_{n=0}^N \mathcal{E}_{n-1}(u_n)\right\|_p\leq c_p \left\|\sum_{n=0}^N u_n\right\|_p
\end{equation}
with $c_p$ of order $O(p)$ as $p\to \infty$. The order is optimal, as it is already the best possible in the commutative setting.
\end{theorem}
\begin{remark}
If we drop the assumption of the adaptedness of $(u_n)_{n=0}^N$, then the optimal order rises to $O(p^2)$: see \cite{J} and \cite{JX2}. It is also worth noting here that in the classical case, in both adapted and non-adapted settings, the sharp constant is $c_p=p$ (see Wang \cite{Wan}).
\end{remark}
\begin{proof}[Proof of Theorem \ref{better_Doob}]
If $1\leq p\leq 2$, then the estimate follows from the noncommutative analogue of Doob's inequality. Suppose that $p\geq 2$.
Consider the auxiliary sequence $v_n=2\mathcal{E}_{n-1}(u_n)-u_n$. It consists of self-adjoint operators. Furthermore, for any $n=0,\,1,\,2,\,\ldots,\,N$ we have
$ \mathcal{E}_{n-1}(v_n)=\mathcal{E}_{n-1}(u_n)$, $ \mathcal{E}_{n-1}(v_n^2)=\mathcal{E}_{n-1}(u_n^2)$ and
$$ \|v_n\|_p\leq 2\|\mathcal{E}_{n-1}(u_n)\|_p+\|u_n\|_p\leq 3\|u_n\|_p.$$
Consequently, by Remark \ref{weak_remark}, we get that
$$ \left\|\sum_{n=0}^N v_n\right\|_p\leq 3C_p\left\|\sum_{n=0}^N u_n\right\|_p$$
for some constant $C_p$ depending linearly on $p$ as $p\to \infty$. Consequently, by the triangle inequality,
$$ \left\|\sum_{n=0}^N \mathcal{E}_{n-1}(u_n)\right\|_p=\frac{1}{2}\left\|\sum_{n=0}^N (u_n+v_n)\right\|_p\leq \frac{1+3C_p}{2} \left\|\sum_{n=0}^N u_n\right\|_p.$$
The optimality of order in the classical case follows from the paper of Wang \cite{Wan} already mentioned above. This completes the proof.
\end{proof}

\subsection{Application to Schur multipliers} Now we turn our attention to classical objects in matrix theory, the so-called Schur multipliers. We start with recalling some basic definitions. The Schatten $p$-classes, denoted by $S^p$, are the noncommutative $L^p$-spaces associated with the von Neumann algebra $B(\ell^2)$ with the usual trace $\tau=\operatorname{Tr}$. We will mostly work with the finite-dimensional version of Schatten classes, denoted by $S^p_N$, which are equal to the $L^p$-spaces associated with $B(\ell^2_N)$. The elements of $S^p$ (respectively, $S^p_N$) can be represented as infinite (respectively, finite) matrices. Given $1\leq p\leq \infty$, an infinite matrix $m=(m_{ij})_{i,j\geq 1}$ is called the \emph{Schur multiplier} on $S^p$, if there is a finite constant $K_p$ depending only on $p$ such that
$$ \|m*a\|_p\leq K_p\|a\|_p$$
for all $a\in S^p$. Here the symbol `$*$' denotes the Schur (or Hadamard) mulitplication of matrices: $(m_{ij})_{i,j\geq 1}*(a_{ij})_{i,j\geq 1}=(m_{ij}a_{ij})_{i,j\geq 1}$. The norm of the Schur multiplier on $S^p$ will be denoted by $\|m\|_{S^p\to S^p}$. Similarly, for a given finite matrix $(m_{ij})_{i,j=1}^N$, the symbol $\|m\|_{S^p_N\to S^p_N}$ will stand for the norm of the operator $a\mapsto m*a$ acting on $S^p_N$.

The investigation of infinite matrices $m$ corresponding to Schur multipliers on $S^p$ has gained a lot of interest in literature. While the full characterization of such a class seems to be hopeless, much work has been done to construct examples and/or to provide such a characterization if $m$ enjoys some additional geometrical structure (e.g., $m$ is a Toeplitz or a Hankel matrix): see for example Aleksandrov and Peller\cite{AP}, Bennett \cite{Be}, Bourgain \cite{Bo} and Pisier \cite{P}.

We will show that our estimates for tangent martingales can be used to provide some information in this context as well. We will work under the assumption that the multipliers are zero-one matrices (in such a case, sometimes $m$ are referred to as Schur projections \cite{Do,DGi}).
We will need two simple yet crucial facts. First, it is easy to see that we have the following localization principle:
$$ \|m\|_{S^p\to S^p}=\lim_{N\to \infty} \|m^{(N)}\|_{S^p_N\to S^p_N},$$
where $m^{(N)}$ is the truncation of the matrix $m$ to the first $N$ columns and $N$ rows. The second observation is  the straightforward identity
$$ \tau((m*a)b)=\tau(a(m^**b)),$$
valid for all $a\in S^p$ and $b\in S^{p'}$, which in particular yields
\begin{equation}\label{duality}
\|m\|_{S^p\to S^p}=\|m\|_{S^{p'}\to S^{p'}}\quad \mbox{ for all $1<p<\infty$.}
\end{equation}

In our considerations below, we will require some basic facts about the upper triangular projection $T$. This operator acts on infinite matrices $a=(a_{ij})_{i,j\geq 1}$ by  the formula
$$ (Ta)_{i,j}=\begin{cases}
a_{ij} & \mbox{if }i\leq j,\\
0 & \mbox{otherwise.}
\end{cases}$$
It is well-known that for each $1<p<\infty$, $T$ is bounded on the Schatten $p$-class and we have $\|T\|_{S^p\to S^p}=O(p)$ as $p\to \infty$ and $\|T\|_{S^p\to S^p}=O((p-1)^{-1})$ as $p\to 1$ (see \cite{KP}).

The following statement is the main result of this subsection.

\begin{theorem}
Suppose that an infinite matrix $m$ is of the form
$$ \left[ \begin{array}{ccccc}
n_1 & m_2 & m_3 & m_4 &\ldots\\
m_2 & n_2 & m_3 & m_4 & \ldots\\
m_3 & m_3 & n_3 & m_4 & \ldots\\
m_4 & m_4 & m_4 & n_4 &\ldots\\
\ldots & \ldots & \ldots & \ldots & \ldots
\end{array}\right],$$
where $m_i,\,n_i\in \{0,1\}$.
Then  $m$ is a Schur multiplier for all $1<p<\infty$. Furthermore, for all $a\in S^p$ we have
$$\|m*a\|_{p}\leq C_p\|a\|_p,$$
where $C_p$ is of order $O(p)$ as $p\to \infty$ and $O((p-1)^{-1})$ as $p\to 1$. Both orders are optimal.
\end{theorem}
\begin{proof}
We start with some reductions.
By \eqref{duality}, it is enough to prove the claim for $p\geq 2$. Next, we may assume that the entries on the main diagonal of $m$ are equal to $1$: indeed, a general multiplier as in the statement can be decomposed into
$$ \left[ \begin{array}{ccccc}
1 & m_2 & m_3 & m_4 &\ldots\\
m_2 & 1 & m_3 & m_4 & \ldots\\
m_3 & m_3 & 1 & m_4 & \ldots\\
m_4 & m_4 & m_4 & 1 &\ldots\\
\ldots & \ldots & \ldots & \ldots & \ldots
\end{array}\right]
+ \left[ \begin{array}{ccccc}
n_1-1 &  &  &  &\\
 & n_2-1 &  &  & \\
 &  & n_3-1 &  & \\
 &  &  & n_4-1 &\\\
 &  &  &  & \ldots
\end{array}\right].$$
The latter diagonal matrix (denote it by $m_{d}$) satisfies $\|m_d\|_{S^2\to S^2}=\|m_d\|_{S^\infty\to S^\infty}=1$ and hence $\|m_d\|_{S^p\to S^p}\leq 1$ by interpolation. This shows that the first matrix in the above decomposition decides about the size of the multiplier $m$.
The final reduction is to exploit the above localization principle: we may restrict ourselves to matrices of finite dimension $N\times N$.

 Since an arbitrary operator $a\in S^p_N$ can be decomposed into its self-adjoint and skew-symmetric parts, whose $L^p$-norms do not exceed $\|a\|_p$, it suffices to prove that for any self-adjoint $a\in S^p_N$ we have
$$ \|m*a\|_p\lesssim p\|a\|_p.$$
Now we will use the notation introduced in Example \ref{Schur} above. Given a self-adjoint operator $a\in S^p_N$, let $(a_n)_{n=0}^N=(\mathcal{E}_na)_{n=0}^N$ be the associated martingale. Then we have $a_0=da_0=\tau(a)I$ and for any $k= 1,\,2,\,\ldots,\,N$, the difference $da_k$ is of the form
$$
    da_k=\left[\begin{array}{cccccc}
  & \lambda_{k,1} & & & &\\
       & \lambda_{k,2} & &  & &\\
       & \ldots & & & &\\
       & \lambda_{k,k-1} & & & &\\
      \overline{\lambda}_{k,1}\ \overline{\lambda}_{k,2}\ \ldots\ \overline{\lambda}_{k,k-1} & \alpha_k & & & &\\
       & & \beta_k & & &\\
       & & & \beta_k & &\\
       & & & & \ldots &\\
       & & & & & \beta_k
    \end{array}\right]
 $$
for some complex numbers $\lambda_{k,1}$, $\lambda_{k,2}$, $\ldots$, $\lambda_{k,k-1}$ and some real numbers $\alpha_k$, $\beta_k$ satisfying $\alpha_k+(N-k)\beta_k=0$. Since $\gamma_k:=-1+2m_k\in \{-1,1\}$, it follows from the reasoning in Example \ref{Schur} that the sequence $(db_k)_{k=0}^N$, given by $db_0=da_0$ and
$$
    db_k=\left[\begin{array}{cccccc}
  & \gamma_k\lambda_{k,1} & & & &\\
       & \gamma_k\lambda_{k,2} & &  & &\\
       & \ldots & & & &\\
       & \gamma_k\lambda_{k,k-1} & & & &\\
     \gamma_k\overline{\lambda}_{k,1}\ \gamma_k\overline{\lambda}_{k,2}\ \ldots\ \gamma_k\overline{\lambda}_{k,k-1} & \alpha_k & & & &\\
       & & \beta_k & & &\\
       & & & \beta_k & &\\
       & & & & \ldots &\\
       & & & & & \beta_k
    \end{array}\right]
 $$
for $k\geq 1$, is tangent to $(da_k)_{k=0}^N$ (and is obviously a martingale difference sequence). Therefore, Theorem \ref{stronger_theorem} implies that $ \|b_N\|_p\leq C_p\|a_N\|_p$
for some constant $C_p$ depending linearly on $p$ as $p\to \infty$, and hence
$$ \left\|m^{(N)}*a\right\|_p=\left\|\frac{a_N+b_N}{2}\right\|_p\leq \frac{1+C_p}{2}\|a_N\|_p.$$
This establishes the estimate $\|m\|_{S^p\to S^p}\leq O(p)$ as $p\to \infty$. To prove the reverse bound, we will use the properties of triangular projections. Consider the zero-one matrix
$$ m=\left[\begin{array}{ccccccc}
0 & 1 & 0 & 1 & 0 & 1 & \ldots \\
1 & 1 & 0 & 1 & 0 & 1 & \ldots\\
0 & 0 & 0 & 1 & 0 & 1 & \ldots\\
1 & 1 & 1 & 1 & 0 & 1 & \ldots\\
0 & 0 & 0 & 0 & 0 & 1 & \ldots\\
1 & 1 & 1 & 1 & 1 & 1 & \ldots\\
\ldots & \ldots &\ldots &\ldots &\ldots &\ldots & \ldots
\end{array}\right],$$
where zeros and ones run in a reversed $L$-pattern. We introduce the following transformation $t$ of the Schatten classes: given a matrix $A\in S^p$, we insert a row of zeros between any two consecutive rows of $A$ and a column of zeros between any two consecutive columns of $A$; furthermore, we insert a column of zeros in front the first column of $A$. Formally, if $A=(a_{ij})_{i,j\geq 1}$, then
$$ t(A)=\left[\begin{array}{ccccccc}
0 & a_{11} & 0 & a_{12} & 0 & a_{13} &\ldots\\
0 & 0 & 0 & 0 & 0 & 0 & \ldots\\
0 & a_{21} & 0 & a_{22} & 0 & a_{23} &\ldots\\
0 & 0 & 0 & 0 & 0 & 0 & \ldots\\
0 & a_{31} & 0 & a_{32} & 0 & a_{33} &\ldots\\
\ldots & \ldots & \ldots & \ldots & \ldots & \ldots & \ldots
\end{array}\right].$$
A crucial observation, which links $m$, $t$ and the triangular projection $T$, is the identity
$$ m*t(A)=t(T(A)).$$
Since adding a row/column of zeros does not change the norm, we have $\|t(A)\|_{S^p}=\|A\|_{S^p}$.
Consequently, we see that $\|m\|_{S^p\to S^p}\geq \|t\circ T\|_{S^p\to S^p}=\|T\|_{S^p\to S^p}=O(p)$ as $p\to \infty$. This completes the proof.
\end{proof}

\section{Applications to harmonic analysis}

The purpose of this final section is to provide some exemplary applications of good-$\lambda$ approach in noncommutative harmonic analysis. More precisely, we will first study $L^p$-estimates for differentially subordinate operators, which have their roots in the classical results on the boundedness of Hilbert transform on the real line (for the relevant definitions, see below); we also investigate $L^p$-estimate of the $j$-th Riesz transform on group von Neumann algebras. Both constants are of optimal order. Finally we will establish a certain square-function-type estimates in the context of contractive semigroups on semifinite von Neumann algebras.

We start with the necessary definitions and notation in the semigroup theory. The literature on the subject is extremely extensive, so we will only introduce some basic facts and notions, and refer the interested reader to \cite{JM,OP,Pet,P,Va} for the more detailed exposition. Throughout, we assume that $(T_t)_{t\geq 0}$ is a semigroup of completely positive maps on a semifinite von Neumann algebra $\mathcal{N}$, satisfying the following \emph{standard assumptions}.

\begin{itemize}
\item[(i)] Every $T_t$ is a unital normal positive map on $\mathcal{N}$.
\item[(ii)] Every $T_t$ is self-adjoint with respect to the trace: $\tau(T_t(x)y)=\tau(xT_t(y))$.
\item[(iii)] The family $(T_t)_{t\geq 0}$ is a strongly continuous semigroup on $L^p(\mathcal{N})$ for every $1\leq p<\infty$ with nonnegative generator $A$, i.e., $T_t=e^{-tA}$.
\item[(iv)] There exists a weakly dense self-adjoint subalgebra $\mathcal{A}\subset \mathcal{N}$ such that $T_t(\mathcal{A})\subset \mathcal{A}$ and $A(\mathcal{A})\subset \mathcal{A}$.
\end{itemize}

The first two conditions imply $\tau(T_t x) = \tau(x)$ for all $x$, so the operators $T_t$ are faithful and are contractive on $L^1(\mathcal{N})$. Furthermore, by interpolation, these operators extend to contractions on $
L^p(\mathcal{N})$, $1 \leq  p < \infty$, and satisfy $\lim_{t\to 0} T_t x = x$ in $L^p(\mathcal{N})$ for all $x \in L^p(\mathcal N)$.
For $1\leq p<\infty$, let $dom_p(A)$ be the class of all operators $x\in L^p(\mathcal{N})$ for which $\lim_{t\to \infty} t^{-1}(T_tx-x)$ exists in $L^p(\mathcal N)$. In the condition (iv) above we actually assume a little more, namely, that $\mathcal{A}\subset dom_p(A)$ for all $p$.  We introduce the gradient form
$$ 2\Gamma(x,y)=A(x^*)y+x^*A(y)-A(x^*y)$$
for $x,\,y\in \mathcal{A}$. The semigroup $(T_t)_{t\geq 0}$ is said to satisfy the condition $\Gamma^2\geq 0$, if for all $x\in \mathcal{A}$ and all $t\geq 0$ we have the inequality $\Gamma(T_tx,T_tx)\leq T_t\Gamma(x,x)$.

In the considerations below, we will also consider the subordinated
Poisson semigroup $(P_t)_{t\geq 0}$ defined by $P_t = \exp(-t A^{1/2})$. This is again a semigroup satisfying (i)-(iii) above; we will also assume that $P_t(\mathcal{A})\subseteq \mathcal{A}$. It is easy to check that $P_t$ can be explicitly expressed in terms of the operators $(T_s)_{s\geq 0}$ via the so-called subordination formula
\begin{equation}\label{subor}
 P_t=\frac{1}{2\sqrt{\pi}}\int_0^\infty t\exp\left(-\frac{t^2}{4u}\right)u^{-3/2}T_u\mbox{d}u.
\end{equation}

We will also assume the existence of the additional structure, the so-called Markov dilation, which will enable us to apply the probabilistic arguments and the good-$\lambda$ approach.

\begin{definition}
We say that a semigroup $(T_t)_{t\geq 0}$ on a von Neumann algebra $\mathcal{N}$ admits  a reverse Markov dilation, if there exists a larger von Neumann algebra $\mathcal{M}$ and a family $\pi_s:\mathcal{N}\to \mathcal{M}$ of trace preserving $^*$-homomorphisms such that
\begin{equation}\label{Mdil}
 \mathcal{E}_{[s}(\pi_t(x))=\pi_s(T_{s-t}x),
\end{equation}
for all $t<s$ and $x\in \mathcal{N}$, where $\mathcal{E}_{[s}$ is the conditional expectation with respect to $\mathcal{M}_{[s}$ and $\mathcal M_{[s}$ is the von Neumann algebra generated by  $\{\pi_r(x):$ $x\in\mathcal N$, $r\geq s\}$.
\end{definition}

In our considerations below we will often refer to noncommutative continuous-time martingales, whose definition carries over, with no essential change, from the discrete-time case presented in Section 2. We would like to emphasize that our approach will be to discretize these processes first and then apply the good-$\lambda$ method - hence we do not need to worry about the technical difficulties which do arise when one passes from the dicrete- to continuous-time processes. We would only like to introduce here the concept of vanishing $p$-th variation, which will be of importance for us later.

\begin{definition}
Let $x=(x_t)_{t\in [0,S]}$ be a continuous-time martingale on some von Neumann algebra $\mathcal{M}$, adapted to some filtration $(\mathcal{M}_t)_{t\in [0,S]}$. We say that the $p$-th variation of $x$ is zero, if for any refining sequence of partitions $(t_k^{(n)})_{k=0}^{N_n}$, $n=1,\,2,\,\ldots$ of the interval $[0,S]$ satisfying $\max_{1\leq k\leq N_{n}}|t_{k}^{(n)}-t_{k-1}^{(n)}|\xrightarrow{n\to\infty}0$ we have
$$ \lim_{n\to\infty}\sum_{k=1}^{N_n} \|x_{t_k}-x_{t_{k-1}}\|_{L^p(\mathcal{M})}^p=0.$$
\end{definition}

\begin{remark}In a sense, the vanishing $p$-th variation measures the regularity of paths; in the classical setting, any continuous-path martingale enjoys this property for any $p>2$. In the noncommutative setting, this condition is implied by the so-called almost uniform continuity: see e.g. \cite{JM,JM2} for details.
\end{remark}

\subsection{Differential subordination of operators}\label{diffo} Now we will introduce a certain domination relation for pairs of self-adjoint elements of a given von Neumann algebra, and prove that it implies an appropriate $L^p$ estimate. Let us start with the motivation coming from the classical harmonic analysis.
Consider the heat semigroup acting on $L^\infty(\R)$ by the convolution
$$ T_tf(x)=\int_\R f(y)\cdot \frac{1}{2\sqrt{\pi t}}\exp\left(-\frac{(x-y)^2}{4t}\right)\mbox{d}y,\qquad t>0.$$
This semigroup satisfies the standard assumptions and the associated Poisson semigroup is given by
$$ P_tf(x)=\int_\R f(y)\cdot \frac{1}{\pi}\frac{t}{t^2+(y-x)^2}\mbox{d}y,\qquad t>0.$$
Suppose further that $f:\R\to \R$ is a $C^\infty$ function and let $Hf$ be its nonperiodic Hilbert transform given by the Cauchy principal value
$$  {H} f(x)=\frac{1}{\pi}\mbox{p.v.}\int_\R\frac{f(y)}{x-y}\mbox{d}y. $$
Let $u(t,x)=P_tf$ and $v(t,x)=P_t(Hf)$ be the extensions of $f$ and $Hf$ to the halfplane $(0,\infty)\times \R$. It is well-known that $u$ and $v$ satisfy Cauchy-Riemann equations, which in particular implies the following domination between $u$ and $v$:
\begin{equation}\label{differential harmonic}
  \left(\frac{d v}{dt}\right)^2+\left(\frac{dv}{dx}\right)^2\leq \left(\frac{d u}{dt}\right)^2+\left(\frac{du}{dx}\right)^2
\end{equation}
  (actually, we even have equality here). Using martingale methods (cf. \cite{BW}) one can prove that this domination yields the $L^p$-boundedness of the Hilbert transform with the  constant $\cot(\pi/2p^*)$, where $p^*=\max\{p,p/(p-1)\}$ (this constant is the best possible, see \cite{Pi}). In other words, the domination condition \eqref{differential harmonic} on the half-plane enforces the appropriate boundary behavior of $u$ and $v$ (i.e., of $f=u|_\R$ and $Hf=v|_\R$).

There is a natural question about the noncommutative analogue of the above phenomenon. We start with the introduction of the appropriate version of \eqref{differential harmonic}.

\begin{definition}
Let$(T_t)_{t\geq 0}$ be a standard semigroup on some von Neumann algebra $\mathcal{N}$, with the associated gradient form $\Gamma$ and the subordinated Poisson semigroup $(P_t)_{t\geq 0}$ given by \eqref{subor}. Assume further that  $x$, $y$ are two self-adjoint elements of $\mathcal{N}$. We say that $y$ is differentially subordinate to $x$, if
\begin{equation}\label{Poisson_differential}
 \left(\frac{d}{dt}P_ty\right)^2+\Gamma(P_ty,P_ty)\leq \left(\frac{d}{dt}P_tx\right)^2+\Gamma(P_tx,P_tx)\qquad \mbox{for all }t>0.
\end{equation}
\end{definition}

Our primary goal is to show that the above domination implies the $L^p$ bound between $x$ and $y$; the good-$\lambda$ approach will allow us to obtain the corresponding constants of optimal order. To study this problem, we introduce an additional Brownian component into the picture, following the argumentation in \cite{JM,JM2}, which actually can be tracked back to the classical works of Meyer \cite{Mey}. Let $(T_t)_{t\geq 0}$ be a standard semigroup with generator $A$ admitting a reverse Markov dilation $(\pi_s:\mathcal{N}\to \mathcal{M})_{s\geq 0}$. Introduce a new generator
$$ \widehat{A}=-\frac{d^2}{dt^2}\otimes I+I\otimes A,$$
acting on (a weak dense self-adjoint subalgebra) of the von Neumann algebra $\widehat{\mathcal{N}}=L^\infty(\R) \overline \otimes \mathcal{N}$ equipped with the standard tensor trace. This generator leads to a new semigroup $\widehat{T}_t = \exp(-t \widehat{A})$ with the corresponding gradient form
$$ \widehat{\Gamma}(f(t),g(t))= \frac{df^*(t)}{dt}\frac{dg(t)}{dt}+\Gamma( f (t), g(t)).$$
Next, suppose that $(B_t)_{t\geq 0}$ is a one-dimensional Brownian motion on some probability space $(\Omega,\F,\mathbb{P})$, started at zero and run at the double speed (so that its generator is $d^2/dt^2$). We define the stopping times $\sigma_a=\inf\{t\geq 0:B_t=-a\}$ for any $a>0$.

Arguing as in  \cite[Lemma 2.5.3]{JM}, one can show the following.

\begin{lemma}\label{MM9}
 For any $x\in L^1(\mathcal{N})+\mathcal{N}$ and any $a,S>0$, the process
\begin{equation}\label{marting}
\widehat{n}_{a,S}(x)=\big(\pi_{S-t\wedge \sigma_a}(P_{a+B_{t\wedge \sigma_a}}x)\big)_{0\leq t\leq S}
\end{equation}
is a martingale on the von Neumann algebra $\widehat{\mathcal{M}}=L^\infty(\Omega)\overline\otimes \mathcal{M}$, adapted to the filtration $\big(\widehat{\mathcal{M}}_t\big)_{0\leq t\leq S}=\big(\sigma(B_s)_{s\leq t}\overline\otimes \mathcal{M}_{[S-t}\big)_{0\leq t\leq S}$. Furthermore, the process
$$ |\widehat{n}_{a,S}(x)_t|^2-2\int_{0}^{ t\wedge \sigma_a} \pi_{S-u}\widehat{\Gamma}(P_{a+B_u}x,P_{a+B_u}x)\mbox{d}u,\qquad 0\leq t\leq S,$$
is a martingale with respect to $(\widehat{\mathcal{M}}_t)_{0\leq t\leq S}$.
\end{lemma}

We will prove the following statement.

\begin{theorem} \label{opcor}
Let $2\leq p<\infty$. Suppose that $x$, $y$ are two self-adjoint elements of $L^p(\mathcal{N})$ such that $\lim_{t\to\infty} P_ty=0$ in $L^p$, $y$ is differentially subordinate to $x$ and for any $a,\,S>0$, the martingale $\widehat{n}_{a,S}(y)$ has vanishing $p$-th variation.
Then we have
$$\|y\|_{L^p(\mathcal{N})}\leq C_p\|x\|_{L^p(\mathcal{N})}$$
for some constant $C_p$ of order $O(p)$ as $p\to\infty$. The order is optimal.
\end{theorem}
\begin{proof} We will prove the claim for $p>2$ only; it will be clear how to modify the argument in the case $p=2$.
Pick $x$ and $y$ as in the statement. Let $a,\,S>0$ be arbitrary numbers and let $0=t_0<t_1<t_2<\ldots<t_N=S$ be a partition of $[0,S]$. Consider the discrete-time martingales
$$ X_k=\widehat{n}_{a,S}(x)_{t_k},\qquad Y_k=\widehat{n}_{a,S}(y)_{t_k}-\widehat{n}_{a,S}(y)_0,\qquad k=0,\,1,\,2,\,\ldots, N.$$
In addition, let $Z_N=\left(\sum_{n=0}^N |dY_n|^p\right)^{1/p}$ and let $(\mathcal{M}_k)_{k= 0}^N$ be the natural filtration generated by $(X_k)_{k=0}^N$ and $(Y_k)_{k=0}^N$. Then the triple $(X_N,Y,Z_N)$ satisfies the good-$\lambda$ testing conditions. Indeed, the second condition is evident; the first follows from $dX_0^2=(\pi_S(P_ax))^2\geq 0=dY_0^2$ and
\begin{align*}
 \mathcal{E}_{k-1}dX_k^2 &=2\mathcal{E}_{k-1}\left\{\int_{t_{k-1}\wedge\sigma_a}^{t_k\wedge \sigma_a} \pi_{S-u}\widehat{\Gamma}(P_{a+B_u}x,P_{a+B_u}x)\mbox{d}u\right\}\\
 &\geq 2\mathcal{E}_{k-1}\left\{\int_{t_{k-1}\wedge \sigma_a}^{t_k\wedge \sigma_a} \pi_{S-u}\widehat{\Gamma}(P_{a+B_u}y,P_{a+B_u}y)\mbox{d}u\right\}=\mathcal{E}_{k-1}dY_k^2
\end{align*}
for $k\geq 1$, where the first and the last equality is due to the second part of Lemma \ref{MM9}, while the middle inequality is due to the differential subordination (compare \eqref{Poisson_differential} with the definition of $\widehat{\Gamma}$).
Consequently, we get
$$ \|Y_N\|_{L^p(\widehat{\mathcal{M}})}\leq C_p\left(\|X_N\|^2_{L^p(\widehat{\mathcal{M}})}+\|Z_N\|^2_{L^p(\widehat{\mathcal{M}})}\right)^{1/2}.$$
Now, since the $p$-th variation of $\widehat{n}_{a,S}(y)$ vanishes, taking the limit with partition $\{t_n\}$ yields
$$ \|\pi_{S-S\wedge \sigma_a}(P_{a+B_{S\wedge \sigma_a}}y)-\pi_S(P_ay)\|_{L^p(\widehat{\mathcal{M}})}\leq C_p\|\pi_{S-S\wedge \sigma_a}(x)\|_{L^p(\widehat{\mathcal{M}})}.$$
Observe that the right-hand side is equal to $C_p\|x\|_{L^p(\mathcal{N})}$, while the left-hand side is not smaller than
$$ \|\pi_{S-S\wedge \sigma_a}(P_{a+B_{S\wedge \sigma_a}}y)\|_{L^p(\widehat{\mathcal{M}})}-\|\pi_S(P_ay)\|_{L^p(\widehat{\mathcal{M}})}=\|P_{a+B_{S\wedge \sigma_a}}y\|_{L^p(\widehat{\mathcal{N}})}-\|P_ay\|_{L^p(\mathcal{N})}.$$
If we let $S\to \infty$, the right-hand side tends to $\|y\|_{L^p(\mathcal{N})}-\|P_ay\|_{L^p(\mathcal{N})}$.
Combining the above estimates, letting $a\to \infty$ and using the assumption $P_ay\to 0$, we get the desired estimate. The optimality of the linear order of the constant follows from the context of classical Hilbert transform on $\R$ presented above (it is easy to check that all the assumptions on the semigroups $(T_t)_{t\geq 0}$, $(P_t)_{t\geq 0}$ are satisfied).
\end{proof}

We would like to mention that there is an alternative domination expressed solely in the language of the semigroup $(T_t)_{t\geq 0}$, which also implies the corresponding $L^p$ bound. The definition is as follows.

\begin{definition}
Suppose that $(T_t)_{t\geq 0}$ is a semigroup on a given von Neumann algebra $\mathcal{N}$ satisfying the standard assumptions, and let $x,\,y\in \mathcal{N}$ be two self-adjoint operators. We say that $x$ dominates $y$, if for any $t> 0$ we have
$$ \Gamma(T_ty,T_ty)\leq \Gamma(T_tx,T_tx).$$
\end{definition}

We will prove the following analogue of Theorem \ref{opcor}.

\begin{theorem}\label{opth}
Let $(T_t)_{t\geq 0}$ be a semigroup of positive contractions on a given semifinite von Neumann algebra, satisfying the standard assumptions and admitting a reverse Markov dilation.
Let $2\leq  p<\infty$ and assume that $x$, $y$ are two self-adjoint elements of $L^p(\mathcal{N})$ such that $x$ dominates $y$, $\lim_{t\to\infty}T_ty=0$ in $L_p(\mathcal N)$ and the martingale $(\pi_t(T_ty))_{0\leq t\leq S}$ has vanishing $p$-th variation for each $S>0$. Then we have
\begin{equation}\label{Prin}
 \|y\|_{L^p(\mathcal{N})}\leq C_p\|x\|_{L^p(\mathcal{N})},
\end{equation}
where $C_p$ is of order $O(p)$ as $p \to \infty$.
\end{theorem}
\begin{proof}
Again, we focus on the case $p>2$. Let $S>0$ be a fixed number and let $0=t_0<t_1<t_2<\ldots<t_N=S$ be an arbitrary partition of $[0,S]$. Consider the martingales
$$ X_n=\pi_{t_{N-n}}(T_{t_{N-n}}x),\qquad Y_n=\pi_{t_{N-n}}(T_{t_{N-n}}y)-\pi_S(T_Sy)$$
for $n=0,\,1,\,2,\,\ldots,\,N$, which are both adapted to the filtration $(\mathcal{M}_n)_{n=0}^N=(\mathcal{M}_{[t_{N-n}})_{n=0}^N$ (note that $Y_0=0$). Setting
$$ Z_N=\left(\sum_{n=0}^N |dY_n|^p\right)^{1/p},$$
we check as previously that the triple $(X_N,Y,Z_N)$ satisfies the good-$\lambda$ testing conditions and hence
$$ \|Y_N\|_{L^p(\mathcal{M})}\leq C_p\left(\|X_N\|_{L^p(\mathcal{M})}^2+\|Z_N\|_{L^p(\mathcal{M})}^2\right)^{1/2},$$
where $C_p$ is the constant of Theorem \ref{main-theorem}. Passing to the limit with the partition, we get
$$ \|\pi_0(y)-\pi_S(T_Sy)\|_{L^p(\mathcal{M})}\leq C_p \|\pi_0(x)\|_{L^p(\mathcal{M})}= C_p\|x\|_{L^p(\mathcal{N})}.$$
It remains to let $S\to \infty$ to get the claim, since then $T_Sy\to 0$.
\end{proof}

\subsection{Riesz transforms on group von Neumann algebras}
Our second application concerns sharp $L^p$ bounds for noncommutative Riesz transforms associated with conditionally negative length functions on group von Neumann algebras. Recall that classical Riesz transforms \cite{St} in $\R^d$ are the operators $R_j=\partial_j (-\Delta)^{-1/2}$, $j=1,\,2,\,\ldots,\,d$, where $\Delta$ is the usual Laplacian. These objects are higher-dimensional analogues of the nonperiodic Hilbert transform, and their $L^p$-boundedness is of fundamental importance to harmonic analysis.

Inspired by the recent paper \cite{JMP}, we will apply the good-$\lambda$ approach to establish $L^p$-bound for the $j$-th Riesz transforms arising in the context of group von Neumann algebras. Let $G$ be a discrete group with left regular representation $\lambda:G\to \mathcal{B}(\ell_2(G))$ given by $\lambda(g)\delta_h=\delta_{gh}$. Here, as usual, $(\delta_g)_{g\in G}$ refers to the unit vector basis of $\ell_2(G)$. Then $\mathcal{L}(G)$, the associated group von Neumann algebra, is the weak operator closure of the linear span of $\lambda(G)$ in $\mathcal{B}(\ell_2(G))$. We equip this algebra with the standard trace $\tau$ uniquely determined by the equalities $\tau(\lambda(g))=1$ if $g=e$ and $\tau(\lambda(g))=0$ if $g\neq e$, where $e$ is the identity of $G$. Any element $f$ of $\mathcal{L}(G)$ admits the Fourier expansion
$$ f=\sum_{g\in G} \hat{f}(g)\lambda(g),$$
with $\tau(f)=\hat{f}(e)$. Consider the semigroup $T_\psi=(T_{\psi,t})_{t\geq 0}$ of operators on $\mathcal{L}(G)$, whose action is determined by the requirement
$$ T_{\psi,t}\lambda(g)=e^{-t\psi(g)}\lambda(g),\qquad t\geq 0,\,g\in G,$$
for some function $\psi:G\to \R$. Here we assume that $\psi$ is \emph{a conditionally negative length}, which amounts to saying that $\psi$ is real-valued, satisfies $\psi(e)=0$, $\psi(g)=\psi(g^{-1})$ for all $g\in G$ and also enjoys the inequality $\sum_{g,h\in G} \overline{a_g}a_h\psi(g^{-1}h)\leq 0$ for all sequences $(a_g)_{g\in G}$ of complex numbers which sum up to $0$. Then, by Schoenberg's theorem, $T_{\psi}$ satisfies the standard assumptions. Furthermore, it follows from the results of Ricard \cite{Ric} that the semigroup admits the reverse Markov dilation $\pi=(\pi_s:\mathcal{L}(G)\to \mathcal{M})_{s\geq 0}$. Let $A_\psi$ denote the generator of $T_\psi$. It is known that this operator acts via multiplication: $A_\psi\lambda(g)=\psi(g)\lambda(g)$. To define the associated Riesz transforms, we need to introduce an appropriate differential structure linked with $\psi$. Namely, conditionally negative lengths correspond to the affine representations $(\mathcal{H}_\psi,\alpha_\psi,b_\psi)$, where $\alpha_\psi:G\to O(\mathcal{H}_\psi)$ is an orthogonal representation over a real Hilbert space $\mathcal{H}_\psi$ and $b_\psi:G\to \mathcal{H}_\psi$ is a mapping satisfying the cocycle law
$$ b_\psi(gh)=\alpha_{\psi,g}(b_\psi(h))+b_\psi(g).$$
Motivated by the equality $\partial_j \exp(2\pi i\langle x,\cdot\rangle)=2\pi i x_j\exp(2\pi i\langle x,\cdot\rangle)$, the $j$-th Riesz transform was defined in \cite{JMP} by setting
$$ R_{\psi,j}f=\partial_{\psi,j}A_{\psi}^{-1/2}f=2\pi i \sum_{g\in G} \frac{\langle b_\psi(g),e_j\rangle_{\mathcal{H}_\psi}}{\sqrt{\psi(g)}}\hat{f}(g)\lambda(g),$$
where $(e_j)_{j\geq 1}$ is a certain fixed orthonormal basis of $\mathcal{H}_\psi$.
Our primary goal is to establish the $L^p$ bounds $(p\geq 2)$ for these objects, with the optimal order of the constant as $p\to \infty$. To this end, we will need a certain stochastic representation of Riesz transforms. To shorten and simplify the notation, in what follows, we will write $T$, $P$ instead of $T_\psi$ and $P_\psi$. Recall the martingales introduced in \eqref{marting} above.

\begin{theorem}\label{rpr}
For any $f,\,\varphi\in \mathcal{L}(G)$ with finite Fourier expansion we have the representation
\begin{equation}\label{identitys}
 \begin{split}
&\tau\left(\varphi R_{\psi,j}f\right)\\
&=-\lim_{a\to\infty}\lim_{S\to \infty}\lim_{K\to \infty}2\tau\left(\widehat{n}_{a,S}(\varphi)_S \sum_{m=0}^\infty \pi_{S-t_m^K}\left(\partial_{\psi,j}P_{a+B_{t_m^K}}f\right)(B_{t_{m+1}^K}-B_{t_m^K})\right),
\end{split}
\end{equation}
where $t_m^K=(m\cdot 2^{-K})\wedge S\wedge\sigma_a$.
\end{theorem}
\begin{proof}
It is enough to check the identity for $f=\lambda(g)$ and $\varphi=\lambda(h)$, by the bilinearity of both sides with respect to $f$ and $\varphi$. For such a choice of $f$ and $\varphi$,  it suffices to show the equality for $g=h^{-1}$, since otherwise both sides are zero. Directly from the definition of $R_{\psi,j}$, we compute that
$$ \tau\left(\varphi R_{\psi,j}f\right)=\frac{2\pi i\langle b_\psi(g),e_j\rangle_{\mathcal{H}_\psi}}{\sqrt{\psi(g)}}.$$
To study the right-hand side of \eqref{identitys}, pick arbitrary numbers $0\leq s<t\leq S$ and observe that
\begin{align*}
&\tau\bigg\{\bigg(\pi_{S-t\wedge \sigma_a}(P_{a+B_{t\wedge \sigma_a}}\varphi)-\pi_{S-s\wedge \sigma_a}(P_{a+B_{s\wedge \sigma_a}}\varphi\bigg)\\
&\qquad \qquad \qquad \qquad  \qquad \bigg(\pi_{S-s\wedge \sigma_a}(\partial_{\psi,j}P_{a+B_{s\wedge \sigma_a}}f)(B_{t\wedge \sigma_a}-B_{s\wedge \sigma_a})\bigg)\bigg\}\\
&=\tau\bigg\{\pi_{S-s\wedge \sigma_a}\bigg[\big(T_{t\wedge \sigma_a-s\wedge \sigma_a}P_{a+B_{t\wedge \sigma_a}}\varphi-P_{a+B_{s\wedge \sigma_a}}\varphi\big)\\
&\qquad \qquad \qquad \qquad \qquad
(\partial_{\psi,j}P_{a+B_{s\wedge \sigma_a}}f)(B_{t\wedge \sigma_a}-B_{s\wedge \sigma_a})\bigg]\bigg\}\\
&=\E \bigg\{\left(e^{-(t\wedge \sigma_a-s\wedge \sigma_a)\psi(g)}e^{-(a+B_{t\wedge \sigma_a})\sqrt{\psi(g)}}-e^{-(a+B_{s\wedge \sigma_a})\sqrt{\psi(g)}}\right)\\
&\qquad \qquad \qquad \qquad  \qquad \left(2\pi i\langle b_\psi(g),e_j\rangle_{\mathcal{H}_\psi}e^{-(a+B_{s\wedge \sigma_a})\sqrt{\psi(g)}}(B_{t\wedge \sigma_a}-B_{s\wedge \sigma_a})\right)\bigg\}.
\end{align*}
This, by standard stochastic calculus, is equal to
\begin{align*}
 -2\pi i\langle b_\psi(g),e_j&\rangle_{\mathcal{H}_\psi}\sqrt{\psi(g)}\E \int_{s\wedge \sigma_a}^{t\wedge \sigma_a} e^{-(u-s\wedge \sigma_a)\psi(g)-(a+B_u)\sqrt{\psi(g)}-(a+B_{s\wedge \sigma_a})\sqrt{\psi(g)}}2du\\
&=-2\pi i\langle b_\psi(g),e_j\rangle_{\mathcal{H}_\psi}\sqrt{\psi(g)}\bigg\{\E \int_{s\wedge \sigma_a}^{t\wedge \sigma_a} e^{-2(a+B_u)\sqrt{\psi(g)}}2\mbox{d}u+o(t-s)\bigg\}.
\end{align*}
Therefore,
\begin{equation}\label{intermee}
\begin{split}
&\lim_{S\to \infty}\lim_{K\to \infty}\tau\left(\widehat{n}_{a,S}(\varphi)_S \sum_{m=0}^\infty \pi_{S-t_m^K}\left(\partial_{\psi,j}P_{a+B_{t_m^K}}f\right)(B_{t_{m+1}^K}-B_{t_m^K})\right)\\
&\qquad =-\lim_{S\to \infty} 2\pi i\langle b_\psi(g),e_j\rangle_{\mathcal{H}_\psi}\sqrt{\psi(g)}\E\int_0^{S\wedge \sigma_a} e^{-2(a+B_u)\sqrt{\psi(g)}}2\mbox{d}u\\
&\qquad = -2\pi i\langle b_\psi(g),e_j\rangle_{\mathcal{H}_\psi}\sqrt{\psi(g)}\E\int_0^{\sigma_a} e^{-2(a+B_u)\sqrt{\psi(g)}}2\mbox{d}u.
\end{split}\end{equation}
Now we will compute expectation of the latter integral,
using some elementary properties of Brownian motion and its maximal function $B^*=(B^*_t)_{t\geq 0}=(\max_{s\leq t}B_s)_{t\geq 0}$. It is well-known (see \cite[p.110]{RY}) that for any $u\geq 0$ the density of $(B_u,B^*_u)$ equals
$$ g_u(\alpha,\beta)=\frac{1}{2\sqrt{\pi}}u^{-3/2}(2\beta-\alpha)\exp\left(-\frac{(2\beta-\alpha)^2}{4u}\right)1_{\{ \alpha\leq \beta,\,\beta\geq 0\}}.$$
Replacing $(B_u,B_u^*)$ by $(-B_u,(-B_u)^*)$, we see that the last expectation in \eqref{intermee} is
\begin{align*}
 &\E\int_0^{\infty} e^{(-2a-2B_u)\sqrt{\psi(g)}}1_{\{(-B_u)^*\leq a\}}2\mbox{d}u\\
 &=\int_0^a \int_0^\beta \int_0^\infty e^{(-2a+2\alpha)\sqrt{\psi(g)}}\frac{1}{2\sqrt{\pi}}u^{-3/2}(2\beta-\alpha)\exp\left(-\frac{(2\beta-\alpha)^2}{4u}\right)2\mbox{d}u\mbox{d}\alpha\mbox{d}\beta\\
 &=\int_0^a \int_0^\beta e^{(-2a+2\alpha)\sqrt{\psi(g)}}2\mbox{d}\alpha\mbox{d}\beta\\
 &=\frac{1-e^{-2a\sqrt{\psi(g)}}}{2\psi(g)}-\frac{ae^{-2a\sqrt{\psi(g)}}}{\sqrt{\psi(g)}}.
\end{align*}
Therefore, letting $a\to \infty$ in \eqref{intermee} we obtain the desired assertion.
\end{proof}

\begin{theorem}
For any $2\leq p<\infty$ and any $j$ we have
$$ \|R_{\psi,j}\|_{\mathcal{L}^p(G)\to \mathcal{L}^p(G)}\leq c_p,$$
where $c_p=O(p)$ as $p\to \infty$. The order is the best possible: in the classical setting the norm is equal to $\cot(\pi/2p)$.
\end{theorem}
\begin{proof} For $p=2$ there is nothing to prove, so from now on we assume $p>2$.
Take an arbitrary $f$ with a finite expansion. Fix $a,\,S,\,K$ and consider the martingales $x=x^{a,S,K}$ and $y=y^{a,S,K}$ on $\widehat{\mathcal{M}}=L^\infty(\Omega)\otimes \mathcal{M}$ (recall that $\mathcal{M}$ is the target space of the dilation $\pi$) given by
$$ x_k=\widehat{n}_{a,S}(f)_{t_k^K}, \qquad y_k=\sum_{m=0}^{k-1} \pi_{S-t_m^K}\left(\partial_{\psi,j}P_{a+B_{t_m^K}}f\right)(B_{t_{m+1}^K}-B_{t_m^K}),$$
for $k=0,\,1,\,2,\,\ldots$ (with respect to their natural filtration).
Both sequences stabilize after $N$ steps, where $N$ is an arbitrary number  bigger than $2^K S$. As usual, put
$$ z_N=z_N^{a,S,K}=\left(\sum_{k=0}^N |dy_k|^p\right)^{1/p}.$$
Then $(\kappa_{K}x_N,y,z_N)$ satisfy the good-$\lambda$ testing conditions, where $\kappa_{K}>1$ is a certain function depending on $K$ (and $f$, but we keep the function fixed), converging to $1$ as $K\to \infty$. Indeed, the second is evident, while the first is due to
\begin{align*}
 \mathcal{E}_{k-1}(dx_k^2)&=2\mathcal{E}_{k-1}\int_{t_{k-1}^K}^{t_k^K}\pi_{S-u}\widehat{\Gamma}(P_{a+B_u}f,P_{a+B_u}f)\mbox{d}u\\
 &\geq \mathcal{E}_{k-1}\int_{t_{k-1}^K}^{t_k^K}\pi_{S-u}\left(\partial_{\psi,j}P_{a+B_{u}}f\right)^2\mbox{d}u\\
 &\geq  \mathcal{E}_{k-1}\left\{\kappa_K^{-2}\pi_{S-t_{k-1}^K}\left(\partial_{\psi,j}P_{a+B_{t_{k-1}^K}}f\right)^2(t_k^K-t_{k-1}^K)\right\}=\kappa_K^{-2}\mathcal{E}_{k-1}(dy_k^2).
\end{align*}
Here in the last inequality we have used the fact that the expansion of $f$ is finite. Consequently, using good-$\lambda$ approach, we obtain that
$$ \|y_N^{a,S,K}\|_{L^p(\widehat{\mathcal{M}})}\leq C_p\left(\kappa_K^2\|x_N^{a,S,K}\|_{L^p(\widehat{\mathcal{M}})}^2+\|z_N^{a,S,K}\|_{L^p(\widehat{\mathcal{M}})}^2\right)^{1/2}$$
for some constant $C_p$ of the linear order as $p\to \infty$. Note that
$$ \|x_N^{a,S,K}\|_{L^p(\widehat{\mathcal{M}})}=\|\pi_{S-S\wedge \sigma_a}(P_{a+B_{S\wedge \sigma_a}}f)\|_{L^p(\widehat{\mathcal{M}})}=\|P_{a+B_{S\wedge \sigma_a}}f\|_{L^p(\widehat{\mathcal{M}})},$$
which converges to $\|f\|_{\mathcal{L}^p(G)}$ as $S\to \infty$.
Now let us go back to Theorem \ref{rpr}. The above estimates imply that for any $\varphi\in \mathcal{L}^{p'}(G)$ of finite Fourier expansion we have
\begin{align*}
 &\tau\left(-\widehat{n}_{a,S}(\varphi)_S \sum_{m=0}^\infty \pi_{S-t_m^K}\left(\partial_{\psi,j}P_{a+B_{t_m^K}}f\right)(B_{t_{m+1}^K}-B_{t_m^K})\right)\\
 &\leq \|\widehat{n}_{a,S}(\varphi)_S\|_{L^{p'}(\widehat{\mathcal{M}})}\|y_N^{a,S,K}\|_{L^p(\widehat{\mathcal{M}})}\\
&\leq C_p\|\varphi\|_{\mathcal{L}^{p'}(G)}\left(\kappa_K^2\|f\|_{{\mathcal{L}^p(G)}}^2+\|z_N^{a,S,K}\|_{L^p(\widehat{\mathcal{M}})}^2\right)^{1/2}.
\end{align*}
However, $\|z_N^{a,S,K}\|_{L^p(\widehat{\mathcal{M}})}\to 0$ as $K \to \infty$, which follows from the (classical) continuity of Brownian paths. Consequently, by Theorem \ref{rpr}, we obtain
$$ \tau(\varphi R_{\psi,j}f)\leq 2C_p\|\varphi\|_{\mathcal{L}^{p'}(G)}\|f\|_{\mathcal{L}^p(G)}.$$
This yields the claim.
\end{proof}

\subsection{Square-function estimates} Our final application is an $L^p$ bound between a self-adjoint element $a\in\mathcal{A}$ and the associated square-function expressed in terms of the associated gradient form. Before we formulate the result, let us introduce a certain requirement closely related to the notion of standard dilation. Suppose that $(T_t)_{t\geq 0}$ is a standard semigroup admitting a reverse Markov dilation $(\pi_t)_{t\geq 0}$. At some places below we will need to assume that there is a von Neumann algebra $\mathcal{M}_{0]}$ with the associated conditional expectation $\mathcal{E}_{0]}$ such that
\begin{equation}\label{Mdil2}
\mathcal{E}_{0]}(\pi_t(x))=\pi_0(T_t(x))\qquad \mbox{for }t>0\mbox{ and }x\in \mathcal{N}.
\end{equation}

We will establish the following fact.

\begin{theorem}\label{semigr1}
Let $2\leq p<\infty$ and assume that  $(T_t)_{t\geq 0}$ is a semigroup satisfying the standard assumptions, the condition $\Gamma^2\geq 0$ and admitting a reverse Markov dilation. Suppose further that the semigroup has the property that for each self-adjoint operator $x$ and any $S>0$, the reverse-time martingale $(\mathcal{E}_{[t}x)_{0\leq t\leq S}=(\pi_t(T_tx))_{0\leq t\leq S}$ has vanishing $p$-th variation. Let $a\in \mathcal{A}$ be a self-adjoint operator.

(i) If $\lim_{t\to\infty} T_ta=0$ in $L_p(\mathcal N)$, then
\begin{equation}\label{mainsemigr}
 \|a\|_{L^p(\mathcal{N})}\leq C_p\left\|\left(\int_0^\infty \Gamma(T_sa,T_sa)\mbox{d}s\right)^{1/2}\right\|_{L^p(\mathcal{N})}
 \end{equation}
for some constant $C_p$ of order $O(p)$ as $p\to \infty$.

(ii) If \eqref{Mdil2} holds, then
\begin{equation}\label{mainsemigr22}
\left\|\left(\int_0^\infty \Gamma(T_sa,T_sa)\mbox{d}s\right)^{1/2}\right\|_{L^p(\mathcal{N})}\leq C_p\|a\|_{L^p(\mathcal{N})},
 \end{equation}
for some constant $C_p$ of order $O(p)$ as $p\to \infty$.
\end{theorem}

\begin{remark}
The above inequalities were proved in \cite[Theorem 2.4.10]{JM} with constants of worse order; the bound \eqref{mainsemigr} was obtained by martingale methods, while the reverse estimate was established with the use of $H^\infty$ calculus. Our approach to both inequalities exploits the good-$\lambda$ technique, which allows the improvement of the orders of constants to be linear. Roughly speaking, the proof of the left inequality presented in \cite{JM} rests on using twice the dual version of Doob's inequality; the advantage of the good-$\lambda$ method is that it enables to combine this double application into a single step.
\end{remark}

\noindent{\textbf{Part I. On the inequality \eqref{mainsemigr}.}} We will study the estimate for $p>2$ only (for $p=2$ the modification is straightforward). Let $S$ be a fixed positive number and let $\{0=s_0<s_1<\ldots<s_M=S\}$ be an arbitrary partition of the interval $[0,S]$. The noncommutative process $(u_s)_{s\in [0,S]}$, given for $s\in [s_j,s_{j+1}]$ by the formula
$$ u_s=\pi_s\big(T_{s+s_{j+1}}a\big)-\pi_{s_{j+1}}\big(T_{2s_{j+1}}a\big)+\sum_{k=j+1}^{M-1}\bigg(\pi_{s_k}\big(T_{s_k+s_{k+1}}a\big)-\pi_{s_{k+1}}\big(T_{2s_{k+1}}a\big)\bigg),$$
is a reverse-time martingale adapted to the filtration $(\mathcal{M}_{[s})_{s\in [0,S]}$. Indeed, the martingale property
$$ \mathcal{E}_{[s}u_t=u_s\qquad \mbox{for all }t<s$$
follows easily from \eqref{Mdil} (it is enough to check it for both $s,\,t$ belonging to some interval $[s_j,s_{j+1}]$).  Directly from our assumption on the semigroup, the $p$-variation of $u$ on $[s_j,s_{j+1}]$ vanishes, and hence so does the full $p$-th variation (i.e., on the whole interval $[0,S]$).
Let $\{0=t_0<t_1<t_2<\ldots<t_N=S\}$ be  a refinement of $\{s_j\}$. Introduce the associated discrete-time martingale $y=(y_n)_{n=0}^N$ given by $y_n=u_{t_{N-n}}$, $n=0,\,1,\,2,\,\ldots,\,N$. This sequence starts from $0$ and is adapted to the filtration  $(\mathcal{M}_n)_{n=0}^N=(\mathcal{M}_{[t_{N-n}})_{n=0}^N$. Furthermore, consider the operators
$$ x_N=\pi_0\left(\sum_{j=0}^{M-1}2(s_{j+1}-s_j)\Gamma(T_{s_{j+1}}a,T_{s_{j+1}}a)\right)^{1/2},\qquad z_N=\left(\sum_{k=0}^N |dy_k|^p\right)^{1/p}.$$

\begin{lemma}\label{gc1}
The triple $(x_N,y,z_N)$ satisfies the good-$\lambda$ testing conditions.
\end{lemma}
\begin{proof}
 The second requirement is obviously satisfied (we have $dy_k^2\leq z_N^2$ for each $k$), so we focus on the first condition. If $n<k\leq N$, then we have
\begin{equation}\label{chainsem}
\begin{split}
&\tau((R_{n-1}-R_n)dy_kR_{n-1}dy_k(R_{n-1}-R_n))\\
&\leq \tau((R_{n-1}-R_{n})dy_k^2)\\
&=\tau\left[(R_{n-1}-R_n)\bigg(\pi_{t_{N-k}}\big(T_{t_{N-k}+s_{j+1}}a\big)-\pi_{t_{N-k+1}}\big(T_{t_{N-k+1}+s_{j+1}}a\big)\bigg)^2\right],
\end{split}
\end{equation}
where $s_{j+1}$ is the unique element of the partition $\{s_m\}$ such that both points $t_{N-k},t_{N-k+1}$ belong  to $[s_j,s_{j+1}]$. Let us insert the conditional expectation with respect to $\mathcal{M}_{k-1}=\mathcal{M}_{[t_{N-k+1}}$ under the latter trace. The projection $R_n-R_{n-1}$ belongs to $\mathcal{M}_{k-1}$ and
\begin{align*}
 &\mathcal{E}_{[t_{N-k+1}}\bigg(\pi_{t_{N-k}}\big(T_{t_{N-k}+s_{j+1}}a\big)-\pi_{t_{N-k+1}}\big(T_{t_{N-k+1}+s_{j+1}}a\big)\bigg)^2\\
 &=\mathcal{E}_{[t_{N-k+1}}\bigg(\pi_{t_{N-k}}\big(T_{t_{N-k}+s_{j+1}}a\big)\bigg)^2-\bigg(\pi_{t_{N-k+1}}\big(T_{t_{N-k+1}+s_{j+1}}a\big)\bigg)^2\\
 &=\mathcal{E}_{[t_{N-k+1}}\pi_{t_{N-k}}\big(T_{t_{N-k}+s_{j+1}}a\big)^2-\pi_{t_{N-k+1}}\big(T_{t_{N-k+1}+s_{j+1}}a\big)^2\\
 &=\pi_{t_{N-k+1}}\bigg(T_{t_{N-k+1}-t_{N-k}}\big(T_{t_{N-k}+s_{j+1}}a\big)^2-\big(T_{t_{N-k+1}+s_{j+1}}a\big)^2\bigg).
\end{align*}
For any $t>0$, the function $f(r)=T_{t-r}((T_ra)^2)$ is differentiable and we have
$$ f'(r)=T_{t-r}A(T_ra)^2-T_{t-r}((AT_ra)T_ra)-T_{t-r}(T_ra(AT_ra))=-2T_{t-r}\Gamma(T_ra,T_ra).$$
Consequently, we see that for any $s<t$,
$$ T_{t-s}((T_sa)^2)-(T_ta)^2=\int_s^t 2T_{t-r}\Gamma(T_ra,T_ra)\mbox{d}r$$
and hence
\begin{align*}
 &\pi_{t_{N-k+1}}\bigg(T_{t_{N-k+1}-t_{N-k}}\big(T_{t_{N-k}+s_{j+1}}a\big)^2-\big(T_{t_{N-k+1}+s_{j+1}}a\big)^2\bigg)\\
 &=\pi_{t_{N-k+1}}\bigg(T_{t_{N-k+1}-t_{N-k}}\big(T_{t_{N-k}}T_{s_{j+1}}a\big)^2-\big(T_{t_{N-k+1}}T_{s_{j+1}}a\big)^2\bigg)\\
 &=\pi_{t_{N-k+1}}\int_{t_{N-k}}^{t_{N-k+1}} 2T_{t_{N-k+1}-r}\Gamma\big(T_{r+s_{j+1}}a,T_{r+s_{j+1}}a\big)\mbox{d}r.
\end{align*}
Since $\Gamma^2\geq 0$, the above expression does not exceed
\begin{align*}
&\pi_{t_{N-k+1}}T_{t_{N-k+1}}\int_{t_{N-k}}^{t_{N-k+1}} 2\Gamma\big(T_{s_{j+1}}a,T_{s_{j+1}}a\big)\mbox{d}r\\
&=2(t_{N-k+1}-t_{N-k})\mathcal{E}_{[t_{N-k+1}}\pi_0 \left(\Gamma\big(T_{s_{j+1}}a,T_{s_{j+1}}a\big)\right).
\end{align*}
Combining the above observations with \eqref{chainsem} yields
\begin{align*}
&\tau((R_{n-1}-R_n)dy_kR_{n-1}dy_k(R_{n-1}-R_n))\\
&\qquad \qquad \leq 2\tau\bigg((R_{n-1}-R_n)(t_{N-k+1}-t_{N-k})\pi_0\left(\Gamma\big(T_{s_{j+1}}a,T_{s_{j+1}}a\big)\right)\bigg)
\end{align*}
and therefore
\begin{align*}
&\sum_{k=n+1}^N \tau((R_{n-1}-R_n)dy_kR_{n-1}dy_k(R_{n-1}-R_n))\\
&\qquad \qquad \leq \sum_{j=0}^{M-1} \tau\bigg((R_{n-1}-R_n)(s_{j+1}-s_j)\pi_0\left(2\Gamma\big(T_{s_{j+1}}a,T_{s_{j+1}}a\big)\right)\bigg).
\end{align*}
Summing over all $n$ we see that the good-$\lambda$ testing condition (i) is satisfied.
\end{proof}

\begin{proof}[Proof of the inequality \eqref{mainsemigr}]
The application of good-$\lambda$ approach (i.e., the inequality \eqref{main_moment}) gives
\begin{align*}
 &\left\|\sum_{k=0}^{M-1}\bigg(\pi_{s_k}\big(T_{s_k+s_{k+1}}a\big)-\pi_{s_{k+1}}\big(T_{2s_{k+1}}a\big)\bigg)\right\|_{L^p(\mathcal{M})}\\
&\leq C_p\left(\left\|\pi_0\left(\sum_{j=0}^{M-1}2(s_{j+1}-s_j)\Gamma(T_{s_{j+1}}a,T_{s_{j+1}}a)\right)^{1/2}\right\|_{L^p(\mathcal{M})}^2\right.\\
&\qquad \qquad \qquad \qquad \qquad \qquad \qquad \qquad \qquad \left.+\left(\sum_{k=0}^N \|dy_k\|_{L^p(\mathcal{M})}^p\right)^{2/p}\right)^{1/2}.
\end{align*}
Now we go to the limit with the partition $\{t_n\}$. The $p$-th variation of $u$ vanishes, so the sum $ \left(\sum_{k=0}^N \|dy_k\|_{L^p(\mathcal{M})}^p\right)^{2/p}$ converges to $0$. This establishes the bound
\begin{align*}
 &\left\|\sum_{k=0}^{M-1}\bigg(\pi_{s_k}\big(T_{s_k+s_{k+1}}a\big)-\pi_{s_{k+1}}\big(T_{2s_{k+1}}a\big)\bigg)\right\|_{L^p(\mathcal{M})}\\
&\qquad \qquad \qquad \leq C_p\left\|\left(\sum_{j=0}^{M-1}2(s_{j+1}-s_j)\Gamma(T_{s_{j+1}}a,T_{s_{j+1}}a)\right)^{1/2}\right\|_{L^p(\mathcal{N})}.
\end{align*}
If we pass to the limit with the partition $(s_j)_{j=0}^M$, we see that the right-hand side converges to $ C_p\left\|\left(\int_0^S 2\Gamma(T_{r}a,T_{r}a)\mbox{d}r\right)^{1/2}\right\|_{L^p(\mathcal{N})}\leq C_p\left\|\left(\int_0^\infty 2\Gamma(T_{r}a,T_{r}a)\mbox{d}r\right)^{1/2}\right\|_{L^p(\mathcal{N})}.$
To relate the left-hand side to $\|a\|_{L^p(\mathcal N)}$, we pick an operator $v\in \mathcal{A}$ satisfying $\|v\|_{L^{p/(p-1)}}\leq 1$ and $\|a\|_{L^p(\mathcal N)}\leq 2\tau(v^*a)$, and compute that
\begin{align*}
&\left\|\sum_{k=0}^{M-1}\bigg(\pi_{s_k}\big(T_{s_k+s_{k+1}}a\big)-\pi_{s_{k+1}}\big(T_{2s_{k+1}}a\big)\bigg)\right\|_{L^p(\mathcal{M})}\\
&\geq \tau\left[\pi_0(v)^*\sum_{j=0}^{M-1} \bigg(\pi_{s_j}(T_{s_j+s_{j+1}}a)-\pi_{s_{j+1}}(T_{2s_{j+1}}a)\bigg)\right]\\
&=\tau\left[\sum_{j=0}^{M-1} \bigg(\pi_{s_j}(T_{s_j}v^*)-\pi_{s_{j+1}}(T_{s_{j+1}}v^*)\bigg)\bigg(\pi_{s_j}(T_{s_j+s_{j+1}}a)-\pi_{s_{j+1}}(T_{2s_{j+1}}a)\bigg)\right]\\
&=2\sum_{j=0}^{M-1}\int_{s_j}^{s_{j+1}} \tau\bigg(\pi_r\Gamma(T_rv^*,T_{r+s_{j+1}}a)\bigg)\mbox{d}r\\
&=2\sum_{j=0}^{M-1}\int_{s_j}^{s_{j+1}} \tau\bigg(\Gamma(T_rv^*,T_{r+s_{j+1}}a)\bigg)\mbox{d}r.
\end{align*}
Passing to the limit with the partition $(s_j)_{j=0}^M$ and then letting $S\to \infty$, we see that the latter expression converges to
$$ 2\int_0^\infty \tau(\Gamma(T_rv^*,T_{2r}a)\big)\mbox{d}r=2\int_0^\infty \tau(vAT_{3r}a)\mbox{d}r=\frac{2}{3}\tau(va),$$
where the last passage follows easily from the spectral resolution for $A$. Putting all the above facts together, we obtain the desired bound
\begin{align*}
\|a\|_p\leq 3C_p\left\|\left(\int_0^\infty 2\Gamma(T_ra,T_ra)\mbox{d}r\right)^{1/2}\right\|_{L^p(\mathcal{N})}. \qquad \qquad \qedhere
\end{align*}
\end{proof}

\noindent\textbf{Part II. On the inequality \eqref{mainsemigr22}.} The reasoning will be similar to that above. Suppose that $p>2$, fix $S>0$ and let $0=t_0<t_1<t_2<\ldots<t_M=S$ be an arbitrary partition of $[0,S]$. By the assumption on the Markov dilation, the process $u=(u_s)_{s\in [0,S]}$ given by $u_s=\pi_s(T_sa)$ is martingale adapted to $(\mathcal{M}_{[s})_{s\in [0,S]}$, whose $p$-th variation on $[0,S]$ is zero.
Let  $\e_0$, $\e_1$, $\ldots$, $\e_n$ be the sequence of independent Rademacher variables on some given probability space $(\Omega,\F,\mathbb{P})$. We embed $\mathcal{M}$ into a larger von Neumann algebra $\widetilde{\mathcal{M}}=\mathbb{M}_{N+1}\overline{\otimes} L^\infty(\Omega,\F,\mathbb{P})\overline{\otimes} \mathcal{M}$ equipped with the usual tensor trace $\tilde{\tau}$ and the filtration $(\mathbb{M}_{N+1}\overline{\otimes} L^\infty(\Omega,\F_n,\mathbb{P})\overline{\otimes} \mathcal{M}_n)_{n=0}^N$. Here, for each $n$, $\F_n$ is the $\sigma$-algebra generated by $\e_0$, $\e_1$, $\ldots$, $\e_n$ and $\mathcal{M}_n$ is the von Neumann algebra generated by the operators $\pi_t(v)$ for $v\in\mathcal{N}$ and $t\geq t_{N-n}$.
We may treat $\widetilde{\mathcal{M}}$ as the algebra of $(N+1)\times (N+1)$ matrices with entries being random operators. Consider the adapted discrete-time martingales $x=(x_n)_{n=0}^N$ and $y=(y_n)_{n=0}^N$ in $\widetilde{\mathcal{M}}$ determined by the equalities $x_0=e_{1,1}\otimes 1\otimes \pi_{t_N}(T_{t_N}a)$, $y_0=0$ and, for $n=1,\,2,\,\ldots,\,N-1$,
$$ dx_n=(e_{1,1}+e_{n+1,n+1})\otimes 1\otimes \bigg(\pi_{t_{N-n}}(T_{t_{N-n}}a)-\pi_{t_{N-n+1}}(T_{t_{N-n+1}}a)\bigg)$$
and
$$dy_n=\big(e_{1,n+1}+e_{n+1,1}\big)\otimes \e_n \otimes \left(\int_{t_{N-n}}^{t_{N-n+1}} 2\pi_r\Gamma(T_{r}a,T_{r}a)\mbox{d}r\right)^{1/2}.$$
 Finally, set $ z_N=\left(\sum_{n=0}^N |dy_n|^p\right)^{1/p}.$

\begin{lemma}
The triple $(x_N,y,z_N)$ satisfies the good-$\lambda$ testing conditions.
\end{lemma}
\begin{proof}
We have already seen in the previous sections that the testing condition (ii) is satisfied. To check (i), we argue as in the proof of Lemma \ref{gc1} to obtain
\begin{align*}
&\mathcal{E}_{k-1}^{\widetilde{\mathcal{M}}}(dx_k^2)\\
&=(e_{1,1}+e_{k+1,k+1})\otimes 1\otimes \mathcal{E}_{[t_{N-k+1}}\bigg(\pi_{t_{N-k}}(T_{t_{N-k} }a)-\pi_{t_{N-k+1}}(T_{t_{N-k+1} }a)\bigg)^2\\
&=(e_{1,1}+e_{k+1,k+1})\otimes 1\otimes \pi_{t_{N-k+1}}\int_{t_{N-k}}^{t_{N-k+1}} 2T_{t_{N-k+1}-r}\Gamma\big(T_{r}a,T_{r}a\big)\mbox{d}r\\
&=(e_{1,1}+e_{k+1,k+1})\otimes 1\otimes \mathcal{E}_{[t_{N-k+1}}\int_{t_{N-k}}^{t_{N-k+1}} 2\pi_r\Gamma\big(T_{r}a,T_{r}a\big)\mbox{d}r\\
&=\mathcal{E}_{k-1}^{\widetilde{\mathcal{M}}}(dy_k^2)
\end{align*}
for $1\leq k\leq N$.
Consequently, we may write
\begin{align*}
\sum_{n=0}^N \sum_{k=n+1}^N \tilde{\tau}\bigg((R_{n-1}-R_n)dy_kR_{n-1}dy_k\bigg)&\leq \sum_{n=0}^N \sum_{k=n+1}^N \tilde{\tau}\bigg((R_{n-1}-R_n)dy_k^2\bigg)\\
&= \sum_{n=0}^N \sum_{k=n+1}^N \tilde{\tau}\bigg((R_{n-1}-R_n)dx_k^2\bigg)\\
&= \sum_{n=0}^N \tilde{\tau}\bigg((R_{n-1}-R_n)(x_N^2-x_n^2)\bigg)\\
&\leq \tilde{\tau}((I-R_N)x_N^2)
\end{align*}
and the claim follows.
\end{proof}

\begin{proof}[Proof of the estimate \eqref{mainsemigr22}]
The application of good-$\lambda$ approach yields the inequality
\begin{equation}\label{m&m9}
 \|y_N\|_{L^p(\widetilde{\mathcal{M}})}\leq C_p\left(\|x_N\|_{L^p(\widetilde{\mathcal{M}})}^2+\|z_N\|_{L^p(\widetilde{\mathcal{M}})}^2\right)^{1/2}.
\end{equation}
Let us first deal with the expression on the left-hand side. We easily check that
$$ y_N^2\geq e_{1,1}\otimes 1\otimes \int_0^S 2\pi_r \Gamma(T_{r }a,T_{r }a)\mbox{d}r $$
and therefore, applying \eqref{Mdil2} and the condition $\Gamma^2\geq 0$,
\begin{align*}
 \mathcal{E}_{0]}\int_0^S 2\pi_r \Gamma(T_{r }a,T_{r }a)\mbox{d}r&=\pi_0\int_0^S 2T_r \Gamma(T_{r }a,T_{r }a)\mbox{d}r\\
 &\geq \pi_0\int_0^S 2 \Gamma(T_{2r}a,T_{2r}a)\mbox{d}r=\pi_0\int_0^{2S}\Gamma(T_ra,T_ra)\mbox{d}r.
\end{align*}
Since the conditional expectation is a contraction on $L^{p/2}$, we get
\begin{equation}\label{innnr}
 \left\|\int_0^{2S} \Gamma(T_ra,T_ra)\mbox{d}r\right\|_{L^{p/2}(\mathcal{N})}\leq \|y_N^2\|_{L^{p/2}(\widetilde{\mathcal{M}})}=\|y_N\|_{L^{p}(\widetilde{\mathcal{M}})}^2.
\end{equation}
To handle the right-hand side of \eqref{m&m9}, we pass to the limit with the partition $\{t_j\}$. We get
\begin{align*}
\|x_N\|_{L^p(\widetilde{\mathcal{M}})}^p&=\left\|e_{1,1}\otimes 1\otimes \pi_{t_0}(T_0a)\right\|_{L^p(\widetilde{\mathcal{M})}}^p\\
&\quad +\sum_{n=1}^{N}\left\|e_{n+1,n+1}\otimes 1\otimes \bigg(\pi_{t_{N-n}}(T_{t_{N-n}}a)-\pi_{t_{N-n+1}}(T_{t_{N-n+1}}a)\bigg)\right\|_{L^p(\widetilde{\mathcal{M}})}^p\\
&\to \left\|e_{1,1}\otimes 1\otimes \pi_0(a)\right\|_{L^p(\widetilde{\mathcal{M})}}^p=\|a\|_{L^p(\mathcal{N})}^p.
\end{align*}
Furthermore, we have
\begin{align*}
 \|dy_n\|_{L^p(\widetilde{\mathcal{M}})}&=2^{1/p}\left\|\int_{t_{N-n}}^{t_{N-n+1}} 2\pi_r\Gamma(T_{r}a,T_{r}a)\mbox{d}r\right\|_{L^{p/2}(\mathcal{M})}^{1/2}\\
&\leq 2^{1/p}(t_{N-n+1}-t_{N-n})^{1/2}\sup_{r\in [0,S]}\|\Gamma(T_ra,T_ra)\|_{L^{p/2}(\mathcal{N})}^{1/2},
\end{align*}
which immediately gives $\|z_N\|_{L^p(\widetilde{\mathcal{M}})}\to 0$.
Combining the above observations with \eqref{innnr} and \eqref{m&m9} gives
$$  \left\|\left(\int_0^{2S} \Gamma(T_ra,T_ra)\mbox{d}r\right)^{1/2}\right\|_{L^{p}(\mathcal{N})}\leq C_p\|a\|_{L^p(\mathcal{N})},$$
which completes the proof.
\end{proof}

\section*{Acknowledgment} {The authors thank Dmitriy Zanin and Dejian Zhou for kindly pointing out the simpler conditions \eqref{strong test condition} and the careful reading of the  paper.}


\begin{thebibliography}{99}
\bibitem{A}{R. Adamczak, {\it Moment inequalities for U-statistics}, Ann. Probab. \textbf{34} (2006), 2288--2314.}
\bibitem{AFM}{H. Aimar, L. Forzani and F.J. Mart\'{i}n-Reyes, {\it On weighted inequalities for one-sided singular
integrals}, Proc. Amer. Math. Soc. {\bf125} (1997), 2057--2064.}
\bibitem{AP}{A. B. Aleksandrov and V. V. Peller, {\it  Hankel and Toeplitz-Schur multipliers}, Math. Ann. \textbf{324} (2002), no. 2, 277--327.}
\bibitem{BK}{R. J. Bagby and D. S. Kurtz, {\it A rearranged good-$\lambda$ inequality}, Trans. Amer. Math. Soc. \textbf{293} (1986), no. 1, 71--81.}
\bibitem{Ba}{R. Ba\~nuelos, {\it A sharp good-$\lambda$ inequality with an application to Riesz transforms}, Michigan Math. J. {\bf 35} (1988), no. 1, 117--125.}
\bibitem{BM}{R. Ba\~nuelos and C. N. Moore, {\it Laws of the iterated logarithm, sharp good-$\lambda$ inequalities and $L^p$-estimates for caloric and harmonic functions}, Indiana Univ. Math. J. {\bf 38} (1989), no. 2, 315--344.}
\bibitem{BW}{R. Ba\~nuelos and G. Wang, {\it Sharp inequalities for martingales with applications to the Beurling-Ahlfors and Riesz transformations}, Duke Math. J. \textbf{80} (1995), 575--600.}
\bibitem{BLP}{S. Barza, V. Lie, and N. Popa, {\it Approximation of infinite matrices by matricial Haar polynomials}, Ark. Mat. \textbf{43} (2005), no. 2, 251--269.}
\bibitem{BC12}{T. N. Bekjan and Z. Chen, {\it Interpolation and $\Phi$-moment inequalities of noncommutative martingales}, Probab. Theory Related Fields. \textbf{152} (2012), no. 1-2, 179--206.}
\bibitem{BCL}{T. N. Bekjan, Z. Chen, P. Liu, Y. Jiao, {\it Noncommutative weak Orlicz spaces and martingale inequalities}, Studia Math. \textbf{204} (2011), no. 3, 195--212.}
\bibitem{BCO}{T. N. Bekjan, Z. Chen and A. Os\k ekowski, {\it Noncommutative maximal inequalities associated with convex functions}, Trans. Amer. Math. Soc. \textbf{369} (2017), no. 1, 409--427.}
\bibitem{BCPY}{T. N. Bekjan, Z. Chen, M. Perrin and Z. Yin, {\it Atomic decomposition and interpolation for Hardy spaces of noncommutative martingales}, J. Funct. Anal. \textbf{258} (2010), no. 7, 2483--2505.}
\bibitem{Be}{G. Bennett, {\it Schur multipliers}, Duke Math. J. \textbf{44} (1977), 603--639.}
\bibitem{Bo}{J. Bourgain, Vector-valued singular integrals and the H1-BMO duality,
in: Probability Theory and Harmonic Analysis (Mini-conference on Probability Theory and Harmonic Analysis, Cleveland, 1983), J.-A. Chao and W.A. Woyczy\'nski (eds), Monographs and Textbooks in Pure and Appl. Math. 98, Marcel Dekker, New York, 1986, 1--19.}
\bibitem{Buc}{S. M. Buckley, {\it Estimates for operator norms on weighted spaces and reverse Jensen inequalities}, Trans. Amer. Math. Soc. \textbf{340} (1993), 253--272.}
\bibitem{Bu0}{D. L. Burkholder, {\it Distribution function inequalities for martingales}, Ann. Probab. \textbf{1} (1973), 19--42.}
\bibitem{Bu0.5}{D. L. Burkholder, {\it One-sided maximal functions and $H^p$}, J. Funct. Anal. \textbf{18} (1975), 429--454.}
\bibitem{Bu1}{D. L. Burkholder, {\it Boundary value estimation of the range of an analytic function,} Michigan Math. J. \textbf{25} (1978), no. 2, 197--211.}
\bibitem{B0}{D. L. Burkholder,{\it  Boundary value problems and sharp inequalities for martingale transforms}, Ann. Probab. 12 (1984), 647--702.}
\bibitem{Bhar}{{D. L. Burkholder}, {\it Differential subordination of harmonic functions and martingales}, Harmonic Analysis and Partial Differential Equations (El Escorial, 1987), Lecture Notes in Mathematics 1384 (1989), 1--23.}
\bibitem{B2}{D. L. Burkholder, {\it Explorations in martingale theory and its applications}, Lecture Notes in Math., 1464, Springer, Berlin, 1991.}
\bibitem{BG}{D. L. Burkholder and R. F. Gundy, {\it Extrapolation and interpolation of quasi-linear operators on martingales}, Acta Math. \textbf{124} (1970), 249--304.}
\bibitem{BG2}{D. L. Burkholder and R. F. Gundy, {\it Distribution function inequalities for the area integral}, Collection of articles honoring the completion by Antoni Zygmund of 50 years of scientific activity, VI. Studia Math. \textbf{44} (1972), 527--544.}
\bibitem{Cla}{A.-D. Claire, {\it On ergodic theorems for free group actions on noncommutative spaces},  Probab. Theory Relat. Fields \textbf{135} (2006), 520--546.}
\bibitem{CF}{R. R. Coifman and C. Fefferman, {\it Weighted norm inequalities for maximal functions and singular integrals}, Studia Math. \textbf{51} (1974), 241--250.}
\bibitem{CP}{J. M. Conde-Alonso and J. Parcet, {\it Atomic blocks for noncommutative martingales}, Indiana Univ. Math. J. \textbf{65} (2016), 1425--1443.}
\bibitem{CV2}{S. G. Cox and M. C. Veraar, {\it Vector-valued decoupling and the Burkholder-Davis-Gundy inequality}, Illinois J. Math. \textbf{55} (2011), 343--375.}
\bibitem{Cu}{I. Cuculescu, {\it Martingales on von Neumann algebras}, J. Multivariate Anal. \textbf{1} (1971), 17--27.}
\bibitem{dP}{V. H. de la Pena, S. J. Montgomery-Smith and J. Szulga, {\it Contraction and decoupling inequalities for multilinear forms and U-statistics}, Ann. Probab. \textbf{22} (1994),  1745--1765.}
\bibitem{DG}{V. H. de la Pe\~na and E. Gin\'e. Decoupling. From dependence to independence, Randomly stopped processes. U-statistics and processes. Martingales and beyond. Probability and its Applications (New York). Springer-Verlag, New York, 1999.}
\bibitem{Do}{I. Doust,  Norms of $0$-$1$ matrices in $C_p$, Geometric analysis and applications
(Canberra, 2000), pp. 50--55, Proc. Centre Math. Appl. Austral. Nat. Univ., 39, Austral. Nat. Univ., Canberra, 2001.}
\bibitem{DGi}{I. Doust and T. A. Gillespie, {\it Schur multiplier projections on the von Neumann-Schatten classes}, J. Operator Theory \textbf{53} (2005), no. 2, 251--272.}
\bibitem{Hi-1}{P. Hitczenko, {\it Comparison of moments for tangent sequences of random variables}, Probab. Theory Rel. Fields \textbf{78} (1988), 223--230.}
\bibitem{Hi0}{P. Hitczenko, {\it Best constant in the decoupling inequality for non-negative random variables}, Statist. Probab. Lett. \textbf{9} (1990), 327--329.}
\bibitem{Hi}{P. Hitczenko, {\it Best constants in martingale version of Rosenthal's inequality}, Ann. Probab. \textbf{18} (1990), no. 4, 1656--1668.}
\bibitem{Hi2}{P. Hitczenko, {\it On a domination of sums of random variables by sums of conditionally independent ones}, Ann. Probab. {\bf 22} (1994), no. 1, 453--468.}
\bibitem{Hofmann} {S. Hofmann, C. Kenig, S. Mayboroda and J. Pipher, {\it Square function/non-tangential maximal function estimates and the Dirichlet problem for non-symmetric elliptic operators}, J. Amer. Math. Soc. \textbf{28} (2015), no. 2, 483--529.}
\bibitem{HJP2}{G. Hong, M. Junge and J. Parcet, {\it Asymmetric Doob inequalities in continuous time}, J. Funct. Anal. {\bf 273} (2017), no. 4, 1479--1503.}
\bibitem{HJP1}{G. Hong, M. Junge and  J. Parcet, {\it Algebraic Davis decomposition and asymmetric Doob inequalities}, Comm. Math. Phys. {\bf 346} (2016), no. 3, 995--1019.}
\bibitem{HM}{G. Hong and T. Mei, {\it John-Nirenberg inequality and atomic decomposition for noncommutative martingales}, J. Funct. Anal. {\bf 263} (2012), no. 4, 1064--1097.}
\bibitem{JN}{S. Janson and K. Nowicki, {\it The asymptotic distributions of generalized U-statistics with applications to random graphs}, Probab. Theory Related Fields \textbf{90} (1991), 341--375.}
\bibitem{Ji}{Y. Jiao, {\it Burkholder's inequalities in noncommutative Lorentz spaces}, Proc. Amer. Math. Soc. \textbf{138} (2010), 2431--2441.}
\bibitem{JOW}{Y. Jiao, A. Os\k{e}kowski, L. Wu, {Inequalities for noncommutative differentially subordinate martingales}, Adv. Math. \textbf{337} (2018), 216-259.}
\bibitem{JSXZ}{Y. Jiao, F. Sukochev, G. Xie and D. Zanin, {\it $\Phi$-moment inequalities for independent and freely independent random variables.},  J. Funct. Anal.  (270) 94 (2016),  4558--4596.}
\bibitem{JSZ}{Y. Jiao, F. Sukochev and D. Zanin, {\it Johnson-Schechtman and Khintchine inequalities in noncommutative probability theory}, J. Lond. Math. Soc. (2) \textbf{94} (2016),  113--140.}
\bibitem{JSZZ}{Y. Jiao, F. Sukochev, D. Zanin and D. Zhou, {\it Johnson-Schechtman inequalities for noncommutative martingales}, J. Funct. Anal. \textbf{272} (2017), 976--1016.}
\bibitem{JZWZ}{Y. Jiao, D. Zhou, L. Wu and D. Zanin, {\it Noncommutative dyadic martingales and Walsh--Fourier series},  J. Lond. Math. Soc. (2) \textbf{97} (2018), no. 3, 550--574.}
\bibitem{J}{M. Junge, {\it Doob's inequality for non-commutative martingales}, J. Reine Angew. Math. \textbf{549} (2002), 149--190.}
\bibitem{JM}{M. Junge and T. Mei, {\it Noncommutative Riesz transforms - a probabilistic approach}, Amer. J. Math. \textbf{132} (2010), 611--680.}
\bibitem{JM2}{M. Junge and T. Mei, {\it BMO spaces associated with semigroups of operators}, Math. Ann. \textbf{352} (2012), 691--743.}
\bibitem{JMP}{M. Junge, T. Mei and J. Parcet, {\it Noncommutative Riesz transforms - dimension free bounds and Fourier multipliers}, J. Eur. Math. Soc.  \textbf{20} (2018), no. 3, 529--595.}
\bibitem{JX-M}{M. Junge and Q. Xu, {\it Noncommutative maximal ergodic theorems}, J. Amer. Math. Soc. {\bf20} (2007), no. 2, 385--439.}
\bibitem{JX}{M. Junge and Q. Xu, {\it Noncommutative Burkholder/Rosenthal inequalities}, Ann. Probab. \textbf{31} (2003), 948--995.}
\bibitem{JX2}{M. Junge and Q. Xu, {\it On the best constants in some non-commutative martingale inequalities}, Bull. London Math. Soc. \textbf{37} (2005), 243--253.}
\bibitem{JX3}{M. Junge and Q. Xu, {\it Noncommutative Burkholder/Rosenthal inequalities. II. Applications}, Israel J. Math. \textbf{167} (2008), 227--282.}
\bibitem{KR1}{K. V. Kadison and J. R. Ringrose, Fundamentals of the Theory of Operator
Algebras. Vol. I, Elementary Theory, Academic Press, New York, 1983.}
\bibitem{KR2}{R. V. Kadison and J. R. Ringrose, Fundamentals of the theory of operator
algebras. Vol. II, Advanced Theory, Academic Press, Orlando, FL, 1986.}
\bibitem{KP}{S. Kwapie\'n and A. Pe\l czy\'nski, {\it The main triangle projection in matrix spaces and its applications}, Studia Math. \textbf{34} (1970), 43--67.}
\bibitem{KW}{S. Kwapie\'n and W. A. Woyczy\'nski. Tangent sequences of random variables: basic inequalities and their applications. In Almost everywhere convergence (Columbus, OH, 1988),  237--265. Academic Press, Boston, MA, 1989.}
\bibitem{KW2}{S. Kwapie\'n and W. A. Woyczy\'nski. Random series and stochastic integrals: single and
multiple. Probability and its Applications. Birkh\"auser Boston Inc., Boston, MA, 1992.}
\bibitem{M}{J. Maas, {\it Malliavin calculus and decoupling inequalities in Banach spaces}, J. Math. Anal. Appl. \textbf{363} (2010), 383--398.}
\bibitem{MT}{T. McConnell and M. S. Taqqu, {\it Decoupling Inequalities for Multilinear Forms in Independent Symmetric Random Variables}, Ann. Probab. \textbf{14} (1986), 943--954.}
\bibitem{Mey}{P. A. Meyer, {D\'emonstration probabiliste de certaines in\'egalit\'es de Littlewood-Paley. I. Les in\'egalit\'es
classiques.} In: S\'eminaire de Probabilit\'es,  Lecture Notes in Math., vol. 511, 125--141. Springer, Berlin (1976).}
\bibitem{MW}{B. Muckenhoupt and R. Wheeden, {\it Weighted norm inequalities for fractional integrals}, Trans. Amer. Math. Soc. \textbf{192} (1974), 261--274.}
\bibitem{NVW}{J. M. A. M. van Neerven, M. C. Veraar, and L. Weis, {\it Stochastic integration in UMD Banach spaces}, Ann. Probab. \textbf{35} (2007), 1438--1478.}
\bibitem{NW1}{J. M. A. M. van Neerven and L. Weis, {\it Stochastic integration of functions with values in a
Banach space}, Studia Math. \textbf{166} (2005), 131--170.}
\bibitem{NW2}{J. M. A. M. van Neerven and L. Weis, {\it Stochastic integration of operator-valued functions with respect to Banach space-valued Brownian motion}, Potential Anal. \textbf{29} (2008),  65--88.}
\bibitem{O0}{A. Os\k ekowski, {\it Inequalities for dominated
martingales}, Bernoulli 13 no. 1 (2007), 54--79.}
\bibitem{Os2}{A. Os\k{e}kowski, {Sharp martingale and semimartingale inequalities}, Monografie Matematyczne \textbf{72} (2012), Birkh\"auser, 462 pp.}
\bibitem{OP}{N. Ozawa and S. Popa, {\it On a class of $II_1$ factors with at most one cartan subalgebra}, Ann. Math. \textbf{172} (2010), 713--749.}
\bibitem{PR}{J. Parcet, N. Randrianantoanina, {\it Gundy's decomposition for non-commutative martingales and applications}, Proc. London Math. Soc. \textbf{93} (2006), 227--252.}
\bibitem{Pe}{M. Perrin, {\it A noncommutative Davis' decomposition for martingales}, J. Lond. Math. Soc. (2) \textbf{80} (2009), no. 3, 627--648.}
\bibitem{Pet}{J. Peterson, {\it $L^2$-rigidity in von Neumann algebras}, Inventiones Math. \textbf{175} (2009), 417--433.}
\bibitem{Pi}{S. K. Pichorides, {\it On the best values of the constants in the theorems of M. Riesz, Zygmund and Kolmogorov}, Studia Math. \textbf{44} (1972), 165--179.}\bibitem{P}{G. Pisier, Similarity Problems and Completely Bounded Maps, vol. 1618 of Lecture Notes in Mathematics, Springer, Berlin, Germany, 1995.}
\bibitem{PX}{G. Pisier and Q. Xu, {\it Non-commutative martingale inequalities}, Comm. Math. Physics \textbf{189} (1997), 667--698.}
\bibitem{px2003}{G. Pisier and Q. Xu, {\it Non-commutative $L^p$-spaces}, in Hand book of the geometry of Banach spaces, Vol. 2, 1459--1517, North-Holland, Amsterdam, 2003.}
\bibitem{Po}{S. Popa, {\it On a class of type II1 factors with Betti numbers invariants}, Ann. of Math. \textbf{163} (2006), 809--899.}
\bibitem{R1}{N. Randrianantoanina, {\it Non-commutative martingale transforms}, J. Funct. Anal. \textbf{194} (2002), 181--212.}
\bibitem{R2}{N. Randrianantoanina, {\it Square function inequalities for non-commutative martingales}, Israel J. Math. \textbf{140} (2004), 333--365.}
\bibitem{R3}{N. Randrianantoanina, {\it Conditioned square functions for noncommutative martingales}, Ann. Probab. \textbf{35} (2007), 1039--1070.}
\bibitem{RW}{N. Randrianantoanina and L. Wu, {\it Martingale inequalities in noncommutative symmetric spaces}, J. Funct. Anal. \textbf{269} (2015), 2222--2253.}
\bibitem{RW17}{N. Randrianantoanina and L. Wu, {\it Noncommutative Burkholder/Rosenthal  inequalities associated with convex functions}, Ann. Inst. H.
Poincar$\acute{e}$ Probab. Statist. \textbf{53} (2017), 1575--1605.}
\bibitem{RWX}{N. Randrianantoanina, L. Wu and Q. Xu, {\it Noncommutative Davis type decompositions and applications}, J. Lond. Math. Soc. (to appear), 2018.}
\bibitem{RY}{D. Revuz and M. Yor, {\it Continuous martingales and Brownian motion}, 3rd edition, Springer Verlag, 1999.}
\bibitem{Ric}{\'E. Ricard, {\it A Markov dilation for self-adjoint Schur multipliers}, Proc. Amer. Math. Soc. \textbf{136} (2008), 4365--4372.}
\bibitem{St} E. M. Stein, {Singular integrals and
Differentiability Properties of Functions,} Princeton University
Press, Princeton, 1970.
\bibitem{T}{M. Takesaki, Theory of Operator Algebras. I, Springer-Verlag, New York, 1979.}
\bibitem{Va}{N. ~ Th. Varopoulos, L. Saloff-Coste, and T. Coulhon, Analysis and Geometry on Groups, Cambridge
Tracts in Math., vol. 100, Cambridge University Press, Cambridge, 1992.}
\bibitem{Wan}{G. Wang, {\it Sharp inequalities for the conditional square function of a martingale}, Ann. Probab. \textbf{19} (1991), 1679--1688.}
\bibitem{W}
{G. Wang,  {\it Differential subordination and strong
differential subordination for continuous-time martingales and
related sharp inequalities}, Ann. Probab. 23 (1995), no. 2,
522--551.}
\bibitem{Z}{J. Zinn, {\it Comparison of martingale difference sequences}, Lect. Notes Math., vol. 1153, 453--457, 1985.}
\end{thebibliography}
\end{document}